\documentclass[a4paper,11pt]{article} 
\usepackage[utf8]{inputenc}
\usepackage[english]{babel}

\usepackage{amsmath, amsthm, amssymb}
\usepackage{comment}

\usepackage[mathscr,mathcal]{eucal} 
\usepackage{dsfont}
\usepackage{enumerate} 
\usepackage{hyperref}

\usepackage{tikz}

\newtheorem{theorem}{Theorem}[section]
\newtheorem  {lemma}[theorem]{Lemma}         
\newtheorem  {corollary}[theorem]     {Corollary}
\newtheorem  {prop}[theorem]   {Proposition}

\newtheorem  {defi}{Definition}[section]
\newtheorem {notation}[defi]    {Notation}
\newtheorem* {notation*}    {Notation}

\newtheorem* {theorem*}      {Theorem}
\newtheorem* {lemma*}        {Lemma}
\newtheorem* {corollary*}    {Corollary}
\newtheorem* {prop*}  {Proposition}
\newtheorem* {defi*}   {Definition}
\newtheorem* {remark*}       {Remark}
\newtheorem  {remark}[defi]        {Remark}

\newtheorem* {claim*}        {Claim}

\newtheorem* {assum*}        {Assumption}
\newtheorem  {assum}[defi]        {Assumption}

\def \eps{\epsilon}

\def \P {\mathbb P}
\def \E {\mathbb E}
\def \var {\text{Var}}
\def \cov {\text{Cov}}
\def \unif {\mathrm{Unif}}

\def \Lvert {\middle|\:}
\def \given {\;|\;}
\def \Given {\;\Lvert\;}

\def \wst {\mathsf{WST}}
\def \mst {\mathsf{MST}}
\def \maxst {\mathsf{MaxST}}
\def \ust {\mathsf{UST}}
\def \Reff {\mathcal{R}_{\mathrm{eff}}}

\def \path {\mathrm{Path}}

\def \c {\boldsymbol{\mathit{c}}}
\def \tc{{\mathrm{(tc)}}}
\def \tn{{\mathrm{(tn)}}}
\def \C{{\mathcal{C}^{T}}}
\def \CN{{\mathcal{CN}^{T}_n}}
\def \A{\mathcal{A}^T}
\def \tW{\widetilde{W}}
\def \tE{E}

\author{\'Agnes K\'usz \footnote{Department of Stochastics, Institute of Mathematics, Budapest University of Technology and Economics, M\H{u}egyetem rkp.~3., Budapest 1111 Hungary. HUN-REN Alfr\'ed R\'enyi Institute of Mathematics, Re\'altanoda u. 13-15, Budapest 1053 Hungary. kusza [at] renyi.hu, kuszagnes1205 [at] gmail.com}}
\title{The local limit of weighted spanning trees on balanced networks}
\date{}

\begin{document}
\maketitle
\begin{abstract}
We prove that the local limit of the weighted spanning trees on any simple connected high degree almost regular sequence of electric networks is the Poisson(1) branching process conditioned to survive forever, by generalizing \cite{nachmias2022local} and closing a gap in their proof. We also study the local statistics of the $\wst$s on high degree almost balanced sequences, which is interesting even for the uniform spanning trees.

Our motivation comes from studying an interpolation $\{\wst^{\beta}(G)\}_{\beta\in [0, \infty)}$ between $\ust(G)$ and $\mst(G)$ by $\wst$s on a one-parameter family of random environments. This model has recently been introduced in \cite{makowiec2024diameter, kusz2024diameter}, and the phases of several properties have been determined on the complete graphs. 

We show a phase transition of $\wst^{\beta_n}(G_n)$ regarding the local limit and the expected edge overlaps for high degree almost balanced graph sequences $G_n$, without any structural assumptions on the graphs. Our general framework results in a better understanding even in the case of complete graphs, where it narrows the window of the phase transition of \cite{makowiec2024local}. We also study the phase transition regarding the expected total length of our model on the complete graph.
	\end{abstract}

\tableofcontents
\section{Introduction and results}

\subsection{Local limit of \texorpdfstring{$\wst$}{}'s on deterministic electric networks}\label{subsec:intro_wst_on_det}
Maybe surprisingly at first sight, the local limit of the uniform spanning trees on any almost regular graph sequence with degrees tending to infinity is the same as the one for the complete graph \cite{nachmias2022local}. Moreover, in the linear degree case, the local limit of the $\ust$s of any dense convergent graph sequence is determined by the limiting graphon \cite{hladky2018local}. In this paper, we study how far we can generalize these results: how general the graph sequence can be, and what happens when we randomize the edge weights. 

Our motivation comes from the study of the weighted spanning trees on a one-parameter family of random environments on the complete graph interpolating between the uniform and the minimum spanning trees: it is conjectured that the typical diameter of the model has a large intermediate phase, while the local limit has a much sharper phase transition between the $\ust$-like and the $\mst$-like behavior \cite{kusz2024diameter, makowiec2024diameter, makowiec2024local}. The robustness of the local limit we are showing in this paper is in contrast with the fact that the study of the typical diameter relies heavily on that we are on the complete graph.

A \emph{spanning tree} of a connected graph $G$ is an acyclic connected subgraph of $G$ containing all vertices of $G$. A graph $G$ equipped with nonnegative edge weights $\c:=\{c(e)\}_{e\in E(G_n)}$ is called an \emph{electric network}. In this paper, we write $c(u,v)=c((u,v))$ for any $(u,v)\in E(G)$ and $C_v:=\sum_{u:\; (u, v)\in E(G)}c(u, v)$. An electric network is called \emph{irreducible} or \emph{connected} if, between any pairs of vertices, there exists a path formed by edges with positive conductance.

For any connected electric network $(G, \c)$, the weighted spanning tree model is defined as the following: 
\begin{defi}\label{defi:wst}
The \emph{weighted spanning tree $\wst(\c)$} assigns to any spanning tree $T$ of $G$ the probability
$$\P(\wst(\c)=T):= \frac{\prod_{e\in T}c(e)}{Z(\c)} \text{ with }Z(\c):=\sum_{\substack{T'\text{ is a}\\\text{ spanning tree of } G}}\prod_{e\in T'}c(e).$$    
\end{defi}
This model can be generated easily using random walks on electric networks \cite[Subsection 4.1]{PTN}; the two famous algorithms are Wilson's \cite{wilson1996generating} and Aldous-Broder algorithms \cite{aldous1990random, broder1989generating}, making the model accessible for study.

For $(G, \mathbf{1}_{E(G)})$, this is the classical uniform spanning tree model $\ust(G)$. Both the global properties, e.g., the diameter \cite{szekeres1983distribution, aldous, peres2004scaling, chung2012diameter, michaeli2021diameter, alon2022diameter} and the scaling limit \cite{aldous1991continuumI, aldous1991continuumII, aldous1993continuum, archer2024ghpHigh, archer2024ghpDense}, and the local properties \cite{grimmett1980random, nachmias2022local, hladky2018local} of $\ust(G)$ on several graphs $G$ are understood.

For our first theorem, we consider \emph{high degree almost regular} sequences of electric networks $\{(G_n, \c_n)\}_{n\ge 1}$: 
\begin{defi}\label{def:almost_reg}
We call a sequence of electric networks $\{(G_n, \c_n)\}_{n\ge 1}$ \emph{high degree almost regular} with parameter $\gamma_n$ if, as $n\rightarrow \infty$,
\begin{itemize}
    \item[$\circ$] at least $(1-o(1))|V(G_n)|$ of the vertices $v$ of $G_n$ have conductances $C_v=(1\pm o(1))\gamma_n$, 
    \item[$\circ$] the sum of the conductances of the vertices is $\sum_{v\in V(G_n)}C_v=(1\pm o(1))|V(G_n)|\gamma_n$, and
    \item[$\circ$] $\lim_{n\rightarrow \infty}\sup_{(u, v)\in E(G_n)}c(u,v)/\gamma_n=0$.
\end{itemize}
\end{defi}

This definition for $(G_n, \mathbf{1}_{E(G_n)})$ gives back the high degree almost regular sequences of graphs from \cite{nachmias2022local} with parameter $d_n=\gamma_n$: the first two conditions of Definition \ref{def:almost_reg} are word by word the two conditions of almost regularity \cite{nachmias2022local}, and our third condition for $c\equiv 1$ is equivalent to $d_n\rightarrow \infty$ meaning that the sequence of graphs is of high degree.

In this paper, we are interested in the local properties---especially the local (weak) limit or Benjamini--Schramm limit---of the $\wst$s. For the definition of local limits of graph sequences, see Subsection~\ref{subsec:background}.

The local limit of $\ust(K_n)$ for the complete graph $K_n$ is the Poisson(1) Galton-Watson tree conditioned to survive forever \cite{kolchin1977branching, grimmett1980random}. In \cite{nachmias2022local}, it was shown that the local limit of $\ust(G_n)$ for any \emph{high degree almost regular} sequence $\{G_n\}_{n\ge 1}$ of connected, simple graphs is the same as the one for $G_n=K_n$. In \cite{hladky2018local}, they considered graph sequences $G_n$ with a nondegenerate graphon limit, and they showed that the local limit of $\ust(G_n)$ exists and it is a multitype branching tree which can be determined from the limiting graphon of the sequence $\{G_n\}_n$; unlike the regular case, here, the statistics of the degrees is not enough to determine the local limit.

Our first result is a generalization of \cite{nachmias2022local} to weighted spanning trees. We prove this result in Subsection \ref{subsubsection:local-limit_balanced}.
\begin{theorem}\label{thm:local_ust_almost_reg}
We consider a high degree almost regular connected electric network $(G_n, \c_n)$ on a simple graph $G_n$. Then the local limit of $\wst(\c_n)$ is the Poisson(1) Galton-Watson tree conditioned to survive forever.
\end{theorem}

The novelty of this part of our paper is the following:
\begin{itemize}
    \item[$\circ$] We generalize the method of \cite{nachmias2022local} from $\ust$s to $\wst$s, with our primary motivation being the study our weighted spanning tree model $\wst^{\beta_n}(G_n)$ in a random environment defined in Definition~\ref{def:wst_beta}.
    \item[$\circ$] We fill in a gap in the proof of \cite{nachmias2022local}. In \cite{nachmias2022local}, the local limit result was first proven for high degree regular sequences, then they described how to relax regularity to almost regularity in the proof. In Remark \ref{remark:NP_countereg}, we detail that, although the proof of \cite{nachmias2022local} works well for high degree \emph{regular} sequences, the error terms become more delicate for \emph{almost regular} sequences. They have also noticed this gap [Personal communication]. During the correction of this mistake, we are following the outline of the original proof of \cite{nachmias2022local}; however, we need to make some nontrivial, careful modifications of the lemmata and the choices of the definitions of `typical behavior' in many cases.

    For our $\wst^{\beta_n}(G_n)$---see Definition~\ref{def:wst_beta}---purposes, it is essential, even for regular $G_n$s, that the local limit result holds not only for regular, but also for almost regular sequences because of the randomness of our environment.
    
    \item[$\circ$] We generalize the method of \cite{nachmias2022local} to high degree almost \emph{balanced} sequences in Theorem~\ref{thm:local_det} which can be thought as a common generalization of Theorem 1.1 of \cite{nachmias2022local} and Section 3 of \cite{hladky2018local}. Since our formulae are not as explicit as in \cite{nachmias2022local}, we also had to deal with some extra technical difficulties; e.g., we had to show asymptotics instead of upper bounds in several statements. Theorem~\ref{thm:local_det} connects the local limits of uniform spanning trees and the 1-out model, see Remark \ref{rmrk:1-edge_statement}.

    In \cite{hladky2018local}, the structure of the graphs $G_n$ is well-understood: due to Szemerédi’s regularity lemma-like graph partitioning techniques, any dense graph that is close to a nondegenerate graphon can be decomposed into a bounded number of dense expanders such that there are $o(n^2)$ edges between each pair of expanders; the behavior of random walks and effective resistances, and therefore the local limit of the $\ust$s can be obtained from this decomposition. Our Theorem~\ref{thm:local_det} does not use this strong knowledge of the structure of $G_n$, but rather a natural condition of being almost balanced: this is useful to understand the local convergence of $\ust(G_n)$ for almost balanced intermediate convergent graph sequences $\{G_n\}_n$, e.g., in the sense \cite{frenkel2018convergence}, which we are planning in a future paper. 
\end{itemize}

\subsection{\texorpdfstring{$\wst$}{}'s on random environments}
For $\beta\ge 0$, we consider the following spanning tree model on $G$:

\begin{defi}\label{def:wst_beta}
We consider a graph $G$ and $\beta\ge 0$. Given some i.i.d.~$\unif[0,1]$ random labels $\mathbf{U}:=\{U_e\}_{e\in E(G)}$, we define our random weights as $c_{\beta}(e):=\exp(-\beta U_e)$ and our random electric network as having conductances $\c_{G, \beta}:=\{c_{\beta}(e)\}_{e\in E(G)}$. The probability of a spanning tree $T$ of $G$ under the 
 \emph{weighted spanning tree} model $\wst^{\beta}(G)$ is defined as
 \begin{align*}
     \P\left(\wst^{\beta}(G)=T\Given \mathbf{U}\right)=\P\left(\wst(\c_{G, \beta}):=T\Given\mathbf{U}\right).
 \end{align*}
In this paper, we often abbreviate $\c_{G, \beta}$ as $\c_{\beta}$ if it does not cause confusion.
\end{defi}
Note that, for any fixed connected graph $G$, the one-parameter family $\{\wst^{\beta}(G)\}_{\beta\in [0, \infty)}$ interpolates between the two classic spanning tree models: for $\beta=0$, our environment is $\c_{G, \beta}\equiv 1$, so our model is $\ust(G)$, while as $\beta\rightarrow \infty$, this random spanning tree model tends weakly to the unique spanning tree $T$ minimizing $\sum_{e\in T}U_e$ which is usually called the minimum spanning tree of $G$ abbreviated as $\mst(G)$.

The $\ust$s and the $\mst$s are well-studied. For a brief collection of references on the history of $\ust$s, see Subsection \ref{subsec:intro_wst_on_det}. The $\mst$s can be generated by Kruskal's and Prim's algorithms and the properties of the $\mst$s is classically analyzed by studying Bernoulli percolation. The diameter \cite{addario2009critical} and the scaling limit \cite{addario2017scaling} of $\mst(K_n)$ and the diameter of the $\mst$ of the random 3-regular graph \cite{addario2021geometry} are understood. It is shown for the complete graph and for $d_n\rightarrow \infty$-regular graphs $G_n$, that the local limit of $\mst(G_n)$ is the wired minimal spanning forest on the Poisson-weighted infinite tree \cite{addario2013local, nachmias2024wired}. The asymptotics of the expected total length $\E\left[\sum_{e\in \mst(G_n)}U_e\right]$ for $G_n=K_n$ \cite{frieze1985value, aldous1990random, steele2002minimal} and for robust sequences of dense graphs $G_n$ tending to a graphon \cite{hladky2023random} has been determined.

The model $\wst^{\beta}(G)$ has recently been introduced in \cite{makowiec2024diameter, kusz2024diameter} and the behaviors of several properties of $\wst^{\beta_n}(K_n)$ have been understood, including the phase transition around $\beta_n=n^{3+o(1)}$ regarding the agreement of Aldous-Broder algorithm generating $\wst^{\beta_n}(K_n)$ and Prim's invasion algorithm generating $\mst(K_n)$; the phase transition around $\beta_n=n^{2+o(1)}$ regarding the agreement of the models $\mst(K_n)$ and $\wst^{\beta_n}(K_n)$ \cite{kusz2024diameter}; and the phase transition around $\beta_n=n^{1+o(1)}$ regarding the local properties---the local limit, the expected edge overlaps and the expected total length---around $\beta_n=n^{1+o(1)}$ \cite{makowiec2024diameter, kusz2024diameter}. For the collection of results of this model and some more results on the model on Euclidean lattices, see \cite{makowiec2025observables}.

For our model, it is proven \cite{makowiec2024diameter, kusz2024diameter} that the typical diameter of $\wst^{\beta_n}(K_n)$ grows as $\Theta(n^{1/2})$ for $\beta_n\le n^{1+o(1)}$ likewise for $\ust(K_n)$ \cite{michaeli2021diameter}, and it grows as $\Theta(n^{1/3})$ for $\beta_n\ge n^{4/3+o(1)}$ similarly to $\mst(K_n)$ \cite{addario2009critical}. For $\beta_n=n^{\alpha}$ with $1<\alpha<4/3$, it is conjectured that the typical diameter has an exponent strictly between 1/2 and 1/3. For the diameter of a related $\wst$ model on some random environments of bounded degree graphs and on a random two-conductance electric network on the complete graph, see \cite{makowiec2023diameter} and \cite{makowiec2025new}.

In this paper, we study the question whether there is a phase transition of the local properties of $\wst^{\beta_n}(G_n)$ if $G_n$ is a sequence of graphs of growing degrees. It is natural to ask this phase transition question on high degree almost balanced sequences of graphs:
\begin{defi}
We call a sequence of electric networks $\{(G_n, \c_n)\}_{n\ge 1}$ \emph{high degree almost balanced} with parameters $\gamma_n$ and $K>1$ if, as $n\rightarrow \infty$,
\begin{itemize}
    \item[$\circ$] at least $(1-o(1))|V(G_n)|$ of the vertices $v$ of $G_n$ have conductances $C_v\in [\gamma_n, K\gamma_n]$,
    \item[$\circ$] the sum of the conductances $\notin [\gamma_n, K\gamma_n]$ of $G_n$ is $o(1)\gamma_n|V(G_n)|$, and
    \item[$\circ$] $\lim_{n\rightarrow \infty}\sup_{(u, v)\in E(G_n)}c(u,v)/\gamma_n=0$.
\end{itemize}
We call a sequence of graphs $\{G_n\}_{n\ge 1}$ \emph{high degree almost balanced} with parameters $d_n$ and $K>1$ if the electric network $(G_n, \mathbf{1}_{|E(G_n)})$ is almost balanced with parameters $\gamma_n=d_n$ and $K$.
\end{defi}
For the local limit of $\wst^{\beta_n}(G_n)$ on any high degree almost balanced sequence of simple graphs, it turns out that a phase transition happens around $\beta_n=d_n |V(G_n)|^{o(1)}$, regardless the structure of the graphs, assuming that the local limits of $\ust(G_n)$ and $\mst(G_n)$ exist and differ. In order to do this, for $\beta_n\ll d_n$, we show that our networks are examples of high degree almost balanced electric networks. On the other hand, for $\beta_n\gg d_n\log|V(G_n)|$, we show that the conductances of the edges become so different that $\wst^{\beta_n}(G_n)$ and $\mst(G_n)$ typically differ only by $o(|V(G_n)|)$ edges, resulting in the agreement of the local limits.

Our general approach has another advantage: we could decrease the window of the phase transition of \cite{makowiec2024local} even in the special case $G_n=K_n$, since our computation requires milder conditions on $\c_{K_n, \beta_n}$. For the complete graph $K_n$, we improve the parameters in the $\ust$-like phase to $\beta_n\ll n$ from $\beta_n\ll n/\log n$ for the local convergence and the expected edge overlaps; and in the $\mst$-like phase to $\beta_n\gg n\log n$ for both properties from $\beta_n\gg n\log^2 n$ for the expected edge overlaps and from $\beta_n\gg n\log(n)^\lambda$ with $\lambda\rightarrow \infty$ arbitrary slowly for the local limit. We remark that, for $\beta_n\ll n/\log n$, the electric network $\c_{K_n, \beta_n}$ is a balanced expander, while for $\beta_n\ll n$ is only an almost balanced nonexpander which requires more delicate tools to be studied.

\subsubsection{Local limit for high degree almost balanced graphs}
In this paper, $B_{G}(o, r)$ denotes the $r$-ball around $o$ in $G$, and we mean by $\cong$ rooted graph isomorphism, for more details, see Subsection \ref{subsec:background}.

\begin{theorem}\label{thm:local_ust}
    We consider an almost balanced sequence of simple connected graphs $G_n$ of parameters $d_n\rightarrow \infty$ and $K>1$ such that the local limit of $\ust(G_n)$ exists.
    
    If $\beta_n\ll d_n$, then the local limit of $\wst^{\beta_n}(G_n)$ exists and it is the same as the one for $\ust(G_n)$. More precisely, for $X\sim \unif(V(G_n))$ and for any rooted tree $T$, even the quenched neighbor statistics asymptotically are the same with high probability: we have
    \begin{gather}\label{formula:local_wst_vs_ust}
    \left|\P(B_{\wst(\c_n)}(X, r)\cong T\given \mathbf{U})-\P(B_{\ust(G_n)}(X, r)\cong T)\right|=o(1)
\end{gather}
with probability $1-o(1)$.
\end{theorem}
For deterministic networks, see Theorem \ref{thm:local_det}. The proof of Theorem \ref{thm:local_ust} can be found in Subsection \ref{subsubsection:wst_ust_agree}.

\begin{theorem}\label{thm:local_mst}
    We consider a sequence of simple connected graphs $G_n$ such that the local limit of $\mst(G_n)$ exists.

    If $\beta_n\gg \frac{|E(G_n)|}{|V(G_n)|}\log |V(G_n)|$, then the local limits of $\wst^{\beta_n}(G_n)$ and $\mst(G_n)$ agree. More precisely, for $X\sim \unif(V(G_n))$ and any $r\ge 0$, even the $r$-balls agree with high probability: we have \begin{gather}\label{formula:local_wst_vs_mst}
        \P\left(B_{\mst(G_n)}(X, r)= B_{\wst^{\beta_n}(G_n)}(X, r)\given \mathbf{U}\right)=1-o(1)
    \end{gather}
with probability $1-o(1)$.
\end{theorem}
The proof of Theorem \ref{thm:local_mst} can be found in Subsection \ref{subsec:mstlike_local_limit}.

For any high degree almost balanced sequence $G_n$, we have that $\frac{|E(G_n)|}{|V(G_n)|}=\Theta(d_n)$, so there is a phase transition around $\beta_n=d_n|V(G_n)|^{o(1)}$ regarding the local limits (assuming they exist and differ):
\begin{corollary}[The local limit of $\wst^{\beta_n}(G_n)$ on high degree almost balanced graphs]
    We consider an almost balanced sequence of simple connected graphs $G_n$ of parameters $d_n\rightarrow \infty$ and $K>1$ such that the local limits of $\ust(G_n)$ and $\mst(G_n)$ exist.
    \begin{itemize}
        \item[a)] If $\beta_n\ll d_n$, then the local limits of $\wst^{\beta_n}(G_n)$ and $\ust(G_n)$ agree, and (\ref{formula:local_wst_vs_ust}) holds with probability $1-o(1)$.
        \item[b)] If $\beta_n\gg d_n\log |V(G_n)|$, then the local limits of $\wst^{\beta_n}(G_n)$ and $\mst(G_n)$ agree, and (\ref{formula:local_wst_vs_mst}) holds with probability $1-o(1)$.
    \end{itemize}
\end{corollary}
Notice that a $d_n$-regular sequence $G_n$ with $d_n\rightarrow \infty$ satisfy the conditions of this corollary with different local limits of the uniform and minimum spanning trees: the local limit of $\ust(G_n)$ is the Poisson(1) branching tree conditioned to survive forever \cite[Theorem 1.1]{nachmias2022local} which has quadratic volume growth \cite[Remark 1.22]{kesten1986subdiffusive}; and the local limit of $\mst(G_n)$ is the component containing the root in the wired minimal spanning forest on the Poisson-weighted infinite tree \cite[Theorem 6.1]{nachmias2024wired} which has cubic volume growth \cite[Theorem 1.2]{addario2013local}.

\subsubsection{Expected edge overlaps for high degree graphs}

In order to measure how disordered $\wst^{\beta_n}(G_n)$ is, one can consider the expected edge overlaps:
\begin{defi*}\label{def:edge_overlap}
    Conditioning on $\{U_e\}_{e}$, we choose two independent  $\mathcal{T}, \mathcal{T}'\sim \wst(\c_{G, \beta})$ for $\c_{G, \beta}=\{\exp(-\beta U_e)\}_{e\in E(G)}$, and we denote by $\mathcal{O}(G, \beta):=\E[\E[|E(\mathcal{T})\cap E(\mathcal{T}')|\given \{U_e\}_{e\in E(G)}]]$ the \emph{expected edge overlaps}. 
\end{defi*}
An advantage of studying the expected edge overlaps instead of the local limit is that we do not have to assume the existence and disagreement of the local limits of $\ust(G_n)$ and $\mst(G_n)$ to get a phase transition.

\begin{defi}\label{def:high_deg}
We call a sequence of electric networks $\{(G_n, \c_n)\}_{n\ge 1}$ \emph{high degree} with parameter $\gamma_n$ if, as $n\rightarrow \infty$, at least $(1-o(1))|V(G_n)|$ of the vertices $v$ of $G_n$ have conductances $C_v\ge \gamma_n$. 

We call a sequence of graphs $\{G_n\}_{n\ge 1}$ \emph{high degree} with parameter $d_n$ if the electric network $(G_n, \mathbf{1}_{|E(G_n)})$ is high degree with parameter $\gamma_n=d_n$.
\end{defi}

\begin{theorem}\label{thm:edge_overlap_ust} We consider a sequence of simple connected graphs $\{G_n\}_n$ which is high degree with parameter $d_n\rightarrow\infty$. 
        
If $\beta_n\ll d_n$, then $$\mathcal{O}(G_n, \beta_n)=o(1)|V(G_n)|.$$
\end{theorem}
In Theorem \ref{thm:edge_overlap_det}, we give a similar result for electric networks of deterministic conductances. The proof of Theorem \ref{thm:edge_overlap_ust} is in Subsection \ref{subsubsection:wst_ust_agree}.

\begin{theorem}\label{thm:edge_overlap_mst} We consider a sequence of simple connected graphs $\{G_n\}_n$ with $|V(G_n)|\rightarrow \infty$.
        
If $\beta_n\gg \frac{|E(G_n)|}{|V(G_n)|}\log |V(G_n)|$, then $$\mathcal{O}(G_n, \beta_n)=(1-o(1))|V(G_n)|.$$
\end{theorem}
The proof is in Subsection \ref{subsec:mstlike_edge_overlap}.

Then, the phase transition regarding the expected edge overlaps is immediate on high degree almost balanced graphs around $\beta_n=d_n|V(G_n)|^{o(1)}$:
\begin{corollary}
    We consider $d_n\rightarrow \infty$ and a sequence of simple connected graphs $\{G_n\}_n$ satisfying $|{v:\;\deg(v)\ge d_n}|\ge (1-o(1)) |V(G_n)|$ and $|E(G_n)|\le O(1)|V(G_n)|d_n$.
    \begin{itemize}
        \item[a)] If $\beta_n\ll d_n$, then $\mathcal{O}(G_n, \beta_n)=o(1)|V(G_n)|.$
        \item[b)] If $\beta_n\gg d_n\log |V(G_n)|$, then $\mathcal{O}(G_n, \beta_n)=(1-o(1))|V(G_n)|.$
    \end{itemize}
\end{corollary}

\subsubsection{Expected total length for the complete graph}

It is known that the expected total length $\E\left[\sum_{e\in \mst(K_n)}U_e\right]$ tends to $\zeta(3)$ \cite{frieze1985value, aldous1990random, steele2002minimal}.

\begin{theorem}\label{thm:length}We use the notation $L(T):=\sum_{e\in E(T)}U_e$.
    \begin{itemize}
        \item[a)] If $\beta_n\gg n\log n$, then $\E\left[L(\wst^{\beta_n}(K_n))\right]\rightarrow \zeta(3)$ as $n\rightarrow\infty$. 
        
        \item[b)]If $\beta_n=O(n\log n)$, then $\E\left[L(\wst^{\beta_n}(K_n))\right]=O\left(\frac{n\log n}{\beta_n+\log n}\right)$.
        \item[c)] If $\beta_n\ll \frac{n}{\log n}$, then $\E[L(\wst^{\beta_n}(K_n))]=(1+o(1))\frac{n}{\beta_n}\frac{1-\beta_ne^{-\beta_n}-e^{-\beta_n}}{1-e^{-\beta_n}}$.
    \end{itemize}
\end{theorem}
The proof of Theorem \ref{thm:length} is in Subsection \ref{subsec:length}. It was also in version 2 of \cite{kusz2024diameter} but we decided to omit it from the final version of that paper and to include here as the theorem fits here more. Theorem \ref{thm:length} improves the parameter of \cite{makowiec2024local} for the $\mst$-like phase regarding the expected total length from $\beta_n\gg n\log^5 n$ to $\beta_n\gg n\log n$.

It is reasonable to include the expected total length $\E[L(T(G_n))]$ among the local properties for a random spanning tree model $T(G_n)$, since the precise knowledge of the local label statistics determines the asymptotics of $\E[L(T(G_n))]$: for a uniform vertex $X$, if we know $\{U_{(X, u)}\}_{u:\; (X, u)\in T(G_n)}$ with probability at least $1-o(1/n^2)$, then we know $\E[L(T(G_n))]=\frac{n}{2}\E[\sum_{u:\;(X,u)\in T} U_{(X, u)}]$ up to some $o(1)$ additive error as $|\sum_{u:\;(X,u)\in T} U_{(X, u)}|\le n-1$. For example, in \cite{aldous1990random}, the expected total length was studied in the local framework. 

However, unlike the behavior of the expected edge overlap, the expected total length of the $\wst^{\beta}$ and the $\mst$ is sensitive to the graph, even a vanishing proportion of bridge edges can change the asymptotics, since $\E[\mathds{1}[e\in\text{a tree}]U_e]=\E[U_e]=1/2$ for any bridge edge $e$. For example, for any $\log n\ll f_n\ll n$, we denote by $G_n$ the graph obtained as attaching $f_n$ leaves to the complete graph $K_n$. Then, $G_n$ is almost $n$-regular, $\mst(K_n)$ and $\mst(G_n)$ have that same local limit, and $|V(G_n)|=(1+o(1))|V(K_n)|$, but $\E[L(\mst(G_n))]\sim f_n/2\gg\zeta(3)\sim\E[L(\mst(K_n))]$ and the change of the behaviors of $L(\wst^{\beta_n}(K_n))$ and $L(\wst^{\beta_n}(G_n))$ happen around $n^{1+o(1)}$ and $n/f_n$, respectively: there is the phase $n/f_n\ll \beta_n \ll n/\log n$, such that $\E[L(\wst^{\beta_n}(K_n))]=\Theta(n/\beta_n)\gg \E[L(\mst(K_n))]$ while $\E[L(\wst^{\beta_n}(G_n))]=\Theta(n/\beta_n)+f_n/2\sim \E[L(\mst(G_n))]$.

If one would like to study the phases of $\E[L(\wst^{\beta_n}(G_n))]$ for general sequences $G_n$, then natural candidates for $G_n$s are the robust sequences of dense graphs $G_n$ tending to a graphon, since the asymptotics of $\E[L(\mst(G_n))]$ is understood \cite{hladky2023random} for such $G_n$s.
\subsection{Outline}
Section~\ref{sec:ust-like} is devoted to the local limit results for $\wst$s on deterministic high degree almost balanced networks and $\wst^{\beta_n}(G_n)$ on high degree almost balanced graphs with slowly growing $\beta_n$s. After a brief summary of the background collected in Subsection~\ref{subsec:background}, we introduce assumptions on our environments and state our main results on deterministic electric networks, including Theorem~\ref{thm:local_det} which is a generalization of Theorem~\ref{thm:local_ust_almost_reg} and a useful tool for Theorem~\ref{thm:local_ust}. Then, in Subsection~\ref{subsec:proofs_of_results}, we introduce our main lemmata and prove our main results. Many of our results are based on the idea of Lemma~\ref{lemma_main:Reff_edge} which is a generalization of Lemma 3.1 of \cite{nachmias2022local}. Then, we prove our lemmata in Subsection~\ref{subsec:proofs_of_lemmata}.

In Section \ref{sec:mst-like}, which occupies the last few pages of this preprint, we work on $\wst^{\beta_n}(G_n)$ for fast-growing $\beta_n$s. This section supersedes and partly overlaps with the first two versions of \cite{kusz2024diameter} regarding the local properties: we are proving the local limit and the expected edge overlaps results not only for the complete graph, but for high degree almost balanced sequences, too. We felt it important to do this generalization in order to complement the results of Section~\ref{sec:ust-like}. We study the expected edge overlaps in Subsection~\ref{subsec:mstlike_edge_overlap} and the local limit in Subsection~\ref{subsec:mstlike_local_limit}. In Subsection \ref{subsec:length}, we also include the expected local length results from the second version of \cite{kusz2024diameter} for the sake of completeness; this is the only result in our recent work which comes fully from an earlier preprint. 

Therefore, the submitted version of \cite{kusz2024diameter}---which can be found as version 3 of that preprint on arXiv---and this preprint do not overlap. 

\subsection{Acknowledgments}
This work was partially supported by the ERC Synergy Grant No. 810115–DYNASNET and by the Hungarian National Research, Development and Innovation Office grant 152849.

I am extremely grateful to my PhD advisor, Gábor Pete for the useful discussions and for his comments on the manuscript. 

I am thankful to Asaf Nachmias for his suggestion using Lemma 2.11 in the correction of \cite{nachmias2022local} and on how to check its applicability. 

\section{\texorpdfstring{$\wst$}{}'s on deterministic balanced networks and slowly growing \texorpdfstring{$\beta_n$}{}'s}\label{sec:ust-like}
\subsection{Background}\label{subsec:background}
\subsubsection{Weighted spanning trees on deterministic electric networks}

We recall Definition \ref{defi:wst}: for an electric network $(G, \c)$, $\wst(\c)$ chooses any spanning tree $T$ of $G$ with probability $\P(\wst(\c)=T)\propto \prod_{e\in T}c(e)$. 

The two classic algorithms to generate $\wst(\c)$ are Wilson's and the Aldous-Broder algorithms, detailed in, for example Theorem 4.1 and Corollary 4.9 of \cite{PTN}. Note that both of them use random walks on the electric network $(G, \c)$ which steps from $v$ to one of its neighbors $u$ with probability $p(v, u):=\frac{c(v, u)}{C_v}$ where $C_v=\sum_{u:\;u\sim v}c(v, u)$.

The effective resistance between two vertices $u, v\in V(G)$ can be defined in different ways, for example, writing $X_n$ for the random walk on $(G, \c$), $\tau_u^+:=\min\{n\ge 1:\; X_n=u\}$ and $\tau_v:=\min\{n\ge 0:\; X_n=v\}$, 
\begin{gather}\label{formula:def_Reff}
    \mathcal{R}_{\mathrm{eff}}(u\leftrightarrow v):=\frac{1}{C_u\P_u(\tau_v<\tau^+_u)}.
\end{gather}
More generally, for a subset $W\subseteq V(G)$, we consider $\mathcal{R}_{\mathrm{eff}}(u\leftrightarrow W):=(C_u\P_u(\tau_W<\tau^+_u))^{-1}$ for $u\notin W$ and $\mathcal{R}_{\mathrm{eff}}(u\leftrightarrow W):=0$ for $u\in W$. If we would like to emphasize that we are working on an electric network of conductances $\c$, we write $\mathcal{R}^{\c}_{\mathrm{eff}}$ instead of $\mathcal{R}_{\mathrm{eff}}$. 

Now, we collect the background about effective resistances used in this paper.

One of the basic properties of effective resistances is Rayleigh’s Monotonicity Principle: if there are two electric networks $(G, \c)$ and $(G, \c')$ on a graph $G$ satisfying $c(e)\le c'(e)$ for any $e\in E(G)$, then $\mathcal{R}^{\c}_{\mathrm{eff}}(u\leftrightarrow v)\ge \mathcal{R}^{\c'}_{\mathrm{eff}}(u\leftrightarrow v)$ for any $u, v\in V(G)$.

The Nash-Williams inequality \cite[(2.13)]{PTN} is also important for us: if $a\ne z\in V(G)$ and $\Pi_1,\ldots ,\Pi_k$ pairwise disjoint cutsets separating $a$ and $z$, then
$$\Reff(a \leftrightarrow z) \ge \sum_{i=1}^k \frac{1}{\sum_{e\in \Pi_i}c(e)}.$$

We will also use the commute time identity \cite[Corollary 2.21]{PTN}: for $a, x\in V(G)$, we have $$\E_{a}[\tau_{x}]+\E_{x}[\tau_a]=2\sum_{e\in E(G_n)}c(e) \;\Reff(x\leftrightarrow a).$$

In this paper, we study effective resistances a lot because of the relationship between weighted spanning trees and electric networks \cite{kirchhoff1847ueber}, which can be also found in Section 4.2 of \cite{PTN}: 
\begin{theorem*}[Kirchhoff’s Effective Resistance Formula, \cite{kirchhoff1847ueber}] For an electric network $(G, \c)$ and $(u, v)\in E(G)$, one has that
$\P((u, v)\in \wst(\c))=c(u,v)\mathcal{R}_{\mathrm{eff}}(u\leftrightarrow v)$.    
\end{theorem*}

The spatial Markov property holds for the $\wst$s which can be shown by combining Exercises 4.5 and 4.6 of \cite{PTN}:
\begin{theorem*}[Spatial Markov property]
    For an electric network $(G, \c)$, we consider some disjoint sets of edges $A, B \subseteq E(G)$ such that $\c$ restricted to $G\setminus B$ is connected, $A$ has no cycles and $c(e)>0$ for any $e\in A$, and we denote by $(G/A)\setminus B$ the graph obtained
from $G$ by contracting the edges of $A$ and erasing the edges of $B$.

Then, the law of the edges of $\wst(\c)$ on $G$ conditioned on the event $\{A \subseteq \wst(\c)\text{ and }B\cap \wst(\c)=\emptyset\}$ is the same as the law of the union of $A$ and the edges of $\wst(\c)$ on $(G/A)\setminus B$.
\end{theorem*}
There is negative association of the edges of $\wst$s \cite[Exercise 4.6]{PTN}, resulting in
\begin{gather}\label{formula:neg_assoc}
    \P(e_1, \ldots, e_k \in \wst(\c))\le \prod_{1\le i\le k}\P(e_i\in \wst(c))
\end{gather}
for any distinct edges $e_1, \ldots, e_k\in E(G)$.

\subsubsection{Local limits}
The local convergence was first introduced in \cite{benjamini2011recurrence}.
For rooted graphs $(G_1, o_1)$ and $(G_2, o_2)$, we call a bijective map $\Phi: V(G_1)\mapsto V(G_2)$ rooted ismorphism if $\Phi(o_1)=o_2$ and $(x, y)\in E(G_1)$ if and only if $(\Phi(x), \Phi(y))\in E(G_2)$.  We denote by $B_{G}(v, r)$ the $r$-ball around $v$ w.r.t.~the usual graph distance in $G$. 

The set of connected rooted, locally finite graphs, up to rooted isomorphism, is denoted by $\mathcal{G}_{\bullet}$. We endow $\mathcal{G}_{\bullet}$ with a natural metric: 
the distance of $(G_1, o_1)$ and $(G_2, o_2))$ is defined to be $2^{-R}$, where $R$ is the largest $r$ such that there exists a rooted isomorphism between $B_{G_1}(o_1, r)$ and $B_{G_2}(o_2, r)$. We consider the Borel $\sigma$-algebra on this space which is a Polish space.

We say that a sequence of (possibly random) finite graphs $\{G_n\}_n$ converges to a random element $(G, o)$ of $\mathcal{G}_{\bullet}$ if the random rooted graphs $B_{G_n}(X_n, r)$ converge in distribution to $B_G(o, r)$ for any $r\ge 0$ where $X_n\sim \unif(V(G_n))$, i.e., for every $r\ge 0$ and rooted graph $(H, v)$ of height $r$, we have 
$$\lim_{n\rightarrow \infty}\P(B_{G_n}(X_n, r)\cong (H, v))=\P(B_{G}(o, r)\cong (H, v)).$$

\subsection{Results for deterministic electric networks}\label{subsec:det_results}

For the edge overlaps result, we need an electric network generalization of high degree graphs.
\begin{defi}
For a sequence of electric networks $(G_n, \c_n)$ and $\gamma_n>0$, let us consider $W_n:=W(G_n, \gamma_n):=\left\{v\in V(G_n):\; C_v\ge \gamma_n\right\}$ where $C_v=\sum_{(u,v)\in E(G_n)}c(u, v)$.

Whenever we would like to emphasize the dependence of all parameters, we write $W_n(G_n, \c_n, \gamma_n)$.
\end{defi}

\begin{assum}\label{assum:for_edge_overlap}
Suppose that $(G_n, \c_n)$, $\gamma_n>0$, $\delta_n\rightarrow 0$ and $n_0\in \mathbb{N}$ satisfy
    \begin{itemize}
        \item[$\circ$] $|W_n|\ge (1-\delta_n)|V(G_n)|$ for $\forall n\ge n_0$ and 
        \item[$\circ$] $c_n(e)\le \delta_n\gamma_n$ for $\forall e\in E(G_n)$ and $(G_n, \c_n)$ is connected on a simple graph $G_n$ for $\forall n\ge n_0$.
    \end{itemize}
\end{assum}
The first condition imply that $\sum_{v\in W_n}C_v\ge (1-\delta_n)\gamma_n |V(G_n)|$ and $\sum_{v\notin W_n}C_v\le |W_n^c|\gamma_n\le \delta_n|V(G_n)|\gamma_n$, so we have $\sum_{v\notin W_n}C_v=o(1)\sum_{v\in V(G_n)}C_v$; while the second condition excludes the possibility that an edge $e$ has conductance comparable to the average $\frac{1}{|V(G_n)|}\sum_{v\in V(G_n)}C_v$ of the conductances of vertices.

For $\c\equiv 1$, Assumption \ref{assum:for_edge_overlap} means high degree graph sequence: the degree of a typical vertex is at least $d_n\rightarrow\infty$. Moreover, because of the second condition, $\delta_n\ge 1/d_n$ is automatic.

For an electric network $\c$, we denote by $\mathcal{T}(\c)$ and $\mathcal{T}'(\c)$ two independent copies of the weighted spanning tree, and by $\mathcal{O}(\c)$ the expected edge overlaps $\E[|\mathcal{T}(\c)\cap \mathcal{T}'(\c)|]$.
\begin{theorem}\label{thm:edge_overlap_det}
Suppose that $(G_n, \c_n)$, $\gamma_n>0$ and $\delta_n\rightarrow 0$ satisfy Assumption \ref{assum:for_edge_overlap}, $\c_n$ is irreducible and $G_n$ is simple.

Then $ \mathcal{O}(\c_n)=O(\delta_n)|V(G_n)|.$
\end{theorem}

Now, we formulate the high degree balanced networks more precisely.
\begin{defi}
    For a sequence of electric networks $(G_n, \c_n)$, $K>1$ and $\gamma_n>0$, we define the set of vertices of \emph{typical conductance} as $$W_n^\tc:=\left\{v\in V(G_n):\; C_v\in [\gamma_n, K\gamma_n]\right\}.$$
    Whenever we would like to emphasize that we are on $(G_n, \c_n)$, we write $W_n^\tc(\c_n)$. For the electric network $(G_n, \mathbf{1}|_{E{(G_n)}})$ and $\gamma_n:=d_n$, we will abbreviate as $W_n^\tc(G_n):=\{v\in V(G_n):\; \deg(v)\in [d_n, Kd_n]\}$.

    Whenever we emphasize the dependence of all parameters, we write $W_n^\tc(G_n, \c_n, \gamma_n, K)$.
\end{defi}

\begin{assum}\label{assum:for_local_lim}
Suppose that $(G_n, \c_n)$, $K>1$, $\gamma_n>0$, $\delta_n\rightarrow 0$ and $n_0\in \mathbb{N}$ satisfy that for any $n\ge n_0$,
    \begin{itemize}
        \item[$\circ$] $|W_n^\tc|\ge (1-\delta_n)|V(G_n)|$ for $\forall n\ge n_0$,
        \item[$\circ$] $\sum_{v\notin W_n^\tc}C_v\le \delta_n\gamma_n |V(G_n)|$ for $\forall n\ge n_0$ and 
        \item[$\circ$] $c_n(e)\le \delta_n\gamma_n$ for $\forall e\in E(G_n)$ and $(G_n, \c_n)$ is connected on a simple graph $G_n$ for $\forall n\ge n_0$.
    \end{itemize} 
\end{assum}

We denote by $\pi_n$ the unique stationary probability measure of the electric network $(G_n, \c_n)$, and by $\lambda_n$ the uniform probability measure of $V(G_n)$. The first and second conditions of Assumption \ref{assum:for_local_lim} imply that $C_{V(G_n)}:=\sum_{w\in V(G_n)}C_w\in [(1-\delta_n)|V(G_n)|\gamma_n,\; (K+\delta_n)|V(G_n)|\gamma_n]$, so for any so large $n$ that we have $\delta_n\le 1/2$ and any $A\subseteq V(G_n)$, using also these conditions,
\begin{equation}
\begin{aligned} \label{formula:unif_vs_stationary}
    \pi_n(A)=\sum_{w\in A}\frac{C_w}{C_{V(G_n)}}&\in \left[\frac{(|A|\gamma_n-\delta_n\gamma_n|V(G_n)|)^+}{(K+\delta_n)|V(G_n)|\gamma_n}, \frac{K\gamma_n |A|+\delta_n\gamma_n|V(G_n)|}{(1-\delta_n)|V(G_n)|\gamma_n}\right]\\
    &=\left[\frac{(\lambda_n(A)-\delta_n)^+}{K+\delta_n}, \frac{K\lambda_n(A)+\delta_n}{1-\delta_n}\right]\stackrel{\text{large }n}{\subseteq}\left[\frac{(\lambda_n(A)-\delta_n)^+}{2K}, 2K\lambda_n(A)+2\delta_n\right],
\end{aligned}
\end{equation}
where $x^+$ denotes the positive part of $x$. Hence, we obtained that $\lambda_n(A_n)\rightarrow 0$ is equivalent to $\pi_n(A_n)\rightarrow 0$.

For $\c\equiv 1$, we also have $\delta_n\ge 1/d_n$ by the third condition of Assumption \ref{assum:for_local_lim}.

\begin{notation}\label{notation:v_k}
Consider a rooted tree $(T, 1)$ of height $r$ and of vertices $V(T)=\{1, \ldots, k\}$ following the breadth-first search (BFS) of $T$ started at the root vertex 1: we denote by $(p_T(j), j)$ the $j-1^{st}$ edge added by the BFS, where $j$ is the newly explored vertex and $p_T(j)$ its parent. Let $t$ be the number of vertices of distance at most $r-1$ from the root: $V(B_{T}(1, r-1))=\{1, \ldots, t\}$.

For $\mathbf{v}_k:=(v_1, \ldots, v_k)\in V(G_n)^k$, we write $T(\mathbf{v}_k)$ to denote the tree of vertices $V(T(\mathbf{v}_k)):=\{v_1, \ldots, v_k\}$ and of edges $E(T(\mathbf{v}_k)):=\{v_iv_{p_T(i)}:\;2\le i\le k\}$.
\end{notation}

\begin{defi}
    We call a $\mathbf{v}_k\in V(G_n)^k$ \emph{$T$-compatible} if $T(\mathbf{v}_k)$ is an injective $T\mapsto G_n$ homomorphism ,i.e., $v_1, \ldots, v_k$ are distinct and $E(T(\mathbf{v}_k))\subseteq E(G_n)$, moreover $\prod_{e\in E(T(\mathbf{v}_k))}c(e)>0$. In this paper, for any $U\subseteq V(G_n)$, we write $$\C(U):=\{\mathbf{v}_k\in U^k:\mathbf{v}_k\text{ is $T$-compatible}\}.$$
    \end{defi}

\begin{notation}
    For an electric network $(G, \c)$, a rooted tree $(T,1)$ of $k$ vertices and $\mathbf{v}_k$, we write $b_{\c}(v):=\sum_{u:\;u\sim v}\frac{c(u,v)}{C_u}$ and
    $$F_{\c}(T(\mathbf{v}_k)):=\frac{1}{|V(G)|}\exp\left(-\sum_{j=1}^{t}b_{\c}(v_j)\right)\frac{\sum_{j=t+1}^{k}C_{v_j}}{\prod_{j=1}^{k}C_{v_j}}.$$
    
    With a slight abuse of the notation, we also write $b_{G}:=b_{\mathbf{1}|_{E(G_n)}}$ and $F_{G}:=F_{\mathbf{1}|_{E(G_n)}}$, i.e., $b(v):=\sum_{u:\;u\sim v}\frac{1}{\deg(u)}$ and
    $$F_{G}(T(\mathbf{v}_k)):=\frac{1}{|V(G)|}\exp\left(-\sum_{j=1}^{t}b_{G_n}(v_j)\right)\frac{\sum_{j=t+1}^{k}\deg(v_j)}{\prod_{j=1}^{k}\deg(v_j)}.$$
\end{notation}

As usual, for two sequences of positive numbers $\{a_n\}_{n}, \{b_n\}_{n}$, we write $a_n\gg b_n$, $b_n\ll a_n$ or $b_n=o(a_n)$ to mean $\lim_{n\rightarrow \infty}b_n/a_n=0$; we denote $\limsup_{n\rightarrow\infty}b_n/a_n<\infty$ by $b_n=O(a_n)$. If both  $a_n=O(b_n)$ and $b_n=O(a_n)$ holds, then we abbreviate it as $a_n=\Theta(b_n)$. In these notation, we are thinking about $T, k, t$ as parameters; if we mean them otherwise in a proof, we detail that there.

Now, we can state the generalization of Theorem \ref{thm:local_ust_almost_reg} which can be seen as a generalization of Lemma 3.12 of \cite{hladky2018local}:
\begin{theorem}\label{thm:local_det}
Suppose that $(G_n, \c_n)$, $\gamma_n>0$ and $\delta_n\rightarrow 0$ satisfy Assumption \ref{assum:for_local_lim}, $\c_n$ is irreducible and $G_n$ is simple. Then, for any rooted tree $(T, 1)$ of height $r$ and $X\sim \unif(V(G_n))$, we have
\begin{gather}\label{formula:local_det}
    \left|\P(B_{\wst(\c_n)}(X, r)\cong T)-\sum_{\mathbf{v}_k\in \C(W_n^{\tc})} \frac{F_{\c_n}(T(\mathbf{v}_k))}{|\mathrm{Stab}_T|} \prod_{j=2}^k c(v_{p_T(j)}, v_j)\right|=o(1),
\end{gather}
where $|\mathrm{Stab}_T|$ denotes the number of root preserving automorphisms of $(T, 1)$.
\end{theorem}

In Subsection \ref{subsubsection:local-limit_balanced}, we check that Theorem \ref{thm:local_ust_almost_reg} indeed follows from Theorem \ref{thm:local_det}. To do that, we show that $\sum_{\mathbf{v}_k\in \C(W_n^{\tc})} F_{\c_n}(T(\mathbf{v}_k))\prod_{j=2}^k c(v_{p_T(j)}, v_j)\sim e^{-t}(k-t)$ where $k=|V(T)|$ and $t=|V(B_{T}(1, r-1))|$. This gives us for any $T$, that the asymptotics of $\P(B_{\wst(\c_n)}(X, r)\cong T))$ agrees with the one for the Poisson(1) branching process conditioned to survive forever.

Theorem \ref{thm:local_ust} about the local limit of $\wst^{\beta_n}(G_n)$ follows from Theorem \ref{thm:local_det} and 
Lemmata \ref{lemma:rdm_cond_balanced} and \ref{lemma:F_annealed-quenched} which we state precisely later: in Lemma \ref{lemma:rdm_cond_balanced}, we will check that, for any $d_n$-almost balanced sequence $G_n$ and $\beta_n\ll d_n$, the electric network $(G_n, \{\exp(-\beta_nU_e)\}_e)$ is still almost balanced, and in Lemma \ref{lemma:F_annealed-quenched}, we will show that the sum appearing in (\ref{formula:local_det}) for $(G_n, \{\exp(-\beta_nU_e)\}_e)$ and the sum in (\ref{formula:local_det}) for $(G_n, \mathbf{1})$ asymptotically agree both in the annealed and the quenched senses.

\begin{remark}[Connection to the 1-out model]\label{rmrk:1-edge_statement} In the 1-out model of a graph $G$, a random directed subgraph is generated by choosing uniformly a directed edge coming out from each vertex $v$, independently of each other. In our electric network setup, for each $v$, we choose a random edge with probability proportional to its conductance. Then, Theorem \ref{thm:local_det} states that, the local statistics of the weighted spanning trees and the local statistics of the underlying unoriented multigraph of the weighted 1-out model asymptotically agree---see Remark~\ref{rmrk:1-edge_proof} for more details. Similar comparison of the hypertree generalizations of the $\ust$ and 1-edge out model on the special case of high degree $d$-regular graphs was made in \cite{nachmias2025determinantal}.
\end{remark}

\subsection{Proofs of the results} \label{subsec:proofs_of_results}

\subsubsection{Edge-overlap for general graphs}\label{subsubsec:edge_overlap}

The aim of this subsection is to prove Theorem \ref{thm:edge_overlap_det}. We also state Lemma \ref{lemma_main:Reff_edge} which is our main lemma as it is used throughout our paper.

The starting point for the local limit result of \cite{nachmias2022local} is the following: by the Nash-Williams inequality, it is easy to check that, for any pairs of vertices $(u, v)\in E(G)$,
\begin{gather}\label{formula:by_NW}
    \Reff(u\leftrightarrow v)\ge \frac{1}{\deg(u)+1}+\frac{1}{\deg(v)+1}
\end{gather}
holds in the electric network with conductances $c(e)\equiv 1$. Then, on any sequence of high degree regular simple graphs $G_n$, they realized that the sum over $(u,v)\in E(G_n)$ of the RHS of (\ref{formula:by_NW}) is $(1-o(1))|V(G_n)|$. On the other hand, Foster's formula says that $\sum_{(u,v)\in E(G_n)} \Reff(u\leftrightarrow v)=|V(G_n)|-1$, since it is the expected number of edges of the spanning tree by Kirchhoff's formula. Combining these two observations, they obtained the asymptotics of the effective resistance between the endpoints of a typical edge: $\Reff(e^{-}\leftrightarrow e^{+})\sim \frac{1}{\deg(e^{-})}+\frac{1}{\deg(e^{+})}$ with high probability for $e=(e^{-}, e^{+})\sim \unif(E(G_n))$.

In Lemma \ref{lemma_main:Reff_edge}, we extend this reasoning from unweighted graphs to electric networks and from almost regular high degree sequences to high degree sequences. If $(G_n, \c_n)$ is not balanced, then this method is not strong enough for asymptotics of the effective resistance between the endpoints of a typical edge, but it still gives enough information for the proof of Theorem \ref{thm:edge_overlap_ust}.

\begin{lemma}\label{lemma_main:Reff_edge} We suppose that Assumption \ref{assum:for_edge_overlap} is satisfied.
\begin{itemize}
    \item[a)] Then, for every $u\ne v\in V(G_n)$, we have that
    \begin{gather*}
        \Reff(u \leftrightarrow v)\ge \frac{(1-\delta_n)\mathds{1}_{u\in W_n}}{C_u}+\frac{(1-\delta_n)\mathds{1}_{v\in W_n}}{C_v},
    \end{gather*}
    and $\Reff(u \leftrightarrow v)\ge \frac{1}{2C_u}+\frac{1}{2C_v}$.
    \item[b)] Moreover, \begin{gather}\label{formula:lower_bound_close}
        \sum_{(u, v)\in E(G_n)}  c(u, v)\left(\Reff(u\leftrightarrow v)-\frac{(1-\delta_n)\mathds{1}_{u\in W_n}}{C_u}-\frac{(1-\delta_n)\mathds{1}_{v\in W_n}}{C_v}\right)\le 2\delta_n|V(G_n)|.
    \end{gather}
    \item[c)] For $\tE_n(\eps):=\{(u,v)\in E(G_n):\;\Reff(u\leftrightarrow v)\le\eps\}$, we have that 
    \begin{gather}
        \sum_{(u, v)\in E(G_n)\setminus\tE_n(4/\gamma_n)}c(u,v)\Reff(u\leftrightarrow v)\le 4\delta_n|V(G_n)|.
    \end{gather}
\end{itemize}
\end{lemma}
Lemma \ref{lemma_main:Reff_edge} a) is a simple consequence of the Nash-Williams inequality, b) is a generalization of Lemma 3.1 of \cite{nachmias2022local} and c) is an easy corollary of b). The proof is in Subsection \ref{subsubsec:typ_edge}. Observe that, by Kirchhoff's formula, Lemma \ref{lemma_main:Reff_edge} c) gives us \begin{gather}\label{formula:expected_deg_by_kirchhoff}
\E\left[|\wst(\c)\setminus\tE_n(4/\gamma_n)|\right]=\sum_{(u, v)\in E(G_n)\setminus\tE_n(4/\gamma_n)}c(u,v)\Reff(u\leftrightarrow v)=O(\delta_n)|V(G_n)|.
\end{gather}

\begin{remark}\label{remark:NP_countereg} In Lemma 3.1 of \cite{nachmias2022local}, it is shown that $|E(G_n)\setminus \tE_n(\eps_n)|\le \frac{|V(G_n)|}{\eps_n d_n-2}$ for any $d_n$-regular graph $G_n$ and any $\eps_n>\frac{2}{d_n}$; while Lemma 3.1' of \cite{nachmias2022local} states that $|E(G_n)\setminus \tE_n(\eps_n)|\le \frac{(1+O(\delta_n))|V(G_n)|}{\eps_n d_n-2}$ for any $d_n$-\emph{almost} regular sequence and any $\eps_n>\frac{2}{d_n}$. In this remark, we detail that, although Lemma 3.1 is correct, Lemma 3.1' is not true. 

As a counterexample for Lemma 3.1' of \cite{nachmias2022local}, we consider a $d_n\rightarrow \infty$ regular graph sequence $G_n$ with $|V(G_n)|=n$, and from $G_n$, we construct $G_n'$ by adding $m_n$ vertices $M_n$ to $G_n$, and then, from each vertex of $M_n$, add $f_n$ edges going into $V(G_n)$ in an arbitrary way.

If $m_n\ll n$ and $f_n\ll d_n$, then it is clear that $G'_n$ is $d_n$ almost regular. Additionally, if $m_n\gg n/d_n$ also holds, then we have a contradiction with Lemma 3.1' for $\eps_n=1/(f_n+1)$. Indeed, the effective resistance between the endpoints of any edge running between $V(G_n)$ and $M_n$ is at least $1/(f_n+1)$ by (\ref{formula:by_NW}), and there are $m_nf_n$ such edges, but by Lemma 3.1', there should be at most $O(n/(d_n/f_n))$ of them, contradicting with our assumption $m_nf_n\gg n f_n/d_n$.

Lemma 3.1 can not be generalized to Lemma 3.1', as there are worse error terms for almost regular graphs than for regular graphs in the computation. The correct statement of Lemma 3.1' would be that $|E(G_n)\setminus \tE_n(\eps_n)|\le \frac{O(\delta_n)d_n|V(G_n)|}{d_n\eps_n-2}$, but this is not strong enough for the main result.

Fortunately, we can replace Lemma 3.1 with our Lemma \ref{lemma_main:Reff_edge} c): instead of estimating the total number $|E(G_n)\setminus \tE_n(\eps_n)|$, we give a bound on
$\E[|E(\wst(\c))\setminus\tE_n(4/\gamma_n)|]$ by (\ref{formula:expected_deg_by_kirchhoff}), which preserves more information about the effective resistances. 

There are two places in \cite{nachmias2022local} where Lemma 3.1' was used for (almost) regular sequences: in the proof of their main theorem, to show the $o(1)$ term in formula (24) by checking that most of the vertices are good---meaning that they do not touch an edge from $E(G_n)\setminus \tE_n(\frac{d_n}{\log d_n})$---; and in the proof of the tightness, the bound $|E(G_n)\setminus \tE_n(A_k/d_n)|\le \frac{O(1)|V(G_n)|}{A_k-2}$ was used. The first issue is solved by our Lemma \ref{lemma_main:Reff_edge} c) and a careful choice of typicality: we consider nice tuples---see Definition \ref{def:nice}---instead of good vertices. It turns out that our Lemma \ref{lemma_main:Reff_edge} c) is also suitable for correcting the proof of tightness, but we had to rewrite the proof and use Lemma \ref{lemma:unifint_implies_tightness}.
\end{remark}

\begin{proof}[Proof of Theorem \ref{thm:edge_overlap_det}] 
First, we fix a deterministic sequence $\{(G_n, \c_n)\}_{n}$ satisfying Assumption \ref{assum:for_edge_overlap}. We denote by $\mathcal{T}(\c_n)$ and $\mathcal{T}'(\c_n)$ two independent copies of the weighted spanning tree, and by $\mathcal{O}(\c_n)$ the expected edge overlaps $\E[|\mathcal{T}(\c_n)\cap \mathcal{T}'(\c_n)|]$. Our aim is to check that $\mathcal{O}(\c_n)=O(\delta_n)|V(G_n)|$.

We write $f(u):=\frac{(1-\delta_n)\mathds{1}_{u\in W_n}}{C_u}\le \frac{1}{\gamma_n}$. If $0\le a\le x\le 1$, then $x^2\le a^2+2(x-a)$ since $t\rightarrow t^2-2t$ is monotone decreasing on $[0,1]$, so choosing $0\le a:=c(e)(f(u)+f(v))\le c(e)\Reff(u\leftrightarrow v)=:x\le 1$,
\begin{align*}
    \mathcal{O}(\c_n)&=\sum_{e=(u, v)\in E(G_n)}\big(c(e)\Reff(u\leftrightarrow v)\big)^2
    \\&\le \sum_{e\in E(G_n)}c(e)^2\left(f(u)+f(v)\right)^2+2\sum_{e\in E(G_n)}c(e)(\Reff(u\leftrightarrow v)-f(u)-f(v))\\
    &\stackrel{(*)}{\le} \delta_n\gamma_n\frac{2}{\gamma_n}\sum_{(u,v)\in E(G_n)}\left(\frac{\mathds{1}_{u\in W_n}c(u,v)}{C_u}+\frac{\mathds{1}_{v\in W_n}c(u,v)}{C_v}\right)+O(\delta_n)|V(G_n)|\\
    &\le 2\delta_n\sum_{v\in W_n}\frac{C_v}{C_v}+O(\delta_n)|V(G_n)|=O(\delta_n)|V(G_n)|,
\end{align*}
where $(*)$ follows from $f(u)+f(v)\le \frac{2}{\gamma_n}$, $c(e)\le \delta_n\gamma_n$ and Lemma \ref{lemma_main:Reff_edge} b).
\end{proof}

\subsubsection{The local limit for high degree almost regular and balanced networks}\label{subsubsection:local-limit_balanced}
For the proof of our local convergence results Theorem \ref{thm:local_det} and Theorem \ref{thm:local_ust_almost_reg}, we introduce Lemmata \ref{lemma_main:tightness}, \ref{lemma_main:prob_not_nice}, \ref{lemma_main:F_on_W^tc} and \ref{lemma_main:Reff_tree_bal} which hold under Assumption \ref{assum:for_local_lim}. We emphasize that the idea of Lemma \ref{lemma_main:Reff_edge} is used several times during the proofs of these lemmata.

For the local convergence, tightness of the $r$-balls is essential: 
\begin{lemma}[Generalization of Theorem 4.2 of \cite{nachmias2022local}]\label{lemma_main:tightness} If $(G_n, \c_n)$ satisfies Assumption \ref{assum:for_local_lim} for some $K>1$, $\gamma_n>0$ and $\delta_n\rightarrow 0$, then for $X\sim \unif(V(G_n))$, we have 
$$\lim_{M\rightarrow\infty}\sup_n\P\left(|B_{\wst(\c_n)}(X, r)|\ge M\right)=0.$$
\end{lemma}
By Lemma \ref{lemma:unifint_implies_tightness}, a sufficient condition for tightness is the uniform integrability of the degree in the $\wst$ of a uniform vertex which we check in Subsection \ref{subsubsec:tightness_nice}.

In the proof of Theorem \ref{thm:local_ust}, we would like to use not only that a typical vertex has degree $\in [\gamma_n, K\gamma_n]$, but we also need a typicality assumption on the edges coming out of a vertex which we call `having typical neighbors'.

\begin{defi}\label{def:typ_neigh}
We say that a $v\in V(G_n)$ has \emph{typical neighbors} if
$$\sum_{u:\; (u, v)\in E(G_n)\setminus\tE_n(4/\gamma_n)}c(u,v)\Reff(u\leftrightarrow v)+\sum_{u\in V(G_n)\setminus W_n^\tc:~(u, v)\in E(G_n)} \frac{c(v, u)}{\gamma_n}\le \sqrt{\delta_n}$$
where $\tE_n(\eps):=\{(u,v)\in E(G_n):\;\Reff(u\leftrightarrow v)\le\eps\}$. The set of vertices having typical neighbors in the electric network $\c_n$ is denoted by $W_n^{\tn}=W_n^{\tn}(\c_n)$.

We write $\tW_n$ for the set of vertices which are both of typical conductances and of typical neighbors, i.e., $$\tW_n:=W_n^{\tc}\cap W_n^{\tn}.$$

We will write $W_n^{\tn}(G_n)$ and $\tW_n(G_n)$ if we are on the electric network $(G_n, \mathbf{1}|_{E(G_n)})$.
\end{defi}

The sums of Definition \ref{def:typ_neigh} will appear naturally in our proofs. For a $v\in \tW_n$, let us write now $N_v$ for the random neighbor of $v$ chosen with probability $\propto c(v, \cdot)$. We are choosing a neighbor this way in several places of this paper. It is easy to check that, with probability at least $1-O(\sqrt{\delta_n})$, we have $N_v\in W_n^\tc$ and $\Reff(v\leftrightarrow N_v)\le \frac{4}{\gamma_n}$. Moreover, regarding the first sum of the definition, it is useful to notice that $\E[\deg_{\wst(\c)\cap E}(v)]=\sum_{u:\; (u, v)\in E}c(u,v)\Reff(u\leftrightarrow v)$ for any vertex $v$ and any edge set $E$ by Kirchhoff's formula, where we take the expectation w.r.t.~the $\wst$ measure.

Moreover, for any $v\in \tW_n$, we have  $b(v)=\Theta(1)$: by the second inequality of Lemma~\ref{lemma_main:Reff_edge} a), 
\begin{equation}\label{formula:b(v)_typ}
    \begin{aligned}
        b(v)-\sum_{u\in W_n^\tc:~u\sim v} \frac{c(v, u)}{C_u}&=\sum_{u\notin W_n^\tc:~u\sim v} \frac{c(v, u)}{C_u}\le \sum_{u\notin W_n^\tc:~u\sim v} c(v, u)\left(2\Reff(u\leftrightarrow v)\right)\\
        &\le \sum_{\substack{u\in V(G_n):\; u\sim v\\(u, v)\notin \tE_n(4/\gamma_n)}}2c(u,v)\Reff(u\leftrightarrow v)+\sum_{\substack{u\notin W_n^\tc:~u\sim v\\(u, v)\in \tE_n(4/\gamma_n)}} \frac{8c(v, u)}{\gamma_n}\\
        &\le 8\sqrt{\delta_n},
    \end{aligned}
\end{equation}
hence (\ref{formula:b(v)_typ}) immediately implies that $b(v)\in \left(\frac{C_v-\sqrt{\delta_n}\gamma_n}{K\gamma_n}, \frac{C_v}{\gamma_n}+8\sqrt{\delta_n}\right)=\left(\frac{1-\sqrt{\delta_n}}{K}, K+8\sqrt{\delta_n}\right)=\Theta(1)$.

\begin{defi}\label{def:nice}
We call a subgraph $H$ of $G_n$ \emph{nice}, if $V(H) \subseteq \tW_{n}$ and $E(H)\subseteq \tE_n(4/\gamma_n)$ where $\tE_n(\eps)=\{(u,v)\in E(G_n):\;\Reff(u\leftrightarrow v)\le \eps\}$. For any rooted tree $(T, 1)$, we write
$$\CN:=\{\mathbf{v}_k\in V(G_n)^k:\text{ $\mathbf{v}_k$ is $T$-compatible and $T(\mathbf{v}_k)$ is nice}\}.$$
\end{defi}

\begin{lemma}\label{lemma_main:prob_not_nice} If $(G_n, \c_n)$ satisfies Assumption \ref{assum:for_local_lim}, then $|\tW_n|\ge (1-
10\sqrt{\delta_n})|V(G_n)|$, and for $X\sim \unif(V(G_n))$, we have
\begin{gather*}\label{formula:prob_nonnice}
        \P(B_{\wst(\c_n)}(X, r) \text{ is not nice})\rightarrow 0.
    \end{gather*}
\end{lemma}
Lemma \ref{lemma_main:prob_not_nice} is shown in Subsection \ref{subsubsec:tightness_nice}, and it works similarly as the combination of Claim 6.2 and Lemma 4.1 in \cite{nachmias2022local}.

\begin{lemma}\label{lemma_main:F_on_W^tc}
We consider $(G_n, \c_n)$, $U_n \subseteq V(G_n)$ and a rooted tree $(T, 1)$. Denoting $f_n:=\frac{|V(G_n)\setminus U_n|}{|V(G_n)|}$, we have that
    \begin{align*}
        \sum_{\mathbf{v}_k\in \C(W_n^\tc)\setminus \C(U_n)}F_{\c_n}(T(\mathbf{v}_k))\prod_{j=2}^k c(v_{p_T(j)}, v_j)=O(f_n).
    \end{align*}
Moreover, if we additionally assume that $(G_n, \c_n)$ satisfies Assumption \ref{assum:for_local_lim}, then 
\begin{align}\label{formula:F_on_W^tc_second}
    \sum_{\mathbf{v}_k\in \C(W_n^\tc\cap U_n)}\frac{C_{v_1}}{C_{V(G_n)}}\prod_{j=2}^k \frac{c(v_{p_T(j)}, v_j)}{C_{v_{p_T(j)}}}=1+O(\max(f_n, \delta_n))
\end{align}
holds, where $C_{V(G_n)}=\sum_{v\in V(G_n)}C_v$.
\end{lemma}
The proof of Lemma \ref{lemma_main:F_on_W^tc} is in Subsection \ref{subsubsec:explicit}.

\begin{defi}
    For any rooted tree $(T, 1)$, we consider $$\A_n(\eps_n):=\left\{\mathbf{v}_k\in \CN:\; \frac{\P\left(T(\mathbf{v}_k)= B_{\wst(\c_n)}(X, r)\right)}{F_{\c_n}(T(\mathbf{v}_k))\prod_{j=2}^k c(v_{p_T(j)}, v_j)}\in (1-\eps_n, 1+\eps_n)\right\}.$$
\end{defi}
\begin{lemma}\label{lemma_main:Reff_tree_bal} If Assumption \ref{assum:for_local_lim} is satisfied, then for any $T$, there exists $\eps_n\rightarrow 0$ such that
\begin{gather}\label{formula:sum_over_comp_not_in_A}
    \sum_{\mathbf{v}_k\in \CN\setminus \A_n(\eps_n)} \prod_{e\in T(\mathbf{v}_k)} c(e)\le o(1)|V(G_n)|\gamma_n^{k-1}.
\end{gather}
\end{lemma}
We prove this lemma in Subsection~\ref{subsubsec:explicit}. We detail the difference of this lemma from \cite[Lemma 5.1]{nachmias2022local} after the proof of Theorem \ref{thm:local_ust_almost_reg}.

\begin{proof}[Proof of Theorem \ref{thm:local_det}] We start by the outline of the proof, and we are detailing each equalities below.

\begin{align}
    \P(B_{\wst(\c_n)}(X, r)\cong T)&=\P(B_{\wst(\c_n)}(X, r)\cong T\text{ and it is nice})+o(1)\label{formula:main_ineq_first_line}\\
    &=\sum_{\mathbf{v}_k\in \CN}\frac{\P(B_{\wst(\c_n)}(v_1, r)=T(\mathbf{v}_k))}{|V(G_n)||\mathrm{Stab}_T|}+o(1)\label{formula:main_ineq_second_line}\\
    &=\sum_{\mathbf{v}_k\in \A_n(\eps_n)}\frac{\P(B_{\wst(\c_n)}(v_1, r)=T(\mathbf{v}_k))}{|V(G_n)||\mathrm{Stab}_T|}+o(1)\label{formula:main_ineq_third_line}\\
    &=(1+o(1))\sum_{\mathbf{v}_k\in \A_n(\eps_n)}\frac{F_{\c_n}(T(\mathbf{v}_k))}{|\mathrm{Stab}_T|}\prod_{j=2}^k c(v_{p_T(j)}, v_j)+o(1)\label{formula:main_ineq_fourth_line}\\
    &=(1+o(1))\sum_{\mathbf{v}_k\in\C(W_n^\tc)}\frac{F_{\c_n}(T(\mathbf{v}_k))}{|\mathrm{Stab}_T|}\prod_{j=2}^k c(v_{p_T(j)}, v_j)+o(1). \label{formula:main_ineq_fifth_line}
\end{align}

Now, we detail the previous steps.

The equality of (\ref{formula:main_ineq_first_line}) holds by Lemma \ref{lemma_main:prob_not_nice}.

For (\ref{formula:main_ineq_second_line}), note that $\{B_{\wst(\c_n)}(v_1, r)=T(\mathbf{v}_k)\}=\{B_{\wst(\c_n)}(v_1, r)=T(\mathbf{v}_k')\}$ if there is a root preserving automorphism between $T(\mathbf{v}_k)$ and $T(\mathbf{v}_k')$, otherwise these events are disjoint.

For (\ref{formula:main_ineq_third_line}), we check $0\le (\text{the sum in (\ref{formula:main_ineq_second_line})})-(\text{the sum in (\ref{formula:main_ineq_third_line})})=o(1)$ by
\begin{align*}
        \sum_{\mathbf{v}_k\in \CN\setminus \A_n}\P(B_{\wst(\c_n)}(v_1, r)=T(\mathbf{v}_k))&\le \sum_{\mathbf{v}_k\in \CN\setminus \A_n}\prod_{i=2}^k \Reff \left(v_i \leftrightarrow \{v_1, \ldots, v_{i-1}\} \right)c(v_{p_T(i)}, v_i)\\
        &\le \sum_{\mathbf{v}_k\in \CN\setminus \A_n}\prod_{i=2}^k \Reff \left(v_i \leftrightarrow v_{p_T(i)} \right)c(v_{p_T(i)}, v_i)\\
        &\le\left(\frac{4}{\gamma_n}\right)^{k-1}\sum_{\mathbf{v}_k\in \CN\setminus \A_n} \prod_{e\in T(\mathbf{v}_k)} c(e)=o(1)|V(G_n)|,
\end{align*}
where the first step is by the spatial Markov property, the second from monotonicity, the third by the niceness of $\mathbf{v}_k$ and the fourth by Lemma \ref{lemma_main:Reff_tree_bal}.

The equality of (\ref{formula:main_ineq_fourth_line}) holds as we are summing over $\mathbf{v}_k\in \A_n$.

To obtain (\ref{formula:main_ineq_fifth_line}), we are showing that $0\le (\text{the sum in (\ref{formula:main_ineq_fifth_line})})-(\text{the sum in (\ref{formula:main_ineq_fourth_line})})=o(1)$. Keeping in mind that 
$$\C(W_n^\tc)\setminus \A_n\subseteq \left(\C(W_n^\tc)\setminus \C(\tW_n)\right)\cup \{\mathbf{v}_k\in \C(\tW_n):\;E(T(\mathbf{v}_k))\not\subseteq \tE_n(4/\gamma_n) \text{ or}\notin \A_n\},$$
and noting that Lemma \ref{lemma_main:Reff_tree_bal} and the definition of $\tW_n$ gives us
\begin{equation}\label{formula:F_not_A_not_nice_subcomp}
    \begin{aligned}
        &\sum_{\substack{\mathbf{v}_k\in \C(\tW_n):\\E(T(\mathbf{v}_k))\not\subseteq \tE_n(4/\gamma_n)\text{ or }\notin \A_n}} \prod_{j=2}^k c(v_{p_T(j)}, v_j)\\
        &\hspace{3cm}\le k|\tW_n|(K\gamma_n)^{k-2} \left(\max_{v\in \tW_n} \sum_{\substack{u:\;u\sim v\\(u, v)\notin \tE_n(4/\gamma_n)}}c(u,v)\right) +\sum_{\mathbf{v}_k\in \CN\setminus \A_n} \prod_{j=2}^k c(v_{p_T(j)}, v_j)\\
        &\hspace{3cm}=o(1)|V(G_n)|\gamma_n^{k-1},
    \end{aligned}
\end{equation}
we can estimate the sum
\begin{equation}
    \begin{aligned}
            &\sum_{\mathbf{v}_k\in\C(W_n^\tc)\setminus \A_n}F_{\c_n}(T(\mathbf{v}_k))\prod_{j=2}^k c(v_{p_T(j)}, v_j)\\
            &\hspace{4cm}=o(1)+\sum_{\substack{\mathbf{v}_k\in \C(\tW_n):\\E(T(\mathbf{v}_k))\not\subseteq \tE_n(4/\gamma_n)\text{ or }\mathbf{v}_k\notin \A_n}}F_{\c_n}(T(\mathbf{v}_k))\prod_{j=2}^k c(v_{p_T(j)}, v_j)
    \\
    &\hspace{4cm}=o(1)+\sum_{\substack{\mathbf{v}_k\in \C(\tW_n):\\E(T(\mathbf{v}_k))\not\subseteq \tE_n(4/\gamma_n)\text{ or }\mathbf{v}_k\notin \A_n}}\frac{O(1)}{|V(G_n)|} \prod_{j=2}^k \frac{c(v_{p_T(j)}, v_j)}{\gamma_n}=o(1),
    \end{aligned}
\end{equation}
where the first estimate follows from Lemma \ref{lemma_main:F_on_W^tc} with $U_n=\tW_n$, the second from the fact that (\ref{formula:b(v)_typ}) gives us $F_{\c_n}(T(\mathbf{v}_k))=\Theta(1)|V(G_n)|^{-1}\gamma_n^{-(k-1)}$ for any $\{v_1, \ldots, v_k\}\subseteq \tW_n$ and third from (\ref{formula:F_not_A_not_nice_subcomp}).
\end{proof}

\begin{remark}[The proof of Theorem \ref{thm:local_det} also works on the 1-out model]\label{rmrk:1-edge_proof}
Here, we detail why the local statistics of the weighted spanning trees and the local statistics of the 1-out model---considered in Remark \ref{rmrk:1-edge_statement}---asymptotically agree.

We start to check it for high degree balanced sequences---the ones where Assumption~\ref{assum:for_local_lim} with $V(G_n)=W_n^\tc$ hold. The edges of any $T(\mathbf{v}_k)$ can be oriented $k-t$ ways to be compatible with the directed 1-out model, where $k=|V(T)|$, $t=|V(B_{T}(1, r-1))|$, and our convention is that the joint presence of the oriented edges $(u,v)$ and $(v, u)$ results in two parallel edges between $u$ and $v$ in the unoriented graph. We denote these compatible orientations by $\vec{E}_{t+1}, \ldots, \vec{E}_{k}$: to define $\vec{E}_i$, we choose the path $P_i$ from the root $v_1$ towards one of the boundary points $v_{i}\in \{v_{t+1}, \ldots, v_k\}$, and we direct the edges of $P_i$ away from the root and the edges from $E(T(\mathbf{v}_k))\setminus P_i$ towards the root. Here, with a little abuse of notation, we consider the ball $B_{\mathrm{1-out}(\c_n)}(X, r)$ w.r.t.~the undirected graph distance (regardless having directed edges) where $X\sim \unif(V(G_n))$, and we use (the undirected) Notation \ref{notation:v_k} for $p_T(\cdot)$. Then,
\begin{equation}\label{formula:balanced_first}
    \begin{aligned}
            \P\left(B_{\mathrm{1-out}(\c_n)}(X, r)\supseteq \vec{E}_i\right)&=\frac{1}{|V(G_n)|}\prod_{(v_j, v_{p_T(j)})\in P_i}\frac{c(v_{p_T(j)}, v_j)}{C(v_{p_T(j)})}\prod_{(v_j, v_{p_T(j)})\in E(T(\mathbf{v}_k))\setminus P_i}\frac{c(v_j, v_{p_T(j)})}{C(v_{j})}\\
    &=\frac{1}{|V(G_n)|}\frac{C(v_i)\prod_{e\in T(\mathbf{v}_k)}c(e)}{\prod_{j=1}^kC(v_j)},
    \end{aligned}
\end{equation}
and 
\begin{equation}\label{formula:balanced_second}
    \begin{aligned}
    \P\left(B_{\mathrm{1-out}(\c_n)}(X, r)=\vec{E}_i\Given B_{\mathrm{1-out}(\c_n)}(X, r)\supseteq \vec{E}_i\right)&=(1+o(1))\prod_{u\in V(G_n)}\left(1-\sum_{1\le j\le t:\; v_j\sim u}\frac{c(u, v_j)}{C_u}\right)\\
    &=(1+o(1))\exp\left(-\sum_{j=1}^t b_{\c_n}(v_j)\right),
    \end{aligned}
\end{equation}
where the first equality holds as, in the RHS of the first line, we have the probabilities of excluding at most $k^2$ more edges---inside $V(T(\mathbf{v}_k))$---than in the LHS, and $\frac{c(u, v_j)}{C_u}\le \delta_n$. The second bound holds as $1-x=\exp(-x+o(x))$ as $x\rightarrow 0$. Hence, we indeed get that 
\begin{align*}
    \P\left(B_{\mathrm{1-out}(\c_n)}(X, r)=T(\mathbf{v}_k)\right)&=\sum_{i=t+1}^k\P\left(B_{\mathrm{1-out}(\c_n)}(X, r)=\vec{E}_i\right)\\
    &=(1+o(1))F_{\c_n}(T(\mathbf{v}_k)) \prod_{j=2}^k c(v_{p_T(j)}, v_j).
\end{align*}

For the almost balanced setting, one can check that the analogous steps of the proof of Theorem \ref{thm:local_det} also work for the 1-out model---some steps are even simpler:
\begin{align}
    \P(B_{\mathrm{1-out}(\c_n)}(X, r)\cong T)&=\P(B_{\mathrm{1-out}(\c_n)}(X, r)\cong T\text{ and it is nice})+o(1)\label{formula:1-ot_first_line}\\
    &=\sum_{\mathbf{v}_k\in \CN}\frac{\P(B_{\mathrm{1-out}(\c_n)}(v_1, r)=T(\mathbf{v}_k))}{|V(G_n)||\mathrm{Stab}_T|}+o(1)\label{formula:1-ot_second_line}\\
    &=(1+o(1))\sum_{\mathbf{v}_k\in \CN}\frac{F_{\c_n}(T(\mathbf{v}_k))}{|\mathrm{Stab}_T|}\prod_{j=2}^k c(v_{p_T(j)}, v_j)+o(1)\label{formula:1-ot_third_line}\\
    &=(1+o(1))\sum_{\mathbf{v}_k\in\C(W_n^\tc)}\frac{F_{\c_n}(T(\mathbf{v}_k))}{|\mathrm{Stab}_T|}\prod_{j=2}^k c(v_{p_T(j)}, v_j)+o(1). \label{formula:1-ot_fourth_line},
\end{align}
where the estimates (\ref{formula:1-ot_second_line}) and (\ref{formula:1-ot_fourth_line}) work the same way as in the proof of Theorem \ref{thm:local_det}.

One can check (\ref{formula:1-ot_first_line}) if they prove the analogous Lemma \ref{lemma_main:prob_not_nice} for the 1-out model. This can be done from tightness the same way as for $\wst$. The tightness for 1-out model can be done similarly as Lemma \ref{lemma_main:tightness}, see Remark \ref{rmrk:1-edge_tightness}.

The bound (\ref{formula:1-ot_third_line}) is obtained by relaxing the computation of the balanced setting: for any $\mathbf{v}_k\in \C(\tW_n)$, one can show that $\P\left(B_{\mathrm{1-out}(\c_n)}(X, r)=T(\mathbf{v}_k)\right)=(1+o(1))F_{\c_n}(T(\mathbf{v}_k)) \prod_{j=2}^k c(v_{p_T(j)}, v_j)$. Here, (\ref{formula:balanced_first}) clearly holds for any $\mathbf{v}_k\in \C(W^\tc_n)$, and the only additional detail needed for (\ref{formula:balanced_second}) in case of $\mathbf{v}_k\in \C(\tW_n)$ is (\ref{formula:b(v)_typ}).
\end{remark}

\begin{proof}[Proof of Theorem \ref{thm:local_ust_almost_reg}] The idea of the proof is that one can estimate $F_{\c_n}(T(\mathbf{v}_k))$ for typical compatible $\mathbf{v}_k$s of almost regular sequences using that $b_{\c_n}(v)=(1+o(1))$ and $C_v=(1+o(1))\gamma_n$ holds for most of the vertices. Since this is the only result in the paper that specifically is about almost regular sequences, we have optimized our notation primary to the more general almost balanced scenario. As a consequence, a small additional argument is needed here for the sharper estimates under almost regularity.

For any high degree $\gamma_n$-almost regular sequence $G_n$, there exist some $\delta_n\rightarrow 0$ such that, for any fixed $0<\eps<\frac{1}{2}$, $G_n$ is almost balanced with parameters $(1-\delta_n)\gamma_n$, $K:=1+\eps$ and $\delta_n$. By formula (\ref{formula:b(v)_typ}), $b_{\c_n}(v)=(1+O(\eps)+o_{\eps}(1))$ and $C_v=(1+O(\eps)+o_{\eps}(1))\gamma_n$ hold for any $v\in \tW_n$, where the constant corresponding to $O(\eps)$ depends only on $T$ an it is uniform for $0<\eps<\frac{1}{2}$. Then, for any $ \mathbf{v}_k\in \C(\tW_n)$, we have $F_{\c_n}(T(\mathbf{v}_k))=(1+O(\eps)+o_{\eps}(1))\frac{1}{|V(G_n)|}e^{-t}\frac{(k-t)\gamma_n}{\gamma_n^{k}}$ and $\frac{C_{v_1}}{C_{V(G_n)}}\prod_{j=2}^k \frac{1}{C_{v_{p(j)}}}=(1+O(\eps)+o_{\eps}(1))\frac{1}{|V(G_n)|\gamma_n^{k-1}}$. Moreover, $C_{V(G_n)}=(1+o(1))|V(G_n)|\gamma_n)$ is satisfied by almost regularity. Consequently,
\begin{align*}
        \P(B_{\wst(\c_n)}(X, r)\cong T)&=\sum_{\mathbf{v}_k\in \C(W_n^{\tc})} \frac{F_{\c_n}(T(\mathbf{v}_k))}{|\mathrm{Stab}_T|} \prod_{j=2}^k c(v_{p_T(j)}, v_j)+o_{\eps}(1)\\
        &=\sum_{\mathbf{v}_k\in \C(\tW_n)} \frac{F_{\c_n}(T(\mathbf{v}_k))}{|\mathrm{Stab}_T|} \prod_{j=2}^k c(v_{p_T(j)}, v_j)+o_{\eps}(1)\\
        &=(1+O(\eps)+o_{\eps}(1))\frac{e^{-t}(k-t)}{|\mathrm{Stab}_T|}\sum_{\mathbf{v}_k\in \C(\tW_n)} \frac{C_{v_1}}{C_{V(G_n)}}\prod_{j=2}^k \frac{c(v_{p_T(j)}, v_j)}{C_{v_{p(j)}}}+o(1)\\
        &=(1+O(\eps)+o_{\eps}(1))\frac{e^{-t}(k-t)}{|\mathrm{Stab}_T|}+o_{\eps}(1),
\end{align*}
where the first equality follows from Theorem \ref{thm:local_det}, the second from Lemma \ref{lemma_main:F_on_W^tc} with $U_n=\tW_n$, the third from the previous paragraph, and the fourth from (\ref{formula:F_on_W^tc_second}) and Lemma \ref{lemma_main:F_on_W^tc} with $U_n=\tW_n$.

We let first $n\rightarrow \infty$ and then $\eps\rightarrow 0$ to get $\P(B_{\wst(\c_n)}(X, r)\cong T)=\frac{e^{-t}(k-t)}{|\mathrm{Stab}_T|}+o(1)$. Writing $(\Gamma, o)$ for the Poisson(1) branching process conditioned to survive forever, in Subsection~1.2 of \cite{nachmias2022local}, it is shown that $\P(B_{\Gamma}(o, r)\cong T)=\frac{e^{-t}(k-t)}{|\mathrm{Stab}_T|}$, which finishes our proof.
\end{proof}
Note an important difference compared to \cite{nachmias2022local}: in \cite{nachmias2022local}, they only show the bound $\P(B_{\wst(\c_n)}(X, r)\cong T)\le (1+o(1))\frac{e^{-t}(k-t)}{|\mathrm{Stab}_T|}$, while we express the asymptotics of $\P(B_{\wst(\c_n)}(X, r)\cong T)$. For them, the upper bound was enough as $\sum_{T\text{ is a finite tree}}\frac{e^{-t}(k-t)}{|\mathrm{Stab}_T|}=1$ holds by the locally finiteness of the Poisson(1) branching process conditioned to survive. In our case, we do not have an explicit limiting object, so we need two-sided bounds, therefore our proof of Lemma \ref{lemma_main:Reff_tree_bal} follow a different strategy compared to \cite[Lemma 5.1]{nachmias2022local}.

\subsubsection{The local limits of \texorpdfstring{$\ust(G_n)$}{} and \texorpdfstring{$\wst^{\beta_n}(G_n)$}{} agree}\label{subsubsection:wst_ust_agree}
The aim of this subsection is to prove Theorem \ref{thm:edge_overlap_ust} and Theorem \ref{thm:local_ust}.

\begin{lemma}\label{lemma:rdm_cond_balanced} We consider $G_n$, $d_n\gg 1$, $\beta_n\ll d_n$ and $\widetilde{\delta}_n\ll 1$ satisfying $\left(\max(\beta_n, 1)/d_n\right)^{1/4}\le O(\widetilde{\delta}_n)$. We denote our random electric network by $\c_{\beta_n}:=\{\exp(-\beta_nU_e)\}_{e\in E(G_n)}$ and we write $\mu_n:=\E\left[e^{-\beta_nU_f}\right]=\frac{1-e^{-\beta_n}}{\beta_n}$.

\begin{itemize}
    \item[a)] If $(G_n, \mathbf{1}|_{E{(G_n)}})$ satisfies Assumption \ref{assum:for_edge_overlap} with parameters $d_n$ and $\widetilde{\delta}_n$---meaning that it is high degree---, then Assumption \ref{assum:for_edge_overlap} holds for $\c_{\beta_n}$, $\gamma_n:=\frac{\mu_nd_n}{2}$ and $\delta_n:=2\widetilde{\delta}_n$ with probability $1-O(\widetilde{\delta}_n)$.
    \item[b)] If $(G_n, \mathbf{1}|_{E{(G_n)}})$ is high degree almost balanced in the sense the it satisfies Assumption \ref{assum:for_local_lim} with parameters $d_n$, $\widetilde{K}$ and $\widetilde{\delta}_n$, then, with probability $1-O(\widetilde{\delta}_n)$, $(G_n, \c_{\beta_n})$ is also high degree almost balanced: Assumption \ref{assum:for_local_lim} holds for $\gamma_n:=\mu_nd_n/2$, $K:=2\widetilde{K}+1$ and $\delta_n:=(2\widetilde{K}+1)\widetilde{\delta}_n$.
\end{itemize}
\end{lemma}
We prove Lemma \ref{lemma:rdm_cond_balanced} using Chebyshev's inequality in Subsection \ref{subsubsec:i.i.d.}.

\begin{lemma}\label{lemma:F_annealed-quenched} We consider $G_n$, $d_n\gg 1$, $\beta_n\ll d_n$ and $\widetilde{\delta}_n\ll 1$ satisfying $\left(\max(\beta_n, 1)/d_n\right)^{1/4}\le O(\widetilde{\delta}_n)$ and a rooted tree $(T, 1)$ of $k$ vertices. We denote our random electric network by $\c_{\beta_n}:=\{\exp(-\beta_nU_e)\}_{e\in E(G_n)}$ and we write $\mu_n:=\E\left[e^{-\beta_nU_f}\right]=\frac{1-e^{-\beta_n}}{\beta_n}$. 

If $(G_n, \mathbf{1}|_{E{(G_n)}})$ is high degree almost balanced with parameters $\widetilde{K}$, $d_n$ and $\widetilde{\delta}_n$, then, denoting $W_n^\tc(G_n):=W_n^\tc(G_n, \mathbf{1}, d_n, \widetilde{K})$ and $W_n^\tc(\c_{\beta_n}):=W_n^\tc(G_n, \c_{\beta_n}, \frac{\mu_nd_n}{2}, 2\widetilde{K}+1)$, we have as $n\rightarrow\infty$ that 
    \begin{align}\label{formula:F_annealed}
        \left|\E\left[\sum_{\mathbf{v}_k\in \C\left(W^\tc_n(\c_{\beta_n})\right)}F_{\c_{\beta_n}}(T(\mathbf{v}_k))\prod_{j=2}^k c_{\beta_n}(v_{p_T(j)}, v_j)\right]-\sum_{\mathbf{v}_k\in \C(W^\tc_n(G_n))}F_{G_n}(T(\mathbf{v}_k))\right|\le o(1).
    \end{align}
Moreover, there exist some $\eta_n=o(1)$, such that, with probability $1-o(1)$,
\begin{align}\label{formula:F_quenched}
        \left|\sum_{\mathbf{v}_k\in \C\left(W^\tc_n(\c_{\beta_n})\right)}F_{\c_{\beta_n}}(T(\mathbf{v}_k))\prod_{j=2}^k c_{\beta_n}(v_{p_T(j)}, v_j)-\sum_{\mathbf{v}_k\in \C(W^\tc_n(G_n))}F_{G_n}(T(\mathbf{v}_k))\right| \le \eta_n.
    \end{align}
\end{lemma}
We prove Lemma \ref{lemma:F_annealed-quenched} in Subsection \ref{subsubsec:i.i.d.}. The proof is based on that, with probability tending to 1, $C_{\c_{\beta_n}}(v):=\sum_{u:\;(u,v)\in E(G_n)}c_{\beta_n}(u, v)$ and $b_{\c_{\beta_n}}(v)$ are concentrated for most of the vertices $v$.

\begin{proof}[Proof of Theorem \ref{thm:edge_overlap_ust}] One can choose some $\widetilde{\delta}_n\ll 1$ such that $\left(\max(\beta_n, 1)/d_n\right)^{1/4}\le O(\widetilde{\delta}_n)$ holds and $(G_n, \mathbf{1}_{|E(G_n)})$ satisfies Assumption \ref{assum:for_edge_overlap} with parameters $d_n$ and $\widetilde{\delta}_n$.

For $\c_{G_n, \beta_n}=\{\exp(-\beta_nU_e)\}_{e\in G_n}$, denote by $A_n$ the event that the sequence satisfies Assumption \ref{assum:for_edge_overlap} with parameters $\mu_nd_n/2$ and $\delta_n:=2\widetilde{\delta}_n$. Then $\P(A_n)=1-O(\delta_n)$ by Lemma $\ref{lemma:rdm_cond_balanced}$. By Theorem \ref{thm:edge_overlap_det} on $A_n$ and the trivial estimate $|\mathcal{T}\cap \mathcal{T}'|\le |V(G_n)|-1$ on $A_n^c$, we obtained that
$$\mathcal{O}(G_n, \beta_n)=\E\big[\mathds{1}[A_n]\mathcal{O}(\c_{G_n, \beta_n})\big]+O(\delta_n)(|V(G_n)|-1)=O(\delta_n)|V(G_n)|.$$    
\end{proof}

\begin{proof}[Proof of Theorem \ref{thm:local_ust}] We choose some $\delta_n\ll 1$ such that $\left(\max(\beta_n, 1)/d_n\right)^{1/4}\le O(\delta_n)$ holds and $(G_n, \mathbf{1}_{|E(G_n)})$ satisfies Assumption \ref{assum:for_local_lim} with parameters $d_n$, $K$ and $\delta_n$. Denoting by $A_n$ the event that $\c_{G_n, \beta_n}=\{\exp(-\beta_nU_e)\}_{e\in G_n}$ satisfies both Assumption \ref{assum:for_local_lim} with parameters $\gamma_n:=\mu_nd_n/2$, $2K+1$ and $(2K+1)\delta_n$ and (\ref{formula:F_quenched}), we have that $\P(A_n)=1-o(1)$.

Combining the previous results, on $A_n$, we have that 
    \begin{equation}
\begin{aligned}
    \P(B_{\wst(\c_{\beta_n})}(X, r)\cong T\given \mathbf{U})&=\sum_{\mathbf{v}_k\in \C(W_n^\tc(\c_{\beta_n}))}\frac{F_{\c_{\beta_n}}(T(\mathbf{v}_k))}{|\mathrm{Stab}_T|}\prod_{j=2}^k c_{\beta_n}(v_{p_T(j)}, v_j)+o(1)\\
    &=\sum_{\mathbf{v}_k\in \C(W_n^\tc(G_n))}\frac{F_{G_n}(T(\mathbf{v}_k))}{|\mathrm{Stab}_T|}+o(1)\\
    &=\P(B_{\ust(G_n)}(X, r)\cong T)+o(1),
\end{aligned}
\end{equation}
where the the second equality follows from Lemma \ref{lemma:F_annealed-quenched} and the first and the third ones from Theorem \ref{thm:local_det}: in the first equality, we use that the $o(1)$ error term of Theorem \ref{thm:local_det} can be chosen uniformly over all electric networks satisfying Assumption \ref{assum:for_local_lim} given the parameters $\gamma_n:=\mu_nd_n/2$, $2K+1$ and $(2K+1)\delta_n$, otherwise, there would be an almost balanced sequence contradicting with Theorem \ref{thm:local_det}.

Then, the agreement of the local limits is automatic:
\begin{align*}
    &\left|\P(B_{\wst(\c_{\beta_n})}(X, r)\cong T)-\P(B_{\ust(G_n)}(X, r)\cong T)\right|\\
    &\hspace{3cm}\le\E\big[\mathds{1}[A_n]|\P(B_{\wst(\c_{\beta_n})}(X, r)\cong T\given \mathbf{U})-\P(B_{\ust(G_n)}(X, r)\cong T)|\big]+\P(A_n^c)\\
    &\hspace{3cm}=o(1).
\end{align*}
\end{proof}

\subsection{Proofs of the lemmata} \label{subsec:proofs_of_lemmata}
\subsubsection{The effective resistance between the endpoints of a typical edge}
We prove Lemma \ref{lemma_main:Reff_edge} in this subsection.
\label{subsubsec:typ_edge}
\begin{proof}[Proof of Lemma \ref{lemma_main:Reff_edge}]
    To show a), we consider the network obtained by adding a vertex $a$ and replacing $(u, v)$ by $(u, a), (a, v)$ with $c(u,a)=c(a, v)=2c(u,v)$. We did not change $\Reff(u \leftrightarrow v)$, and we will apply Nash-Williams inequality with two sets of edges---the edges coming from $u$ and the edges coming from $v$---and we obtain that 
	\begin{align*}
	    \Reff(u \leftrightarrow v) &\ge \frac{1}{2c(u, v)+\sum_{w \notin \{u, v\}} c(u, w) }+\frac{1}{2c(u, v)+\sum_{w \notin \{u, v\}} c(v, w)}\\
     &\ge \frac{1}{C_u+\delta_n\gamma_n}+\frac{1}{C_v+\delta_n\gamma_n}\ge \frac{\mathds{1}_{u\in W_n}}{(1+\delta_n)C_u}+\frac{\mathds{1}_{v\in W_n}}{(1+\delta_n)C_v}\\
     &\ge\frac{(1-\delta_n)\mathds{1}_{u\in W_n}}{C_u}+\frac{(1-\delta_n)\mathds{1}_{v\in W_n}}{C_v}.
	\end{align*}
It is clear from the first line of the previous computation that $\Reff(u \leftrightarrow v)\ge \frac{1}{2C_u}+\frac{1}{2C_v}$.
    
    For part b), from Foster's formula $\sum_{(u,v)\in E(G_n)} c(u,v)\Reff(u \leftrightarrow v)=|V(G_n)|-1\le |V(G_n)|$ we subtract 
    \begin{equation}\label{formula:sum_f}
    \begin{aligned}
         &\sum_{(u,v)\in E(G_n)}c(u,v)(1-\delta_n)\left(\frac{\mathds{1}_{u\in W_n}}{C_u}+\frac{\mathds{1}_{v\in W_n}}{C_v}\right)\ge  \sum_{v\in W_n} \frac{(1-\delta_n)C_v}{C_v}\ge(1-2\delta_n)|V(G_n)|
    \end{aligned}
    \end{equation}
    to get the statement.
Now, we check part c). 

If $(u, v)\in E(G_n)\setminus \tE_n(4/\gamma_n)$, then $\Reff(u\leftrightarrow v)\ge \frac{4}{\gamma_n} \ge 2\frac{\mathds{1}[u\in W_n]}{C_u}+2\frac{\mathds{1}[v\in W_n]}{C_v}$ which is equivalent to $\Reff(u\leftrightarrow v)\le 2\Reff(u\leftrightarrow v)-2\frac{\mathds{1}[u\in W_n]}{C_u}-2\frac{\mathds{1}[v\in W_n]}{C_v}$, so part b) implies that
\begin{align*}
    \sum_{(u, v)\notin \tE_n(4/\gamma_n)}c(u, v)\Reff(u\leftrightarrow v)&\le 
    \sum_{(u, v)\notin \tE_n(4/\gamma_n)}2c(u, v)\left(\Reff(u\leftrightarrow v)-\frac{\mathds{1}[u\in W_n]}{C_u}-\frac{\mathds{1}[v\in W_n]}{C_v}\right)\\
    &\le 4\delta_n|V(G_n)|.
\end{align*}
\end{proof}

Note that Lemma \ref{lemma_main:Reff_edge} gives a result about the effective resistance between the endpoints of a randomly chosen edge:
\begin{corollary}[Analogue of Lemma 3.1 of \cite{nachmias2022local}]\label{cor:Reff_edge}
If a random edge $(X_1, X_2)$ is chosen with probability $\propto c(X_1, X_2)$, then
    \begin{gather*}\P\left(\Reff(X_1\leftrightarrow X_2)\ge \frac{1}{C_{X_1}}+\frac{1}{C_{X_2}}+\eps\right)\le \frac{O(\delta_n)}{\eps}\frac{1}{\gamma_n}.
    \end{gather*}
\end{corollary}

\begin{proof}
We write $p_{\eps}:=\P\left(\Reff(X_1\leftrightarrow X_2)\ge\frac{1}{C_{X_1}}+\frac{1}{C_{X_2}}+\eps\right)$. Diving Foster's formula and (\ref{formula:sum_f}) by $C_{E(G_n)}:=\sum_{e\in E(G_n)}c(e)\ge(1-\delta_n)\frac{1}{2}|V(G_n)|\gamma_n$
\begin{align}\label{formula:bound_of_cor}
		\frac{|V(G_n)|}{C_{E(G_n)}}&\ge\frac{\sum_{(u,v)\in E(G_n)}c(u,v)\Reff(u \leftrightarrow v)}{C_{E(G_n)}}=\E [\Reff(X_1 \leftrightarrow X_2)]\ge \frac{(1-O(\delta_n))|V(G_n)|}{C_{E(G_n)}}+\eps p_{\eps},
	\end{align}
giving us
$$p_{\eps}\le \frac{1}{\eps}\frac{O(\delta_n)|V(G_n)|}{C_{E(G_n)}}=\frac{O(\delta_n)}{\eps}\frac{1}{\gamma_n}.$$
\end{proof}

\subsubsection{Tightness and nice \texorpdfstring{$r$}{}-balls}
\label{subsubsec:tightness_nice}
This subsection contains the proofs of Lemma \ref{lemma_main:tightness} and Lemma \ref{lemma_main:prob_not_nice}.
\begin{lemma}[Theorem 3.1 of \cite{benjamini2015unimodular}]\label{lemma:unifint_implies_tightness} Let $\{H_n\}$ be a sequence of finite graphs, possibly random, and $V_n$ an
independent uniformly drawn vertex of $H_n$. Assume that $\{\deg(V_n)\}_{n\ge 1}$ are uniformly integrable random variables. Then for any integer $r\ge0$, we have
$$\lim_{M\rightarrow\infty}\sup_{n}\P\left(|B_{H_n}(V_n, r)| \ge M\right)=0.$$
\end{lemma}
This lemma is based on mass transport principle.

\begin{proof}[Proof of Lemma \ref{lemma_main:tightness}]
For $X\sim \unif(V(G_n))$, decomposing $\deg_{\wst(\c_n)}(X)=\chi_{n, 1}+\chi_{n, 2}+\chi_{n, 3}$ with
\begin{align*}
    \chi_{n, 1}&:=\mathds{1}\left[X\in W_n^\tc\right]\left|\left\{u\in V(G_n): \;(u,X)\in \wst(\c_n)\text{ and }\Reff(X\leftrightarrow u)\le  \frac{4}{\gamma_n}\right\}\right|,\\
    \chi_{n, 2}&:=\mathds{1}\left[X\in W_n^\tc\right]\left|\left\{u\in V(G_n): \;(u,X)\in \wst(\c_n)\text{ and }\Reff(X\leftrightarrow u)>\frac{4}{\gamma_n}\right\}\right|\text{ and }\\
    \chi_{n, 3}&:=\mathds{1}\left[X\in V(G_n)\setminus W_n^\tc\right]\deg_{\wst(\c_n)}(X),
\end{align*}
our aim is to check that $\{\chi_{n, 1}\}_{n\ge 1}$, $\{\chi_{n, 2}\}_{n\ge 1}$ and $\{\chi_{n, 3}\}_{n\ge 1}$ are uniformly integrable. 

The uniform integrability of $\{\chi_{n, 2}\}_n$ is easy, since by Lemma \ref{lemma_main:Reff_edge} c), we have 
$$\E\left[\chi_{n, 2}\right]\le\frac{2}{|V(G_n)|}\sum_{(u,v)\in E(G_n)\setminus \tE_n(4/\gamma_n)}c(u,v)\Reff(u\leftrightarrow v)=\frac{2}{|V(G_n)|}O(\delta_n)|V(G_n)|=O(\delta_n)\rightarrow 0.$$

The uniform integrability of $\{\chi_{n, 3}\}_n$ follows similarly: by Lemma \ref{lemma_main:Reff_edge} and Assumption~\ref{assum:for_local_lim},
\begin{align*}
    \E\left[\chi_{n, 3}\right]&\le \frac{1}{|V(G_n)|}\sum_{\substack{x\notin W_n^\tc, u\in V(G_n):\\(x, u)\in \tE_n(4/\gamma_n)}} c(x, u)\frac{4}{\gamma_n}+\frac{1}{|V(G_n)|}\sum_{\substack{x\notin W_n^\tc, u\in V(G_n):\\(x, u)\in E(G_n)\setminus \tE_n(4/\gamma_n)}} c(x, u)\Reff(x\leftrightarrow u)\\
    &\le \frac{4}{|V(G_n)|\gamma_n}\sum_{x\notin W_n^ \tc} C_x+O(\delta_n)\le \frac{4}{|V(G_n)|\gamma_n}\delta_n|V(G_n)|\gamma_n+O(\delta_n)=O(\delta_n)\rightarrow 0.
\end{align*}

For the uniform integrability of $\{\chi_{n, 1}\}_n$, note that the negative association (\ref{formula:neg_assoc}) of the edges of the $\wst$ gives us
\begin{align*}
    \P(\chi_{n, 1}\ge k)&\le \frac{1}{|V(G_n)|}\sum_{x\in W_n^\tc}\frac{1}{k!}\sum_{\substack{(y_1, \ldots, y_k):\;y_i\ne y_j \;\forall i\ne j,\\
    (X,y_i)\in \tE_n(4/\gamma_n) \;\forall i\le k}}\P((x, y_i)\in \wst(\c_n)\;\forall i\le k)\\
    &\le \frac{1}{|V(G_n)|k!}\sum_{x\in W_n^\tc}\left(\sum_{y:\;(x,y)\in \tE_n(4/\gamma_n)}c(x, y)\Reff(x\leftrightarrow y)\right)^k
\le \frac{1}{k!}(K\gamma_n)^k \left(\frac{4}{\gamma_n}\right)^k
\end{align*}
which is $
(O(1/k))^k$ as $k\rightarrow \infty$. Then, we have $\sup_n\E[\chi_{n, 1}^2]<\infty$, so Hölder's inequality implies the uniform integrability. 

Then, Lemma \ref{lemma:unifint_implies_tightness} gives us the tightness of the $r$-balls for any $r\ge 0$.
\end{proof}

\begin{proof}[Proof of Lemma \ref{lemma_main:prob_not_nice}] Note that
\begin{align*}
    \sqrt{\delta_n}|V(G_n)\setminus W_n^\tn|\le &2\sum_{(u, v)\in E(G_n)\setminus\tE_n(4/\gamma_n)}c(u,v)\Reff(u\leftrightarrow v)+\sum_{v\in V(G_n)}\sum_{u\notin W_n^\tc:~(u, v)\in E(G_n)} \frac{c(v, u)}{\gamma_n}\\
    &\le 8\delta_n|V(G_n)|+\sum_{u\notin W_n^\tc} \frac{C_u}{\gamma_n}\le 8\delta_n|V(G_n)|+ \frac{\delta_n\gamma_n|V(G_n)|}{\gamma_n}\le 9\delta_n|V(G_n)|
\end{align*}
where the first inequality holds by Definition \ref{def:typ_neigh}, the second by Lemma \ref{lemma_main:Reff_edge} and the third by Assumption \ref{assum:for_local_lim}. Therefore,
we have that $|\tW_n^c|\le |V(G_n)\setminus W_n^{\tc}|+|V(G_n)\setminus W_n^{\tn}|\le 10\sqrt{\delta_n}|V(G_n)|$.

By tightness shown in Lemma \ref{lemma_main:tightness}, for any $\eps>0$ there exists some $M>0$ such that 
    $$\sup_n\P\left(|B_{\wst(\c_n)}(X, 2r))|>M
    \right)\le \eps.$$
    For any $v\in V(G_n)$ and any $(u,v)\in E(G_n)$, we have that
    \begin{align*}
        \P\left(v\in V(B_{\wst(\c_n)}(X, r)), \; |B_{\wst(\c_n)}(X, 2r))|\le M\right)&\le \frac{M}{|V(G_n)|},\\
        \P\left((u,v)\in E(B_{\wst(\c_n)}(X, r)),\; |B_{\wst(\c_n)}(X, 2r))|\le M\right)&\le \frac{2M}{|V(G_n)|}c(u, v)\Reff(u\leftrightarrow v),
    \end{align*}
    e.g., the reasoning behind the second inequality: $(u,v)\in E(B_{\wst(\c_n)}(X, r))$ can happen if $(u,v)\in E(\wst(\c_n))$ and $\{u, v\}\cap B_{\wst(\c_n)}(X, r-1)\ne \emptyset$. For any $u$, there are at most $M$ choices of $X$ satisfying $|B_{\wst(\c_n)}(X, 2r))|\le M$ since this condition automatically implies $|B_{\wst(\c_n)}(u, r))|\le M$.
    
    Therefore, by $|\tW_n|=(1-O(\delta_n^{1/2}))|V(G_n)|$, Lemma \ref{lemma_main:Reff_edge} and the union bound,
    \begin{align*}
    &\P\left(B_{\wst(\c_n)}(X, r)\textrm{ is not nice and }|B_{\wst(\c_n)}(X, 2r))|\le M\right)\\
    &\hspace{3.5cm}\le \frac{M}{|V(G_n)|}|V(G_n)\setminus\tW_n|+\frac{2M}{|V(G_n)|}\sum_{(u, v)\in E(G_n)\setminus \tE_n(4/\gamma_n)}c(u, v)\Reff(u\leftrightarrow v)\\
    &\hspace{3.5cm}\le\frac{M}{|V(G_n)|}O(\sqrt{\delta_n})|V(G_n)|+\frac{2M}{|V(G_n)|}O(\delta_n)|V(G_n)|=O(\sqrt{\delta_n}).
    \end{align*}
    Combining these, we obtain $\P(B_{\wst(\c_n)}(X, r)\textrm{ is not nice})\le \eps+O(\sqrt{\delta_n})$. Letting $\eps\rightarrow 0$ finishes the proof.
\end{proof}

\begin{remark}[Proof of tightness for the 1-edge out model]\label{rmrk:1-edge_tightness} The proof of Lemma \ref{lemma_main:prob_not_nice} for the 1-edge out model is the same if we know tightness which can be done similarly as Lemma \ref{lemma_main:tightness}.

For the tightness of the 1-edge out model, we consider $\chi_{n, 1}$, $\chi_{n, 2}$ and $\chi_{n, 3}$ from the proof of Lemma \ref{lemma_main:tightness}, and 
\begin{align*}
    \chi'_{n, 1}&:=\mathds{1}\left[X\in W_n^\tc\right]\left|\left\{u\in V(G_n): \;(u,X)\in \text{1-out}(\c_n)\text{ and }\Reff(X\leftrightarrow u)\le  \frac{4}{\gamma_n}\right\}\right|,\\
    \chi'_{n, 2}&:=\mathds{1}\left[X\in W_n^\tc\right]\left|\left\{u\in V(G_n): \;(u,X)\in \text{1-out}(\c_n)\text{ and }\Reff(X\leftrightarrow u)>\frac{4}{\gamma_n}\right\}\right|\text{ and }\\
    \chi'_{n, 3}&:=\mathds{1}\left[X\in V(G_n)\setminus W_n^\tc\right]\deg_{\text{1-out}(\c_n)}(X).
\end{align*}
Here, if one would like to show the tightness of the 1-edge out without using the properties of $\{\chi_{n, i}\}_{i=1, 2, 3}$, then the natural way to decompose the degree into $\{\chi'_{n, i}\}_{i=1, 2, 3}$ would not contain effective resistances.

Note that the expected number of of edges between $u$ and $v$ in the 1-edge out model can be bounded as $\frac{c(u,v)}{C_u}+\frac{c(u,v)}{C_v}\le 2c(u,v)\Reff(u\leftrightarrow v)$ by Lemma \ref{lemma_main:Reff_edge}~a), therefore
we automatically have $\E[\chi'_{n, 2}]\le 2\E[\chi_{n, 2}]\rightarrow 0$ and $\E[\chi'_{n, 3}]\le 2\E[\chi_{n, 3}]\rightarrow 0$. Moreover, the tail of $\chi_{n, 1}'$ can be estimated by noticing that, for any $x$, exactly one edge comes out from $x$, and the number of edges coming in $x$ is a sum of independent Bernoulli random variables, so one can check that \begin{align*}
\P(\chi_{n, 1}'\ge k+1)\le \frac{1}{|V(G_n)|k!}\sum_{x\in W_n^\tc}\left(\sum_{y:\;(x,y)\in \tE_n(4/\gamma_n)}\frac{c(x, y)}{C_y}\right)^k\le \frac{1}{k!}(K\gamma_n)^k \left(\frac{8}{\gamma_n}\right)^k,
\end{align*}
implying uniform integrability the same way as in the proof of Lemma \ref{lemma_main:tightness}.
\end{remark}

\subsubsection{Determining explicit probabilities}\label{subsubsec:explicit}
In this subsection, we prove Lemmata \ref{lemma_main:F_on_W^tc} and \ref{lemma_main:Reff_tree_bal}.
\begin{lemma}[Analogue of Corollary 3.3 of \cite{nachmias2022local}]\label{lemma:Reff_walk_regular} 
Suppose that $(G, \c)$ satisfies Assumption \ref{assum:for_local_lim}. Let $(X_1, \ldots, X_k)$ be the first $k$-steps of the random walk w.r.t. $\c$ started from a stationary vertex $X_1$. Then
\begin{gather*}\P\left(\Reff(X_1\leftrightarrow X_k)\ge \frac{1}{C_{X_1}}+\frac{1}{C_{X_k}}+\eps\right)\le \frac{O(\delta_n)}{\eps}\frac{1}{\gamma_n}.
    \end{gather*}
\end{lemma}
\begin{proof} We prove this lemma likewise Corollary \ref{cor:Reff_edge}. 

Note that, $X_j$ is stationary for any $1\le j\le k$, therefore, by (\ref{formula:unif_vs_stationary}), we have
\begin{equation}\label{formula:prob_notdistinct}
    \begin{aligned}
                &\P(\exists i\in \{1, \ldots, k\}: \; X_i\notin W_n^\tc)\le k\pi(V(G_n)\setminus W_n^\tc)=O(\delta_n)\text{ and }\\
        &\P(\exists i\ne j: \; X_i=X_j)\le \sum_{i=1}^k\P(X_i\notin W_n^\tc)+k^2\frac{\delta_n\gamma_n}{\gamma_n}= O(\delta_n),
    \end{aligned}
\end{equation}
which replaces the use of Lemma \ref{lemma_main:Reff_tree_bal} in the lower bound of (\ref{formula:bound_of_cor}), making us possible to generalize it for any $k\ge 2$.

In order to the generalize of the upper bound of the formula (\ref{formula:bound_of_cor}), we have to check the estimate $\E [\Reff(X_1 \leftrightarrow X_k)]\le \frac{(1+O(\delta_n))|V(G_n)|}{C_{E(G_n)}}$, because, from this, the proof of the lemma follows likewise Corollary \ref{cor:Reff_edge} is obtained from $\E [\Reff(X_1 \leftrightarrow X_2)]=\frac{|V(G_n)|-1}{C_{E(G_n)}}$. Following the proof in \cite{nachmias2022local}, we estimate the expected commute time in order to estimate the effective resistances. We would like to use the commute time identity, hence we are first checking that $\E[\E_{X_k}[\tau_{x_1}]]\le\frac{1+\mathrm{error}}{\pi_n(x_1)}$ for any $x_1\in V(G_n)$ where $\pi_n$ is the stationary measure and the error is $o(1)$ for most of $x_1$s. Note that, for any $x_1\in V(G_n)$,
\begin{align*}
    \frac{1}{\pi_n(x_1)}= \E_{x_1}\left[\tau_{x_1}^{+}\right]\ge \E_{x_1}\left[(\tau_{x_1}^{+}-(k-1))\mathds{1}_{\tau_{x_1}^{+}>k-1}\right]\ge \sum_{\substack{(x_1, \ldots, x_k)\\x_1\ne x_i \;\forall 2\le i\le k}}\prod_{i=2}^kp_{n}(x_{i-1}, x_i) \E_{x_k} \left[\tau_{x_1}\right].
\end{align*}
where we write $p_{n}(\cdot, \cdot)$ for the transition probabilities of the walk on $(G_n, \c_n)$.
Denoting by $i$ the last time in the walk such that $X_i=x_1$, we indeed obtain
\begin{align*}
    \E \left[\E_{X_k}\left[\tau_{x_1}\right]\right]&=\E \left[\E_{X_k} \left[\tau_{x_1}\right]\mathds{1}_{X_j\ne x_1 \;\forall 2\le j\le k}\right]  +\sum_{i=2}^{k}\E\left[\E_{X_k}\left[\tau_{x_1}\right]\mathds{1}_{X_i=x_1,\; X_j\ne x_1 \;\forall j\ge i+1}\right]\\
    &\le  \frac{1}{\pi_n(x_1)}+\frac{1}{\pi_n(x_1)}\sum_{i=2}^{k}  \P(X_i=x_1\given X_1=x_1).
\end{align*}
Starting the walk from the stationary distributed $X_1$, all the $X_1, \ldots, X_k$ have stationary distribution, so
\begin{align*}
    \E\left[\E _{X_k} [\tau_{X_1}]\right]& \le \sum_{x_1\in V(G_n)}\pi(x_1)\left(\frac{1}{\pi_n(x_1)}+\frac{1}{\pi_n(x_1)}\sum_{i=2}^{k}  \P(X_i=x_1\given X_1=x_1)\right),
\end{align*}
where $\frac{1}{|V(G_n)|}\sum_{x_1\in V(G_n)}  \P(X_i=x_1\given X_1=x_1)\le 2K\P(X_i=X_1)+2\delta_n$ by (\ref{formula:unif_vs_stationary}). Note that, by Assumption \ref{assum:for_local_lim} and (\ref{formula:unif_vs_stationary}), we have
\begin{align*} \sum_{i=2}^{k}\P(X_i=X_1)&\le \sum_{i=2}^{k}\left(\P(X_i=X_1,\;X_{i-1}\in W_n^\tc)+\P(X_{i-1}\notin W_n^\tc)\right)\\
&\le (k-1)\frac{\delta_n\gamma_n}{\gamma_n}+(k-1)\pi_n(V(G_n)\setminus W_n^\tc)\\
&=O(\delta_n),
\end{align*}
giving us $\E\left[\E_{X_k} [\tau_{X_1}]\right]\le (1+O(\delta_n))|V(G_n)|$.

We use the commute time identity \cite[Corollary 2.21]{PTN} and reversibility to obtain
\begin{align*}
    \E\left[\Reff(X_1\leftrightarrow X_k)\right]=\frac{\E\left[\E_{X_1}[\tau_{X_k}]+\E_{X_k}[\tau_{X_1}]\right]}{{2C_{E(G_n)}}}\le \frac{(1+O(\delta_n))2|V(G_n)|}{2C_{E(G_n)}}.
\end{align*}
\end{proof}

\begin{lemma}[Generalization of Lemma 3.5 of \cite{nachmias2022local}]\label{lemma:from_twovert_to_k+1} Let $(G, \c)$ be an electric network and $\{1, \ldots, k\}$ are distinct vertices of it with $k\ge 3$, and we fix $K>1$. Then, there exist some constants $f:=f(k, K)$ and $g:=g(K, k)$ such that if for any $\gamma>0$, $s_1, \ldots, s_k\in [\gamma, K\gamma]$ and $0<\eps\le \frac{1}{f\gamma}$ we have
\begin{gather}
    \left|\Reff ( i \leftrightarrow j ) -\frac{1}{s_i}-\frac{1}{s_j}\right| \le \eps
\end{gather}
for any distinct $i\ne j\in \{1, \ldots, k\}$, then  
\begin{gather}
    \left | \Reff\left (k \leftrightarrow \{1,\ldots, k-1\}\right) - \frac{1}{s_k}-\frac{1}{\sum_{j=1}^{k-1}s_j}\right| \le g\eps.
\end{gather}   
\end{lemma}
\begin{proof} We do the same steps as in Lemma 3.5 of \cite{nachmias2022local}, but for more general values.

We write $\mathcal{M}$ for the original random walk on $(V(G), \c)$, and we consider two Markov chains $\mathcal{M}_{\mathrm{ret}}$ and $\widetilde{\mathcal{M}}$ on $\{1, \ldots, k\}$ with transition probabilities $p_{\text{ret}}(i, j)$ and $\widetilde{p}(i, j)$ the following way: we define $p_{\text{ret}}(i,j)$ as the probability that the random walk $\mathcal{M}$ started at $i$ returns to $j$ when it first returns to $\{1, \ldots, k\}$ and $\widetilde{M}$ is the Markov chain corresponding to the electric network obtained from $\mathcal{M}_{\text{ret}}$ by removing the loops. 

Note that the reversible measure $\{C_v\}_{v}$ of $\mathcal{M}$ restricted to $\{1, \ldots, k\}$ is a reversible measure of $\mathcal{M}_{\text{ret}}$, since $C_ip_{\text{ret}}(i,j)=C_jp_{\text{ret}}(j,i)$ is easily obtained by the reversibility of the original chain. Hence, we write $c_{\text{ret}}(i,j):=p_{\text{ret}}(i,j)C_i$ for the conductances of this network. By the definition of $\widetilde{M}$, we have $\widetilde{c}(i,j)=c_{\text{ret}}(i,j)$ if $i\ne j$ and $\widetilde{c}(i,i)=0$, and $\widetilde{C}_i=\sum_{j=1}^k\widetilde{c}(i, j)=\sum_{j\ne i}c_{\text{ret}}(i, j)$ is a reversible measure for $\widetilde{M}$.

It is easy to check even directly from the definition (\ref{formula:def_Reff}) that the pairwise effective resistances remain the same in the electric networks $\mathcal{M}$, $\mathcal{M}_{\mathrm{ret}}$ and $\widetilde{\mathcal{M}}$. Moreover, we have that $\P_k(\tau_{\{1, \ldots, k-1\}}<\tau^+_k)=(C_k-c_{\text{ret}}(k, k))/C_k=\widetilde{C}_k/C_k$, so we can write
\begin{align*}
    \widetilde{\mathcal{R}}_{\mathrm{eff}}(k\leftrightarrow \{1, \ldots, k-1\})&=\Reff(k\leftrightarrow \{1, \ldots, k-1\})=\frac{1}{C_k\P_k(\tau_{\{1, \ldots, k-1\}}<\tau^+_k)}=\frac{C_k}{C_k\widetilde{C}_k}=\frac{1}{\widetilde{C}_k}.
\end{align*}

We write $\Delta(\cdot, \cdot)$ for the normalized network Laplacian of $\widetilde{M}$, i.e., $\Delta(i, i) = \widetilde{C}_i$ and $\Delta(i, j) = -\widetilde{c}(i,j)$ 
and $\Delta[a]$ for deletion the rows and column corresponding to $a$. We consider $g_a(i, j):=G_a(i, j)/\widetilde{C}(j)$, where $G_a(i, j)$ denotes the expected number of visits of a random walk started at $i$ to $j$ before it hits $a$. By Exercise 2.62(a) of \cite{PTN}, we have that $\Delta[a]=g_a^{-1}$. It is also useful to define the $(k-1)\times (k-1)$ matrices $A$ and $B$ as
\begin{align*}
A(i,j):=\begin{cases}
        \frac{1}{s_1}\text{ if }i\ne j \in\{2, \ldots ,k\} \\
        \frac{1}{s_1}+\frac{1}{s_j}\text{ if }i=j\in\{2, \ldots ,k\},
    \end{cases}
\end{align*}
and 
\begin{align*}
    B(i,j):=\begin{cases}
        -\frac{s_{i+1}s_{j+1}}{S}\text{ if }i\ne j \\
        s_{i+1}\frac{S-s_{i+1}}{S}\text{ if }i=j,
    \end{cases}
\end{align*}
where $S:=\sum_{\ell=1}^ks_\ell$. The aim of this proof is to check that $A^{-1}=B$ and $A$ and $g_1$ are close, hence $g^{-1}_1\approx B$ are also close which will give us
$$\Reff(k\leftrightarrow \{1, \ldots, k-1\})=\frac{1}{\widetilde{C}_k}=\frac{1}{\Delta[1](k-1, k-1)}=\frac{1}{g^{-1}_1(k-1, k-1)}$$
is close to $1/B(k-1, k-1)=\frac{S}{(S-s_{k})s_{k}}=\frac{1}{s_{k}}+\frac{1}{s_1+\ldots+s_{k-1}}$.

By Exercise 2.68 of \cite{PTN}, we have
\begin{align*}
    g_1(i,j)=\frac{\Reff(1\leftrightarrow j)+\Reff(1\leftrightarrow i)-\Reff(i\leftrightarrow j)}{2},
\end{align*}
implying $|g_1(i,j)-A(i,j)|\le \frac{3}{2}\eps$ any $i, j\in \{2, \ldots, k\}$.

It is an easy linear algebra exercise to check that $AB$ is indeed the identity matrix. Instead of doing this, we give a probabilistic reasoning. We write $\mathcal{N}_{\textrm{ret}}$ and $\widetilde{\mathcal{N}}$ for the Markov chains on $\{1, \ldots, k\}$, corresponding to the electric networks of conductances $c_{\mathcal{N}_{\mathrm{ret}}}(i, j)=c_{\widetilde{\mathcal{N}}}(i, j):=\frac{s_{i}s_{j}}{S}$ for $i\ne j$, $c_{\mathcal{N}_{\mathrm{ret}}}(i, i):=\frac{s_{i}^2}{S}$ and $c_{\widetilde{N}}(i, i):=0$ where $S=\sum_{i=1}^ks_i$ as before. Then, it is clear that $B=\Delta_{\mathcal{\widetilde{N}}}[1]$. On the other hand, since the transition probabilities of $\mathcal{N}_{\mathrm{ret}}$ are  $p_{\mathcal{N}_{\mathrm{ret}}}(i, j)=\frac{s_{i}s_{j}}{S}\frac{S}{s_iS}=\frac{s_j}{S}$, the Green function of $\mathcal{N}_{\mathrm{ret}}$ is simply $G_{1, \mathcal{N}_{\mathrm{ret}}}(i, j)=\frac{s_j}{s_1}$ for $i\ne j$ and $G_{1, \mathcal{N}_{\mathrm{ret}}}(i, i)=\frac{s_i}{s_1}+1$. It is also clear that $G_{1, \mathcal{N}_{\mathrm{ret}}}(i, j)=\frac{1}{1-s_j/S}G_{1, \widetilde{\mathcal{N}}}(i, j)$ for any $i,j\in \{1, \ldots, k\}$, and then $$g_{1, \widetilde{\mathcal{N}}}(i, j)=\frac{G_{1, \mathcal{N}_{\mathrm{ret}}}(i, j)\left(1-\frac{s_j}{S}\right)}{\sum_{\ell:\; \ell \ne j}\frac{s_{\ell}s_{j}}{S}}=\frac{\left(\frac{s_j}{s_1}+\mathds{1}[i=j]\right)\left(1-\frac{s_j}{S}\right)}{s_j(S-s_j)\frac{1}{S}}=\frac{1}{s_1}+\mathds{1}[i=j]\frac{1}{s_j}=A(i,j).$$
Since $\Delta_{\widetilde{\mathcal{N}}}[1]=g_{1, \widetilde{\mathcal{N}}}^{-1}$, we obtained that $B=A^{-1}$.

Now, we are showing that $g_1^{-1}$ and $A^{-1}$ are close if $\eps$ is small enough. We denote by $\|\cdot\|_1$ the maximum of the $\ell_1$ norms of the rows of a matrix.  For $E:=g_1-A$, we have $\|E\|_1<(k-1)\frac{3}{2}\eps<2k\eps$ and $\|A^{-1}\|_1=\|B\|_1\le K\gamma (1+\sum_{j=1}^{k-1}\frac{s_{j+1}}{S})<2K\gamma$, therefore $\|A^{-1}\|_1\|E\|_1\le 4kK\gamma\eps\le \frac{1}{8K}< 1/2<1$ for $\eps\le \frac{1}{32kK^2\gamma}$, so we obtain by \cite{stewart1969continuity} that
$$\|\Delta[1]-B\|_1=\|(A+E)^{-1}-A^{-1}\|_1\le \frac{\|A^{-1}\|_1^2\|E\|_1}{1-\|A^{-1}\|_1\|E\|_1}\le \frac{(2K\gamma)^22k\eps}{1-1/2}=16kK^2\gamma^2\eps,$$
implying $|\widetilde{C}_k-s_{k}\frac{S-s_{k}}{S}|\le 16kK^2\gamma^2\eps\le \frac{\gamma}{2}$. 

Since $s_{k}\frac{S-s_{k}}{S}=\left(\frac{1}{s_k}+\frac{1}{S-s_k}\right)^{-1}\ge \left(\frac{1}{\gamma}+\frac{1}{(k-1)\gamma}\right)^{-1}=\frac{k-1}{k}\gamma\ge \frac{2}{3}\gamma$, we get $\widetilde{C}_k\ge \frac{2}{3}\gamma-\frac{1}{2}\gamma=\frac{1}{6}\gamma$, therefore, by $\left|\frac{1}{a}-\frac{1}{b}\right|=\frac{|b-a|}{ab}$,
$$\left|\frac{1}{\widetilde{C}_k}-\frac{S}{s_{k}(S-s_{k})}\right|\le \frac{6}{\gamma}\frac{3}{2\gamma}16kK^2\gamma^2\eps=144kK^2\eps=g\eps$$ 
for $g:=144kK^2$ and $f:=32kK^2$.
\end{proof}

\begin{defi} \label{def:X_k}
    For an electric network $(G_n, \c_n)$ and a rooted tree $(T, 1)$, we define a random $\mathbf{X}_k:=\{X_1, \ldots, X_k\}\in V(G_n)^k$ the following way: we choose $X_1\in V(G_n)$ from the stationary distribution of $\c_n$, and inductively, $X_{j}$ is a random neighbor of $X_{p_T(j)}$ of distribution proportional to the conductances, where $p_T(j)$ is the parent of $j$ in $T$ likewise in Notation \ref{notation:v_k}.
\end{defi}

\begin{proof}[Proof of Lemma \ref{lemma_main:F_on_W^tc}] 
Since we have $F_{\c_n}(T(\mathbf{v}_k))=O(1)\frac{1}{|V(G_n)|}\frac{1}{\gamma_n^{k-1}}$ for any $\mathbf{v}_k\in \C(W_n^\tc)$ and $\sum_{\mathbf{v}_k\in \C(W_n^\tc):\; v_i=v}\prod_{j=2}^k c(v_{p_T(j)}, v_j)\le (K\gamma_n)^{k-1}$ for any $v\in W_n^\tc$ and any $1\le i\le k$, we get 
\begin{align*}
    \sum_{\mathbf{v}_k\in \C(W_n^\tc)\setminus \C(U_n)}F_{\c_n}(T(\mathbf{v}_k))\prod_{j=2}^k c(v_{p_T(j)}, v_j)&\le \frac{O(1)}{|V(G_n)|\gamma_n^{k-1}}\sum_{\mathbf{v}_k\in \C(W_n^\tc)\setminus \C(U_n)}\prod_{j=2}^k c(v_{p_T(j)}, v_j)\\
    &\le k|W_n^\tc\setminus U_n|\frac{O(1)(K\gamma_n)^{k-1}}{|V(G_n)|\gamma_n^{k-1}}=O(f_n).
\end{align*}

To show (\ref{formula:F_on_W^tc_second}), we generate $\mathbf{X}_k:=(X_1, \ldots, X_k)$ likewise in Definition \ref{def:X_k}, so $X_j$ is stationary for any $1\le j\le k$. Note that $\mathbf{X}_k\in \C(V(G_n))$ by the same computation as (\ref{formula:prob_notdistinct}) with probability $1-O(\delta_n)$. 

By (\ref{formula:unif_vs_stationary}), $\pi_n(V(G_n)\setminus (W_n^\tc\cap U_n))=O(\max(f_n, \delta_n))$ holds, so the LHS of (\ref{formula:F_on_W^tc_second}) is bounded as $$\P(\mathbf{X}_k\in \C(W_n^\tc\cap U_n))\ge 1-k\P(X_1\notin W_n^\tc\cap U_n)-\P(\mathbf{X}_k\notin \C(V(G_n)))=1-O(\max(f_n, \delta_n)).$$
\end{proof}

\begin{lemma}[Generalization of Theorem 3.4 of \cite{nachmias2022local}]\label{lemma:Reff_tree_regular}
    We consider an electric network satisfying Assumption \ref{assum:for_local_lim} and let $\mathbf{X}_k=\{X_1,\ldots,X_k\}$ be a random $k$-tuple drawn as described above. Then, with probability $1-O(\delta_n^{1/2})$, $\mathbf{X}_k$ is $T$-compatible,
\begin{gather}\label{formula:R_eff_k}
    \left| \Reff \left(X_{\ell} \leftrightarrow \{X_1, \ldots, X_{\ell-1}\} \right) - \frac{1}{C_{X_{\ell}}}\frac{\sum_{j=1}^{\ell}C_{X_j}}{\sum_{j=1}^{\ell-1}C_{X_j}}\right| \le \frac{\delta_n^{1/2}}{\gamma_n}.
\end{gather}
and
    \begin{gather*}
        \P\left(T_j(\mathbf{X}_k)\subseteq B_{\wst(\c_n)}(X_1, r)\Given \mathbf{X}_k\right)=(1+O(\delta_n^{1/2}))\frac{\sum_{j=1}^{k}C_{X_j}}{\prod_{j=1}^{k}C_{X_j}}\prod_{j=2}^{k}c(X_j, X_{p(j)}).
    \end{gather*}
\end{lemma}
\begin{proof} We have already checked in (\ref{formula:prob_notdistinct}) that $\mathbf{X}_k\in \C(W_n^\tc)$ with probability $1-O(\delta_n)$. Moreover, by Lemma \ref{lemma_main:Reff_edge} a) and the union bound,
\begin{align*}
    &\P \left( \exists i\ne j \le k: \; \Reff(X_i \leftrightarrow X_j)< \frac{1}{C_{X_i}}+\frac{1}{C_{X_j}}-\frac{2 \delta_n}{\gamma_n} \right )\\
    &\hspace{7cm}\le \sum_{\substack{i, j\in \{1, \ldots, k\}\\i\ne j}}\P(X_i=X_j)+k\sum_{i=1}^{k}\P(X_i\notin W_n^{\tc})=O(\delta_n).
\end{align*}
Moreover, because of reversibility, we can use the union bound and Lemma \ref{lemma:Reff_walk_regular} with $\eps=\sqrt{\delta_n}/(g\gamma_n)$, where $g$ comes from Lemma \ref{lemma:from_twovert_to_k+1}, and we get
\begin{align*}
    &\P \left( \exists i\ne j \le k: \; \Reff(X_i \leftrightarrow X_j) > \frac{1}{C_{X_i}}+\frac{1}{C_{X_j}}+\frac{ \sqrt{\delta_n}}{g\gamma_n} \right )\\
    &\hspace{7cm}\le \sum_{i, j=1}^{k} \P\left(\Reff(X_i \leftrightarrow X_j) > \frac{1}{C_{X_i}}+\frac{1}{C_{X_j}}+\frac{ \sqrt{\delta_n}}{g\gamma_n} \right )\\
    &\hspace{7cm}=O(\sqrt{\delta_n}).
\end{align*}
Hence, with probability $1-O(\sqrt{\delta_n})$, the conditions of Lemma \ref{lemma:from_twovert_to_k+1} follow with $\eps=\sqrt{\delta_n}/(g\gamma_n)$, giving us (\ref{formula:R_eff_k}). 

The second inequality of the Lemma follows directly from (\ref{formula:R_eff_k}) and the spatial Markov property $\P\left(T(\mathbf{X}_k)\subseteq B_{\wst(\c_n)}(X_1, r)\Given \mathbf{X}_k\right)= \prod_{i=2}^k \Reff \left(X_i \leftrightarrow \{X_1, \ldots, X_{i-1}\} \right)c(X_{p_T(i)}, X_i)$.
\end{proof}

Now, we give heuristics how the proof of Lemma 
\ref{lemma_main:Reff_tree_bal} is obtained from Lemma \ref{lemma:Reff_tree_regular} for regular graph sequences: for typical $\mathbf{v}_k$s, Lemma 
\ref{lemma_main:Reff_tree_bal} gives us
\begin{align*}
        \P\left(T(\mathbf{v}_k)\subseteq B_{\ust(G_n)}(v_1, r)\right)&=\prod_{i=2}^k \Reff \left(v_i \leftrightarrow \{v_1, \ldots, v_{i-1}\} \right)\sim\prod_{i=2}^k \left(\frac{i}{(i-1)d_n}\right)=\frac{k}{d_n^{k-1}}.
\end{align*}

Writing $H_n:=\{(u, v)\in E(G_n): v\in \mathbf{v}_t\text{ and } u\in V(G_n)\setminus\mathbf{v}_k\}$, by inclusion-exclusion formula, we have 
\begin{align*}
    &\P\left(B_{\ust(G_n)}(v_1, r)=T(\mathbf{v}_k)\right)\\
    &\hspace{3cm}=\P\Big(T(\mathbf{v}_k)\subseteq B_{\ust(G_n)}(v_1, r)\Big)-\P\left(\bigcup_{e\in H_n}\left\{T(\mathbf{v}_k)\cup\{e\}\subseteq B_{\ust(G_n)}(v_1, r)\right\}\right)\\
    &\hspace{3cm}\sim \sum_{i=0}^{\infty}(-1)^i\binom{td_n}{i}\frac{k+i}{d_n^{k+i-1}} \sim \sum_{i=0}^{\infty}(-1)^i\frac{(td_n)^i}{i!}\frac{k+i}{d_n^{k+i-1}}\\ 
    &\hspace{3cm}=\frac{1}{d_n^{k-1}}\sum_{i=0}^{\infty}k\frac{(-1)^i}{i!}t^i-\frac{1}{d_n^{k-1}}\sum_{i=1}^{\infty}\frac{(-1)^{i-1}t^i}{(i-1)!}=e^{-t}\frac{k-t}{d_n^{k-1}},
\end{align*}
and by following the reasoning of the proof of Theorem~\ref{thm:local_ust_almost_reg}, this gives us $\P\left(B_{\ust(G_n)}(X, r)\cong T\right)\sim e^{-t}\frac{k-t}{|\mathrm{Stab}_T|}$ where $X\sim \unif(V(G_n))$, and these probabilities coincide with the probabilities for the Poisson(1) branching tree conditioned to survive forever. 

For almost balanced sequences, this computation has the following result: for typical $\mathbf{v}_k$s,
\begin{align*}
        \P\left(T(\mathbf{v}_k)\subseteq B_{\wst(\c_n)}(v_1, r)\right)&= \prod_{i=2}^k \Reff \left(v_i \leftrightarrow \{v_1, \ldots, v_{i-1}\} \right)c(v_{p_T(i)}, v_i)\\
        &\sim\prod_{i=2}^k \left(\frac{1}{C_{v_i}}\frac{\sum_{j=1}^{i}C_{v_j}}{\sum_{j=1}^{i-1}C_{v_j}}c(v_{p_T(i)}, v_i)\right)= \frac{\sum_{j=1}^{k}C_{v_j}}{\prod_{i=1}^kC_{v_j}}\prod_{i=2}^kc(v_{p_T(i)}, v_i).
\end{align*}
holds. By inclusion-exclusion formula, we will check that
$$\P\left(B_{\wst(\c_n)}(v_1, r)=T(\mathbf{v}_k)\right)\sim e^{-b(\mathbf{v}_k)}\frac{\sum_{j=t+1}^{k}C_{v_j}}{\prod_{i=1}^kC_{v_j}}\prod_{i=2}^kc(v_{p_T(i)}, v_i).$$

In the proof we are making more precise what typical means.

\begin{proof}[Proof of Lemma 
\ref{lemma_main:Reff_tree_bal}] We start by introducing the notation for this proof. We mean by $\mathbf{x}_k=(x_1, \ldots, x_k)\in \C(V(G_n))$ deterministic vertices  and by $\mathbf{X}_k=(X_1, \ldots, X_k)$ random vertices chosen as in Definition \ref{def:X_k}. We abbreviate $C(\mathbf{x}_\ell):=\sum_{j=1}^{\ell}C_{x_j}$ and  $b(\mathbf{x}_\ell):=\sum_{j=1}^{\ell}b(x_j)$ and $H_n(\mathbf{x}_k):=H_n:=\{(u, v)\in E(G_n): v\in \mathbf{x}_t\text{ and } u\in V(G_n)\setminus\mathbf{x}_k\}$. We write
$\P_{\mathbf{X}_k}(\cdot):=\P(\cdot\given \mathbf{X}_k)$. In this proof, we often consider $\mathbf{e}_i=(e_1, \ldots, e_i)\in H_n^i$ for which we use the notation $e_j=(x_{l(j)}, u_j)$, i.e., the endpoint of $e_j$ inside $V(T(\mathbf{x}_k))$ is denoted by $x_{l(j)}$ and the outer endpoint by $u_j$. We write 
\begin{align*}
    P_n^{\text{(eq)}}(\mathbf{x}_k)&:=\P_{\mathbf{x}_k}\Big(T(\mathbf{x}_k)=B_{\wst(\c_n)}(x_1, r)\Big),\\
    P_n^{\text{(sub)}}(\mathbf{x}_k)&:=\P_{\mathbf{x}_k}\Big(T(\mathbf{x}_k)\subseteq B_{\wst(\c_n)}(x_1, r)\Big),
\end{align*}
\begin{align*}
q_n(T(\mathbf{x}_k))&:=\prod_{j=2}^{k} \Reff \left(\left\{x_1, \ldots, x_{j-1}\right\} \leftrightarrow x_j \right)c(x_j, x_{p(j)})\\
    q_n(\mathbf{e}_i\given T(\mathbf{x}_k))&:=\prod_{j=1}^{i} \Reff \left(V\left(T(\mathbf{x}_k)\cup \bigcup_{\ell=1}^{j-1}e_{\ell}\right) \leftrightarrow u_j \right)c(e_j)\\
    Q_n(i\given T(\mathbf{x}_k))&:=\sum_{\mathbf{e}_i\in H_n^i}q_n(\mathbf{e}_i\given T(\mathbf{x}_k))
    \end{align*}
\begin{align*}
s_n(T(\mathbf{x}_k))&:=\prod_{j=2}^{k} \frac{C(\mathbf{x}_j)}{C(\mathbf{x}_{j-1})}\frac{1}{C_{x_j}}c(x_j, x_{p(j)})=\frac{C(\mathbf{x}_k)}{\prod_{j=1}^{k}C_{x_j}}\prod_{e\in T(\mathbf{x}_k)} c(e),\\
    s_n(\mathbf{e}_i\given T(\mathbf{x}_k))&:=\prod_{j=1}^{i} \frac{C(\mathbf{x}_k)+C(\mathbf{u}_j)}{C(\mathbf{x}_k)+C(\mathbf{u}_{j-1})}\frac{1}{C_{u_j}}c(e_j)=\frac{C(\mathbf{x}_k)+C(\mathbf{u}_i)}{C(\mathbf{x}_k)\prod_{j=1}^{i}C_{u_j}}\prod_{j=1}^ic(e_j),\\
    S_n(i\given T(\mathbf{x}_k))&:=\sum_{\mathbf{e}_i\in H_n^i}s_n(\mathbf{e}_i\given T(\mathbf{x}_k))
\end{align*}
for any $i\ge 1$, and our convention is that $Q_n(0\given \mathbf{x}_k)=q_n(\emptyset\given T(\mathbf{x}_k)):=1$ and $S_n(0\given \mathbf{x}_k)=    s_n(\emptyset\given T(\mathbf{x}_k)):=1$. By $O_i(\cdot)$ and $\Theta_i(\cdot)$, we mean the existence of constants not only depending on $T$ but also on $i$.  

The reader might find an informal list of connections between the quantities defined above useful. We have $P_n^{\text{(sub)}}(\mathbf{x}_k)=q_n(T(\mathbf{x}_k))$ for any $\mathbf{x}_k$ by Kirchhoff's formula and the spatial Markov property, and one expects that typically $q_n(T(\mathbf{X}_k))\sim s_n(T(\mathbf{X}_k))$ and $q_n(\mathbf{e}_i\given T(\mathbf{X}_k))\sim s_n(\mathbf{e}_i\given T(\mathbf{X}_k))$ by applying Lemma \ref{lemma:Reff_tree_regular} several times, and after summing, $Q_n(i\given T(\mathbf{X}_k))\sim S_n(i\given T(\mathbf{X}_k))$. We are going to make these observations precise later in this proof.

Combining the inclusion--exclusion formula and the spatial Markov property,
\begin{equation}\label{formula:incl-excl}
\begin{aligned}
    P_n^{\text{(eq)}}(\mathbf{X}_k)&=\P_{\mathbf{X}_k}\Big(T(\mathbf{X}_k)\subseteq B_{\wst(\c_n)}(X_1, r)\Big)-\P_{\mathbf{X}_k}\Big(T(\mathbf{X}_k)\subsetneq B_{\wst(\c_n)}(X_1, r)\Big)
    \\&=P_n^{\text{(sub)}}(\mathbf{X}_k)-\P_{\mathbf{X}_k}\left(\bigcup_{e\in H_n}\left\{T(\mathbf{X}_k)\cup\{e\}\subseteq B_{\wst(\c_n)}(X_1, r)\right\}\right)\\
    &=P_n^{\text{(sub)}}(\mathbf{X}_k)\sum_{i=0}^{\infty}(-1)^i\sum_{E\subseteq H_n:\; |E|=i}\frac{\P_{\mathbf{X}_k}\left(T(\mathbf{X}_k)\cup E\subseteq B_{\wst(\c_n)}(X_1, r)\right)}{P_n^{\text{(sub)}}(\mathbf{X}_k)}\\
    &=P_n^{\text{(sub)}}(\mathbf{X}_k)\sum_{i=0}^{\infty}\frac{(-1)^i}{i!}Q_n(i\given T(\mathbf{X}_k)),
\end{aligned}
\end{equation}
where the $\frac{1}{i!}$ comes in since the summation runs over subsets in the third line, and while, in the fourth line, we are summing over ordered $i$-tuples: here, we also use that if $\mathbf{e}_i$ does not consist of different edges, then $q_n(\mathbf{e}_i\given T(\mathbf{x}_k))=0$ is automatic.

\emph{Our aim} is to give a closed form of $P_n^{\text{(eq)}}(\mathbf{X}_k)$ with high probability. For any $\eps>0$, we would like to define some $i_0:=i_0(\eps)\in \mathbb{N}$ and $\mathcal{A}'_n:=\mathcal{A}'_n(i_0)\subseteq \CN$ with $\sum_{\mathbf{v}_k\in \CN\setminus \mathcal{A}'_n(i_{0}(\eps_n))} \prod_{e\in T(\mathbf{v}_k)} c(e)\le o(1)|V(G_n)|\gamma_n^{k-1}$ for some $\eps_n\rightarrow 0$ such that the summands of the last line of (\ref{formula:incl-excl}) can be bounded as 
\begin{equation}\label{formula:Q_and_S}
\begin{aligned}
    &P_n^{\text{(sub)}}(\mathbf{x}_k)Q_n(i\given T(\mathbf{x}_k))=(1+O_{i_0}(\delta_n^{1/4}))s_n(T(\mathbf{x}_k))S_n(i\given T(\mathbf{x}_k))
\end{aligned}
\end{equation}
for any $\mathbf{x}_k\in \mathcal{A}_n'$ and for every $i\le i_0$; moreover,
\begin{align}\label{formula:S_and_b}
    &S_n(i\given T(\mathbf{x}_k))=(1+O_{i_0}(\delta_n))\left(b(\mathbf{x}_t)^i+i\frac{C(\mathbf{x}_t)}{C(\mathbf{x}_k)}b(\mathbf{x}_t)^{i-1}\right)
\end{align}
for any $\mathbf{x}_k\in \CN$ and for every $i\le i_0$; and
\begin{align}\label{formula:sum_is_small}
    &\sum_{i=i_0}^{\infty} \frac{1}{i!}\left(Q_n(i\given T(\mathbf{x}_k))+b(\mathbf{x}_t)^i+i\frac{C(\mathbf{x}_t)}{C(\mathbf{x}_k)}b(\mathbf{x}_t)^{i-1}\right)\le \eps
\end{align}
for any $\mathbf{x}_k\in \CN$. These are useful as (\ref{formula:incl-excl}) and (\ref{formula:Q_and_S}, \ref{formula:S_and_b}, \ref{formula:sum_is_small}) imply that  $P_n^{\text{(eq)}}(\mathbf{X}_k)$ is close to
$s_n(T(\mathbf{X}_k)\sum_{i}\frac{(-1)^i}{i!}\left(b(\mathbf{X}_t)^i+i\frac{C(\mathbf{X}_t)}{C(\mathbf{X}_k)}b(\mathbf{X}_t)^{i-1}\right)$ which can be approximated easily with a closed formula. 

\emph{Now, we are defining $i_0$, and we are showing (\ref{formula:sum_is_small}).} Similarly to how we checked (\ref{formula:main_ineq_third_line}), we have $\Reff\left(V(T(\mathbf{x}_k)\cup \cup_{\ell=1}^{j-1}\{e_{\ell}\}) \leftrightarrow u_j \right)\le \Reff\left(x_{l(j)} \leftrightarrow u_j \right)$ by monotonicity of the effective resistances where $e_j=(x_{l(j)}, u_j)$. Therefore, for any $\{x_1, \ldots, x_k\}\subseteq \tW_n$, by Definition \ref{def:typ_neigh},
\begin{equation}\label{formula:Q_upper_bound}
\begin{aligned}
    Q_n(i\given T(\mathbf{x}_k))&\le \sum_{(e_1, \ldots, e_i)\in H_n^i}\prod_{j=1}^{i} \Reff\left(x_{l(j)} \leftrightarrow u_j \right)c(e_j)\\
    &=\left(\sum_{(u, x_j)\in H_n}\Reff \left(x_j \leftrightarrow u \right)c(x_j, u)\right)^i\le \left(\frac{4}{\gamma_n}\gamma_nKt+t\sqrt{\delta_n}\right)^i.
\end{aligned}
\end{equation}
Hence,
\begin{align}\label{formula:finite_sum_1}
    \sum_{i=0}^{\infty} \frac{1}{i!}Q_n(i\given T(\mathbf{x}_k))\le \sum_{i=0}^{\infty} \frac{(4Kt+t\sqrt{\delta_n})^i}{i!}=e^{4Kt+t\sqrt{\delta_n}}<\infty
\end{align}
uniformly for every nice $\mathbf{x}_k$.

Moreover, by (\ref{formula:b(v)_typ}) we have $b(v)\in \left(\frac{1-\sqrt{\delta_n}}{K}, K+8\sqrt{\delta_n}\right)=\Theta(1)$ for any $v\in \tW_n$, therefore
\begin{align*}
b(\mathbf{x}_t)^i+i\frac{C(\mathbf{x}_t)}{C(\mathbf{x}_k)}b(\mathbf{x}_t)^{i-1}\le b(\mathbf{x}_t)^i+\frac{itK\gamma_n}{k\gamma_n}b(\mathbf{x}_t)^{i-1}\le (tK+tO(\sqrt{\delta_n}))^i+\frac{itK}{k}(tK+tO(\sqrt{\delta_n}))^{i-1},
\end{align*}
and then clearly, 
uniformly for every nice $\mathbf{x}_k$, we have
\begin{gather}\label{formula:finite_sum_2}
\sum_{i=0}^{\infty}\frac{1}{i!}\left(b(\mathbf{x}_t)^i+i\frac{C(\mathbf{x}_t)}{C(\mathbf{x}_k)}b(\mathbf{x}_t)^{i-1}\right)<\infty.
\end{gather}
Then, we choose $i_0(\eps)$ such that, for every nice $\mathbf{x}_k$, we have
\begin{align*}
    \sum_{i=i_0}^{\infty} \frac{1}{i!}\left(Q_n(i\given T(\mathbf{x}_k))+b(\mathbf{x}_t)^i+i\frac{C(\mathbf{x}_t)}{C(\mathbf{x}_k)}b(\mathbf{x}_t)^{i-1}\right) \le \eps.
\end{align*}

\emph{Now, we are proving (\ref{formula:Q_and_S}).} We abbreviate as $$\Delta(\mathbf{x}_{k}, \mathbf{e}_i):=\big|q_n(T(\mathbf{x}_k))q_n(\mathbf{e}_i\given T(\mathbf{x}_k))-s_n( T(\mathbf{x}_k))s_n(\mathbf{e}_i\given T(\mathbf{x}_k))\big|.$$

We write $\mathcal{T}_T(i)$ for the set of rooted trees obtained from $(T, 1)$ by connecting $i$ leaves to $V(B_{T}(1, r-1))=\{1, \ldots, t\}$. For $T'\in \mathcal{T}_T(i)$ and a $T$-tuple $\mathbf{x}_k$, we can consider a $T'$-tuple $(\mathbf{x}_k, \mathbf{e}_i)$, and we write that
$\mathbf{x}_{k+i}:=V((\mathbf{x}_k, \mathbf{e}_i))$ in the sense that $x_{k+\ell}:=u_{\ell}$ is the endpoint of $e_{\ell}$ outside of $V(\mathbf{X}_k)$. We consider 
\begin{align*}
    \mathcal{Y}_n^{(0)}(i)&:=\left\{(\mathbf{x}_k, \mathbf{e}_i) \Given \mathbf{x}_k\text{ is $T$-nice, }\mathbf{e}_i\in H_n^i\right\},\\
    \mathcal{Y}_n^{(1)}(i)&:=\left\{(\mathbf{x}_k, \mathbf{e}_i) \Given \mathbf{x}_k\text{ is $T$-nice, }\mathbf{e}_i\in (H_n\cap\tE_n(4/\gamma_n))^i:\; \{u_1, \ldots, u_i\}\subseteq W_n^\tc\right\},\\
    \mathcal{Y}_n^{(2)}(i)&:=\left\{(\mathbf{x}_k, \mathbf{e}_i)\in \mathcal{Y}_n^{(1)}(i) \Given
    \substack{\text{(\ref{formula:R_eff_k}) holds for any edge of $T(\mathbf{x}_k)\cup \{e_1, \ldots, e_i\}$, and}\\\text{$u_1, \ldots, u_i$ are distinct}}\right\}
\end{align*}
where, by `(\ref{formula:R_eff_k}) holds for any edge', we mean that $\left| \Reff \left(x_{\ell} \leftrightarrow \{x_1, \ldots, x_{\ell-1}\} \right) - \frac{1}{C_{x_{\ell}}}\frac{\sum_{j=1}^{\ell}C_{x_j}}{\sum_{j=1}^{\ell-1}C_{x_j}}\right| \le \frac{\delta_n^{1/2}}{\gamma_n}$ for any $2\le \ell\le k+i$ and the $T'$-tuple $\mathbf{x}_{k+i}$.
Our aim is to show that $\sum_{(\mathbf{x}_k, \mathbf{e}_i)\in \mathcal{Y}_n^{(0)}(i)}\Delta(\mathbf{x}_{k}, \mathbf{e}_i)=O_i(|V(G_n)|)$.

In order to deal with $\sum_{(\mathbf{x}_k, \mathbf{e}_i)\in \mathcal{Y}_n^{(0)}(i)\setminus \mathcal{Y}_n^{(1)}(i)}\Delta(\mathbf{x}_{k}, \mathbf{e}_i)$, for any $\mathbf{x}_k\in \CN$, we have
\begin{align*}
   &\sum_{\substack{\mathbf{e}_i\in H_n^i:\\
    \exists j:\;u_j\notin W_n^\tc\text{ or }e_j\notin \tE_n(4/\gamma_n)}} q_n(\mathbf{e}_i\given T(\mathbf{x}_k))\le i \sum_{\substack{\mathbf{e}_i\in H_n^i:\\
    \;u_1\notin W_n^\tc\text{ or }e_1\notin \tE_n(4/\gamma_n)}} q_n(\mathbf{e}_i\given T(\mathbf{x}_k))
    \\
    &\hspace{3cm}\le iQ_n(i-1\given T(\mathbf{x}_k))\sum_{\substack{e_1\in H_n:\\
    u_1\notin W_n^\tc\text{ or }e_1\notin \tE_n(4/\gamma_n)}}\Reff \left(x_{l(1)} \leftrightarrow u_1 \right)c(x_{l(1)}, u_1)\\
    &\hspace{3cm}\le it(\sqrt{\delta_n}+4\sqrt{\delta_n})\left(\frac{4}{\gamma_n}\gamma_nKt+t\sqrt{\delta_n}\right)^{(i-1)}=O_i(\delta_n^{1/2}),
\end{align*}
where the first inequality follows from symmetry, the second from the monotonicity of the effective resistances, the third from(\ref{formula:Q_upper_bound}), (\ref{formula:b(v)_typ}) and Definition \ref{def:typ_neigh}. Moreover, by Lemma \ref{lemma_main:Reff_edge} a), we have $\Reff \left(x_{l(1)} \leftrightarrow u_1 \right)\ge \frac{1}{2C_{u_1}}+\frac{1}{2K\gamma_n}$, so the previous computation also gives us
$$\sum_{\substack{e_1\in H_n:\\
    u_1\notin W_n^\tc\text{ or }e_1\notin \tE_n(4/\gamma_n)}}c(x_{l(1)}, u_1)\left(\frac{1}{C_{u_1}}+\frac{1}{\gamma_n}\right)= O(\delta_n^{1/2})\text{ and } $$
    $$\sum_{\substack{e_1\in H_n}}c(x_{l(1)}, u_1)\left(\frac{1}{C_{u_1}}+\frac{1}{\gamma_n}\right)=O(1).$$
    Hence, expanding $s_n(\mathbf{e}_i\given T(\mathbf{x}_k))=\prod_{\ell=1}^{i}\frac{c(e_{\ell})}{C_{u_{\ell}}}+\sum_{j=1}^i\frac{c(e_j)}{C(\mathbf{X}_k)} \prod_{\ell\ne j}\frac{c(e_{\ell})}{C_{u_{\ell}}}$ where $C(\mathbf{X}_k)=\Theta(\gamma_n)$, we get
\begin{align*}
    &\sum_{\substack{\mathbf{e}_i\in H_n^i:\\
    \exists j:\;u_j\notin W_n^\tc\text{ or }e_j\notin \tE_n(4/\gamma_n)}} s_n(\mathbf{e}_i\given T(\mathbf{x}_k))\le\sum_{j=1}^i\sum_{\substack{\mathbf{e}_i\in H_n^i:\\
    u_j\notin W_n^\tc\text{ or }e_j\notin \tE_n(4/\gamma_n)}} s_n(\mathbf{e}_i\given T(\mathbf{x}_k))= O_i(\delta_n^{1/2}).
\end{align*}
Therefore, using that $\Delta(\mathbf{x}_{k}, \mathbf{e}_i)\le O(1)(q_n(\mathbf{e}_i\given T(\mathbf{x}_k))+s_n(\mathbf{e}_i\given T(\mathbf{x}_k)))\gamma_n^{-(k-1)}\prod_{e\in T(\mathbf{x}_k)}c(e)$ for any nice $\mathbf{x}_k$, we obtain
$$\sum_{(\mathbf{x}_k, \mathbf{e}_i)\in \mathcal{Y}_n^{(0)}(i)\setminus \mathcal{Y}_n^{(1)}(i)}\Delta(\mathbf{x}_{k}, \mathbf{e}_i)\le O_i(\delta_n^{1/2})\sum_{\mathbf{x}_k \in \CN} \left(\frac{O(1)}{\gamma_n^{k-1}}\prod_{e\in T(\mathbf{x}_k)}c(e)\right)\le O_i(\delta_n^{1/2})|V(G_n)|.$$

Now, we handle $\sum_{(\mathbf{x}_k, \mathbf{e}_i)\in \mathcal{Y}_n^{(1)}(i)\setminus \mathcal{Y}_n^{(2)}(i)}\Delta(\mathbf{x}_{k}, \mathbf{e}_i)$. Using $\Delta(\mathbf{x}_{k}, \mathbf{e}_i)=O(1)\gamma_n^{-(k+i-1)}\prod_{j=1}^{k+i-1}c(e_j)$ and
$\P(\mathbf{X}_{k+i}=\mathbf{x}_{k+i})=\Theta_i(1)|V(G_n)|^{-1}\gamma_n^{-(k+i-1)}\prod_{j=1}^{k+i-1}c(e_j)$ for any $\mathbf{x}_{k+i}\in \mathcal{Y}_{n}^{(1)}(i)$ and applying Lemma \ref{lemma:Reff_tree_regular} to every $T'\in \mathcal{T}_T(i)$, we get
\begin{align*}
    &\sum_{\substack{(\mathbf{x}_k, \mathbf{e}_i)\in \mathcal{Y}_n^{(1)}(i)\setminus \mathcal{Y}_n^{(2)}(i)}}\Delta(\mathbf{x}_{k}, \mathbf{e}_i)\\
    &\hspace{1.5cm}\le|V(G_n)|\sum_{\substack{(\mathbf{x}_k, \mathbf{e}_i)\in \mathcal{Y}_n^{(1)}(i)\setminus \mathcal{Y}_n^{(2)}(i)}} \frac{O_i(1)}{|V(G_n)|\gamma_n^{k-1+i}}\prod_{j=2}^{k+i}c(x_{p(j)}, x_j)\\
    &\hspace{1.5cm}\le 
    O_i(1)|V(G_n)|\sum_{T'\in \mathcal{T}_T(i)}\P\left(\mathbf{X}^{T'}_{k+i}\notin\mathcal{C}^{T'}(V(G_n))\text{ or (\ref{formula:R_eff_k}) fails for some edges of $T'(\mathbf{X}^{T'}_{k+i})$}\right)\\
    &\hspace{1.5cm}\le O_i(\delta_n^{1/2})|V(G_n)|.
\end{align*}

Now, we estimate $\sum_{(\mathbf{x}_k, \mathbf{e}_i)\in \mathcal{Y}_n^{(2)}(i)} \Delta(\mathbf{x}_{k}, \mathbf{e}_i)$. The inequality (\ref{formula:R_eff_k}) for $\mathbf{x}_{k+i}$ belonging to any $T'\in \mathcal{T}_T(i)$ and multiplied by $c(x_{\ell}, x_{p_{T'}(\ell)})$ can be rewritten as 
$$c(x_{\ell}, x_{p_{T'}(\ell)})\Reff \left(x_{\ell} \leftrightarrow \{x_1, \ldots, x_{\ell-1}\} \right)\in \left(1-O_i(\delta_n^{1/2}), 1+O_i(\delta_n^{1/2})\right)\frac{c(x_{\ell}, x_{p_{T'}(\ell)})\sum_{j=1}^{\ell}C_{x_j}}{C_{x_{\ell}}\sum_{j=1}^{\ell-1}C_{x_j}}$$
for any $2\le \ell \le k+i$. Multiplying these estimates for $2\le \ell \le k+i$, we get
\begin{align}
q_n(T(\mathbf{x}_k))q_n(\mathbf{e}_i\given T(\mathbf{x}_k))&\in \left((1-O_i(\delta_n^{1/2}))^{k-1+i},(1+O_i(\delta_n^{1/2}))^{k-1+i}\right)s_n( T(\mathbf{x}_k))s_n(\mathbf{e}_i\given T(\mathbf{x}_k))\label{formula:proper_qs}
\end{align}
where $s_n(T(\mathbf{x}_k))s_n(\mathbf{e}_i\given T(\mathbf{x}_k))=\Theta(1)\gamma_n^{-(k-1+i)}\prod_{j=1}^{k+i-1}c(e_j)$, hence
\begin{align*}
    &\sum_{\substack{(\mathbf{x}_k, \mathbf{e}_i)\in \mathcal{Y}_n^{(2)}(i)}}\Delta(\mathbf{x}_{k}, \mathbf{e}_i)\\
    &\hspace{2cm}\le\frac{O_i\left((1+O_i(\delta_n^{1/2}))^{k-1+i}-(1-O_i(\delta_n^{1/2}))^{k-1+i}\right)}{\gamma_n^{k-1+i}}\sum_{\substack{(\mathbf{x}_k, \mathbf{e}_i)\in \mathcal{Y}_n^{(2)}(i)}\;}\prod_{j=2}^{k+i}c(x_{p(j)}, x_j)\\
    &\hspace{2cm}\le O_i(\delta_n^{1/2}) \frac{O_i(1)}{\gamma_n^{k-1+i}}\sum_{\substack{(\mathbf{x}_k, \mathbf{e}_i)\in \mathcal{Y}_n^{(2)}(i)}\;}\prod_{j=2}^{k+i}c(x_{p(j)}, x_j)=O_i(\delta_n^{1/2})|V(G_n)|.
\end{align*}

Therefore,
\begin{align*}
\sum_{\substack{\mathbf{x}_k\in \CN}}\big|q_n(T(\mathbf{x}_k))Q_n(i\given T(\mathbf{x}_k))-s_n( T(\mathbf{x}_k))S_n(i\given T(\mathbf{x}_k))\big|&\le\sum_{\substack{(\mathbf{x}_k, \mathbf{e}_i)\in \mathcal{Y}_n^{(0)}(i)}}\Delta(\mathbf{x}_{k}, \mathbf{e}_i)\\
    &=O_i(\delta_n^{1/2})|V(G_n)|.
\end{align*}
Then, for
\begin{align*}
    \mathcal{A}'_n(i_0):=\left\{\mathbf{x}_k \Given \big|q_n(T(\mathbf{x}_k))Q_n(i\given T(\mathbf{x}_k))-s_n( T(\mathbf{x}_k))S_n(i\given T(\mathbf{x}_k))\big|\le \frac{\delta_n^{1/4}}{\gamma_n^{k-1}}\prod_{e\in T(\mathbf{x}_k)}c(e)\; \forall i\le i_0\right\},
\end{align*}
we have that $\sum_{\mathbf{v}_k\in \CN\setminus \mathcal{A}'_n(i_0)} \prod_{e\in T(\mathbf{v}_k)} c(e)\le \frac{O_i(\delta_n^{1/2})|V(G_n)|}{\delta_n^{1/4}/\gamma_n^{k-1}}=O_{i_0}(\delta_n^{1/4})|V(G_n)|\gamma_n^{k-1}$. Since we have that $s_n( T(\mathbf{x}_k))S_n(i\given T(\mathbf{x}_k))=\Theta_i(1)\gamma_n^{-(k-1)}\prod_{e\in T(\mathbf{v}_k)} c(e)$ for any $\mathbf{x}_k\in \CN$, we indeed obtained (\ref{formula:Q_and_S}) on a large $\mathcal{A}'_n$.  In particularly, (\ref{formula:Q_and_S}) with $i=0$ gives us that $q_n(T(\mathbf{x}_k))\le (1+\delta_n^{1/4})s_n(T(\mathbf{x}_k))\le 2s_n(T(\mathbf{x}_k))$ for large enough $n$ and $\mathbf{x}_k\in \mathcal{A}'_n$.

\emph{In order to check (\ref{formula:S_and_b}),} note that we can rewrite
\begin{equation*}
\begin{aligned}
    &\sum_{\mathbf{e}_i \in H_n^i} \frac{C(\mathbf{X}_k)+C((u_1, \ldots, u_i))}{C(\mathbf{X}_k)\prod_{j=1}^{i}C_{u_j}}\prod_{j=1}^{i}c(e_j)=\sum_{\mathbf{e}_i \in H_n^i} \frac{C(\mathbf{X}_k)+iC_{u_i}}{C(\mathbf{X}_k)}\frac{\prod_{j=1}^{i}c(e_j)}{\prod_{j=1}^{i}C_{u_j}}\\
    &\hspace{2cm}=\sum_{\mathbf{e}_i \in H_n^i} \prod_{j=1}^{i}\frac{c(e_j)}{C_{u_j}}+i\sum_{e_i\in H_n}\frac{c(e_i)}{C(\mathbf{X}_k)}\sum_{\mathbf{e_{i-1}} \in H_n^{i-1}} \prod_{j=1}^{i-1}\frac{c(e_j)}{C_{u_j}}\\
    &\hspace{2cm}=(1+O_i(\delta_n))\left(b(\mathbf{X}_t)^i+i\frac{C(\mathbf{X}_t)}{C(\mathbf{X}_k)}b(\mathbf{X}_t)^{i-1}\right)
\end{aligned}
\end{equation*}
where the first equality holds by symmetry and we have the last inequality from the estimates $|\sum_{e\in H_n}c(e)-C(\mathbf{X}_t)|\le tk\delta_n\gamma_n$ and $|\sum_{\ell=1}^t\sum_{u\in V(G_n)\setminus \{x_1, \ldots, x_k\}}\frac{c(x_\ell, u)}{C_u}-b(\mathbf{x}_t)| \le \frac{kt\delta_n \gamma_n}{\gamma_n}=O(\delta_n)b(\mathbf{x}_t)$ for any $\mathbf{x}_k\in \CN$ by Assumption~\ref{assum:for_local_lim} and (\ref{formula:b(v)_typ}). We remark that, by $s_n(T(\mathbf{X}_k))=\frac{C(\mathbf{X}_k)\prod_{e\in T(\mathbf{X}_k)} c(e)}{\prod_{j=1}^{k}C_{X_j}}$,
\begin{equation}\label{formula:b_closed_form_F}
\begin{aligned}
s_n(T(\mathbf{X}_k))\sum_{i=0}^\infty\frac{(-1)^i}{i!}\left(b(\mathbf{X}_t)^i+i\frac{C(\mathbf{X}_t)}{C(\mathbf{X}_k)}b(\mathbf{X}_t)^{i-1}\right)&=s_n(T(\mathbf{X}_k))\left(e^{-b(\mathbf{X}_t)}-\frac{C(\mathbf{X}_t)}{C(\mathbf{X}_k)}e^{-b(\mathbf{X}_t)}\right)\\
&=\frac{C(\mathbf{X}_k)\prod_{e\in T(\mathbf{X}_k)} c(e)}{\prod_{j=1}^{k}C_{X_j}} \frac{e^{-b(\mathbf{X}_t)}\sum_{j=t+1}^kC_{X_j}}{C(\mathbf{X}_k)}\\
&=|V(G_n)|F_{\c_n}(T(\mathbf{X}_k))\prod_{e\in T(\mathbf{X}_k)} c(e)
\end{aligned}
\end{equation}

\emph{Now, we are summarizing the proof.} On the event $\{\mathbf{X}_k\in \mathcal{A}'_n(i_0(\eps))\}$ and for large enough $n$,
\begin{align*}
&\left|P_n^{\text{(eq)}}(\mathbf{X}_k)-|V(G_n)|F_{\c_n}(T(\mathbf{X}_k))\prod_{e\in T(\mathbf{X}_k)} c(e)\right| \\
&\hspace{2.5cm}=\left|P_n^{\text{(eq)}}(\mathbf{X}_k)-s_n(T(\mathbf{X}_k))\sum_{i=0}^{\infty}\frac{(-1)^i}{i!}\left(b(\mathbf{X}_t)^{i}+i\frac{C(\mathbf{X}_t)}{C(\mathbf{X}_k)}b(\mathbf{X}_t)^{i-1}\right)\right|\\
&\hspace{2.5cm}\le\left|P_n^{\text{(eq)}}(\mathbf{X}_k)-s_n(T(\mathbf{X}_k))\sum_{i=0}^{i_0}\frac{(-1)^i}{i!}\left(b(\mathbf{X}_t)^{i}+i\frac{C(\mathbf{X}_t)}{C(\mathbf{X}_k)}b(\mathbf{X}_t)^{i-1}\right)\right|+\eps_ns_n(T(\mathbf{X}_k))\\
&\hspace{2.5cm}\le\left|P_n^{\text{(eq)}}(\mathbf{X}_k)-s_n(T(\mathbf{X}_k))\sum_{i=0}^{i_0}\frac{(-1)^i}{i!}S_n(i\given T(\mathbf{X}_k))\right|+\left(o_{i_0}(1)+\eps\right)s_n(T(\mathbf{X}_k))\\
&\hspace{2.5cm}\le \left|P_n^{\text{(eq)}}(\mathbf{X}_k)-q_n(T(\mathbf{X}_k))\sum_{i=0}^{i_0}\frac{(-1)^i}{i!}Q_n(i\given T(\mathbf{X}_k))\right|+\left(o_{i_0}(1)+\eps\right)s_n(T(\mathbf{X}_k))\\
&\hspace{2.5cm}\le \left(\eps+o_{i_0}(1)+\eps\right)s_n(T(\mathbf{X}_k))\le 3\eps s_n(T(\mathbf{X}_k))\\
&\hspace{2.5cm}\le \frac{3e^{2Kt}kK}{(k-t)}\eps|V(G_n)|F_{\c_n}(T(\mathbf{X}_k))\prod_{j=2}^k c(X_{p_T(j)}, X_j)
\end{align*}
where the second line follows from (\ref{formula:b_closed_form_F}), the third from (\ref{formula:sum_is_small}), the fourth from (\ref{formula:S_and_b}), the fifth from (\ref{formula:Q_and_S}) and and $q_n(T(\mathbf{X}_k))\le 2s_n(T(\mathbf{X}_k))$, the sixth from (\ref{formula:incl-excl}), (\ref{formula:sum_is_small}) and $q_n(T(\mathbf{X}_k))=P_n^{\text{(sub)}}(\mathbf{X}_k)$ and the seventh from that (\ref{formula:b(v)_typ}) implies $|V(G_n)|F_{\c_n}(T(\mathbf{X}_k))\prod_{j=2}^k c(X_{p_T(j)}, X_j)\ge e^{-2Kt}\frac{k-t}{kK}s_n(T(\mathbf{X}_k))$.

For $C:=\frac{3e^{2Kt}kK}{(k-t)}$, any fixed $\eps>0$ and large enough $n$, we have obtained that $\mathcal{A}'_n(i_0(\eps))\subseteq \mathcal{A}_n(C\eps)$ and the total conductance of the excluded nice tuples is $\sum_{\mathbf{v}_k\in \CN\setminus \mathcal{A}'_n(i_0(\eps))} \prod_{e\in T(\mathbf{v}_k)} c(e)\le o(1)|V(G_n)|\gamma_n^{k-1}$. We finish our proof by choosing $ \eps_n\rightarrow 0$ so slowly that both properties continue to hold for $\eps_n$.
\end{proof}

\subsubsection{The assumption is satisfied by i.i.d.~networks}
\label{subsubsec:i.i.d.}
In this subsection, we prove Lemma \ref{lemma:rdm_cond_balanced} and Lemma \ref{lemma:F_annealed-quenched}.

Recall our notation: our random conductances are $c_{\beta}(e):=\exp(-\beta U_e)$ where $\mathbf{U}:=\{U_e\}_{e\in E(G)}$ are distributed as i.i.d.~$\unif[0,1]$. We write $\mu_n:=\E[\exp(-\beta_nU_e)]=\frac{1-e^{-\beta_n}}{\beta_n}$ and assume in this subsection that $\beta_n\ll d_n$. Note that $\mu_n=\Theta(1)/\max(\beta_n, 1)$.

We are using the following observation several times during the proofs of Lemmata \ref{lemma:rdm_cond_balanced} and \ref{lemma:F_annealed-quenched}:
\begin{lemma}\label{lemma:chebyshev} For a sequence $G_n$, we consider some sets of vertices $U_{1, n}, U_{2, n}\subseteq V(G_n)$, and we write $\vec{E}_n:=\{(u_1, u_2):\;u_1\in U_{1, n},\;u_2\in U_{2, n}, u_1\sim u_2\}$ and $N_n:=|\vec{E}_n|$. Moreover, we consider some $a_n>0$, $B\ge 1$ and some parameters $\{p_{(u_1, u_2)}\}_{(u_1, u_2)\in \vec{E}_n}$ from $[1/B, B]$. Then, for $S_{n}:=\sum_{(u_1, u_2)\in \vec{E}_n}p_{(u_1, u_2)}c_{\beta_n}(u_1, u_2)$, we have 
$$\P\left(\left|S_{n}-\E[S_{n}]\right|\ge a_n\E\left[S_{n}\right]\right)\le \frac{2B^4}{\mu_nN_na_n^2}$$
\end{lemma}
We remark that, most of the times, we are using this lemma with $p\equiv 1$.
\begin{proof}
    For any edge $e:=(u_1, u_2)$ running between $U_{1, n}$ and $U_{2, n}$, we assign $X_e:=(p_{(u_1, u_2)}+\mathds{1}[(u_2, u_1)\in \vec{E}_n]p_{(u_2, u_1)})c_{\beta_n}(u_1, u_2)$. Then, $\E\left[X_{e}\right]\ge B^{-1}\mu_n$ and $\var(X_{n, i})\le (2B)^2(\frac{1-e^{-2\beta_n}}{2\beta_n}-\mu_n^2)=4B^2(\frac{1+e^{-\beta_n}}{2}\mu_n-\mu_n^2)\le 2B^2\mu_n$ since $\sup_{x\ge 0} \frac{1+e^{-x}}{2}-\frac{1-e^{-x}}{x}=\frac{1}{2}$ . Therefore, $\E[S_{n}]\ge N_nB^{-1}\mu_n$ and, by independence, $\var(S_{n})\le N_n2B^2\mu_n$. By Chebyshev's inequality,
    \begin{align*}
        \P\left(\left|S_{n, N_n}-\E[S_{n, N_n}]\ge a_n\E[S_{n, N_n}]\right|\right)\le \frac{\var(S_{n, N_n})}{(a_n\E[S_{n, N_n}])^2}\le \frac{2B^2\mu_nN_n}{a_n^2(B^{-1}\mu_nN_n)^2}=\frac{2B^4}{a_n^2\mu_nN_n}.
    \end{align*}
\end{proof}

\begin{proof}[Proof of Lemma \ref{lemma:rdm_cond_balanced}] We prove only part b), since part a) works the same way, but there are less conditions to check. In this proof, we abbreviate as $W_n^\tc(G_n):=W_n^\tc(G_n, \mathbf{1}, d_n, \widetilde{K})$ and $W_n^\tc(\c_{\beta_n}):=W_n^\tc(G_n, \c_{\beta_n}, \gamma_n, K)=W_n^\tc(G_n, \c_{\beta_n}, \frac{\mu_nd_n}{2}, 2\widetilde{K}+1)$.

For any $v\in W_n^\tc(G_n)$, by Lemma \ref{lemma:chebyshev} with $N_n=\deg(v)\ge d_n$ and $a_n=(2\widetilde{K})^{-1}$, we have
\begin{equation}\label{formula:prob_by_rmrk}
    \begin{aligned}
        \P\left(v\notin W_n^\tc(\c_{\beta_n})\right)&=\P\left(\frac{C_{\beta_n}(v)}{\mu_nd_n/2}\notin [1, 2\widetilde{K}+1]\right)=\P\left(\frac{C_{\beta_n}(v)}{\mu_nd_n}\notin \left[1-\widetilde{K}\frac{1}{2\widetilde{K}}, \widetilde{K}+\widetilde{K}\frac{1}{2\widetilde{K}}\right]\right)\\
&\le \P\left(|C_{\beta_n}(v)-\E[C_{\beta_n}(v)]|\ge \frac{1}{2\widetilde{K}}\E[C_{\beta_n}(v)]\right)= O\left(\frac{\max(1,\beta_n)}{d_n}\right)=O(\delta_n^4),
    \end{aligned}
\end{equation}
where we used the assumption $\max(1,\beta_n)/d_n=O(\widetilde{\delta}_n^4)$ of this lemma.

Hence, we have obtained $\E[|W_n^\tc(G_n)\setminus W_n^\tc(\c_{\beta_n})|]\le |W_n^\tc(G_n)|O(\widetilde{\delta}_n^4)$, so by Markov's inequality
\begin{equation}\label{formula:Markov_for_sizes}
    \begin{aligned}
    \P\left(|V(G_n)\setminus W_n^\tc(\c_{\beta_n})|\ge \delta_n|V(G_n)|\right)
    &\le \P\left(|W_n^\tc(G_n)\setminus W_n^\tc(\c_{\beta_n})|\ge \widetilde{\delta}_n|V(G_n)|\right)\\
    &\le \frac{\E[|W_n^\tc(G_n)\setminus W_n^\tc(\c_{\beta_n})|]}{\widetilde{\delta}_n|V(G_n)|}\le \frac{|W_n^\tc(G_n)|O(\widetilde{\delta}_n^4)}{\widetilde{\delta}_n|V(G_n)|}\\
    &=O(\widetilde{\delta}_n^3),
    \end{aligned}
\end{equation}
where we used in the first inequality the triangle inequality $|V(G_n)\setminus W_n^\tc(\c_{\beta_n})|\le |V(G_n)\setminus W_n^\tc(G_n)|+|W_n^\tc(G_n)\setminus W_n^\tc(\c_{\beta_n})|$, the bound $ |V(G_n)\setminus W_n^\tc(G_n)|\le \widetilde{\delta}_n|V(G_n)|$ and $\delta_n\ge 2\widetilde{\delta}_n$. Therefore, we obtained that the first condition of Assumption \ref{assum:for_local_lim} holds with probability $1-O(\widetilde{\delta}_n^3)$.

Now, we check the second condition of Assumption \ref{assum:for_local_lim} for $(G_n, \c_{\beta_n})$. For each $v\in W_n^{\tc}(G_n)$, by Cauchy-Schwarz inequality, $\E[C_{\beta_n}(v)^2]\le \var(C_{\beta_n}(v))+\E[C_{\beta_n}(v)]^2=O(\mu_n^2d_n^2)$ and (\ref{formula:prob_by_rmrk}), we have
\begin{align*}
    \E\left[\mathds{1}[v\notin W_n^{\tc}(c_{\beta_n})]C_{\beta_n}(v)\right]\le \P(v\notin W_n^{\tc}(c_{\beta_n}))^{1/2}\E[C_{\beta_n}(v)^2]^{1/2}\le O(\delta_n^2)\gamma_n.
\end{align*}By the linearity of the expectation, $\E[C_{\beta_n}(W_n^\tc(G_n)\setminus W_n^\tc(\c_{\beta_n}))]\le O(\widetilde{\delta}_n^2)\gamma_n|V(G_n)|$, so
\begin{equation*}
    \begin{aligned}
        \P\left(\sum_{v\in W_n^\tc(G_n)\setminus W_n^\tc(\c_{\beta_n})}C_{\beta_n}(v)\ge 
    \frac{1}{4}\widetilde{\delta}_n \gamma_n|V(G_n)|\right)\le \frac{\E[C_{\beta_n}(W_n^\tc(G_n)\setminus W_n^\tc(\c_{\beta_n}))]}{\frac{1}{4}\widetilde{\delta}_n\gamma_n|V(G_n)|}=O(\delta_n).
    \end{aligned}
    \end{equation*}

Now, we are dealing with $\sum_{v\notin W_n^\tc(\c_{\beta_n})}C_{\beta_n}(v)$. Recall that, by Assumption \ref{assum:for_local_lim}, we have that $\sum_{v\notin W_n^\tc(G_n)}\deg(v)\le \widetilde{\delta}_n|V(G_n)|d_n$. Note that, if $\sum_{v\notin W_n^\tc(G_n)}\deg(v)\le \widetilde{\delta}_n^2|V(G_n)|d_n$, then by Markov's inequality,
\begin{align*}
    \P\left(\sum_{v\notin W_n^\tc(G_n)}C_{\beta_n}(v)\ge 
    \widetilde{\delta}_n \mu_nd_n|V(G_n)|\right)\le \frac{\widetilde{\delta}_n^2|V(G_n)|d_n\mu_n}{\widetilde{\delta}_n \mu_nd_n|V(G_n)|}=\widetilde{\delta}_n;
\end{align*}
and if $\widetilde{\delta}_n^2|V(G_n)|d_n\le \sum_{v\notin W_n^\tc(G_n)}\deg(v)\le \widetilde{\delta}_n|V(G_n)|d_n$, then by Lemma \ref{lemma:chebyshev} with $N_n\ge \widetilde{\delta}_n^2|V(G_n)|d_n$ and $a_n=\frac{1}{4}$,
\begin{align*}
    \P\left(\sum_{v\notin W_n^\tc(G_n)}C_{\beta_n}(v)\ge 
    \left(1+\frac{1}{4}\right)\widetilde{\delta}_n \mu_nd_n|V(G_n)|\right)\le O\left(\frac{\max(1,\beta_n)}{\widetilde{\delta}_n^2|V(G_n)|d_n}\right)=O(\widetilde{\delta}_n^{2}).
\end{align*}
Therefore, using that $\delta_n \gamma_n|V(G_n)|=(2\widetilde{K}+1)\widetilde{\delta}_n \frac{\mu_nd_n}{2}|V(G_n)|\ge 3\widetilde{\delta}_n \frac{\mu_nd_n}{2}|V(G_n)|$, we have
\begin{align*}
    \P\left(\sum_{v\notin W_n^\tc(\c_{\beta_n})}C_{\beta_n}(v)\ge 
    \delta_n \gamma_n|V(G_n)|\right)&\le\P\left(\sum_{v\notin W_n^\tc(\c_{\beta_n})}C_{\beta_n}(v)\ge 
    3\widetilde{\delta}_n \frac{\mu_nd_n}{2}|V(G_n)|\right)\\
    &\le \P\left(\sum_{v\notin W_n^\tc(G_n)}C_{\beta_n}(v)\ge 
    \frac{5}{4}\widetilde{\delta}_n\mu_nd_n|V(G_n)|\right)+O(\delta_n)\\
    &=O(\delta_n).
\end{align*}

The third condition of Assumption \ref{assum:for_local_lim} is satisfied if $(2\widetilde{K}+1)\widetilde{\delta}_n\ge \frac{1}{\mu_nd_n/2}$ which holds for large enough $n$ since $\frac{1}{\mu_nd_n}=\Theta(1)\frac{\max(1, \beta_n)}{d_n}=O(\widetilde{\delta}_n^4)$.
\end{proof}

We are proving Lemma \ref{lemma:F_annealed-quenched} by showing a concentration result for $F_{\c_{\beta_n}}(\mathbf{v}_k)$ which is based on the concentration of $C_{\beta_n}(v)$ and $b_{\c_{\beta_n}}(v)$ around their expectations. Clearly, $\E[C_{\beta_n}(v)]=\mu_n\deg(v)$ holds, and we have that $\E[b_{\c_{\beta_n}}(v)]=b_{G_n}(v)$ since $\deg(u)\E\left[\frac{c_{\beta_n}(v,u)}{C_{\beta_n}(u)}\right]=\E\left[\frac{C_{\beta_n}(u)}{C_{\beta_n}(u)}\right]=1$ for any $u\sim v$ by symmetry.

\begin{lemma}\label{lemma:b-cond_concentration}
    We use the notation and the assumptions of Lemma \ref{lemma:F_annealed-quenched}. For any $\eps>0$, we define $W_n^{\textrm{(conc)}}(\eps)$ as: $$\left\{v\in W_n^\tc(\c_{\beta_n})\cap \tW_n(G_n):\; \frac{e^{-b_{\c_{\beta_n}}(v)}}{e^{-b_{G_n}(v)}}\in (1-\eps, 1+\eps)\text{ and } \frac{C_{\beta_n}(v)}{\mu_n\deg(v)}\in (1-\eps, 1+\eps)\right\}.$$ Then, there exists some $\eps_n\rightarrow 0$ such that, with probability $1-O(\delta_n)$, we have $$|W_n^{\textrm{(conc)}}(\eps_n)|\ge (1-14\delta_n^{1/2})|V(G_n)|.$$
\end{lemma}

\begin{proof}
We recall the notation $\gamma_n=\mu_nd_n/2$, $K=2\widetilde{K}+1$, $\delta_n=(2\widetilde{K}+1)\widetilde{\delta}_n$, $W_n^\tc(G_n):=W_n^\tc(G_n, \mathbf{1}, d_n, \widetilde{K})$ and $W_n^\tc(\c_{\beta_n}):=W_n^\tc(G_n, \c_{\beta_n}, \gamma_n, K)$.

We consider the vertex sets
\begin{gather*}
    A_n:=\{u\in W^\tc_n(G_n):\; C_{\beta_n}(u)\in (1-\delta_n, 1+\delta_n)\deg(u)\mu_n\},\\
    D_n:=\left\{v\in \tW_n(G_n):\;\left|\sum_{\substack{u\in W_n^{\tc}(G_n):\\u\sim v}}\frac{c_{\beta_n}(v,u)-\mu_n}{\mu_n\deg(u)}\right|\le \delta_n\right\}.
\end{gather*}
and the events
\begin{align*}
    \mathcal{E}_{n}&:=\left\{(G_n, \c_{\beta_n})\text{ is $(\gamma_n, K, \delta_n)$-balanced, }\frac{|A_n|}{|V(G_n)|}\ge 1-\delta_n\text{ and }\frac{|\tW_n(G_n)\setminus D_n|}{|\tW_n(G_n)|}\le \delta_n\right\},
\end{align*}
$\mathcal{E}_n$ has probability $\P(\mathcal{E}_n)=1-O(\delta_n)$. Indeed, by Lemma \ref{lemma:rdm_cond_balanced}, our random electric network is balanced with probability $1-O(\delta_n)$; while the another properties, with probability $1-O(\delta_n)$, follow from Lemma \ref{lemma:chebyshev}:
\begin{itemize}
    \item[$\circ$] For any $v\in W^\tc_n(G_n)$, by Lemma \ref{lemma:chebyshev} with $a_n=\delta_n$, $p\equiv 1$ and $N_n\ge d_n$, likewise (\ref{formula:prob_by_rmrk}), $\P\left(v\notin A_n\right)\le \P\left(C_{\c_{\beta_n}}(v)\notin (1-\delta_n, 1+\delta_n)\E[C_{\c_{\beta_n}}(v)]\right)=O(\widetilde{\delta}_n^4/\delta_n^{2})=O(\delta_n^{2})$ holds, and therefore, likewise (\ref{formula:Markov_for_sizes}), we have $\P\left(|V(G_n)\setminus A_n|\ge \delta_n |V(G_n)|\right)\le O(\delta_n^2/\delta_n)=O(\delta_n)$.
    \item[$\circ$] For any $v\in \tW_n(G_n)$, we have that $|u\in W_n^{\tc}(G_n):\;u\sim v|=\Theta(d_n)$, so, similarly, by Lemma \ref{lemma:chebyshev} with $a_n=\delta_n/\widetilde{K}$, $p_{(u, v)}=\frac{d_n}{\deg(u)}$ and $N_n=\Theta(d_n)$, we get that $\P\left(v\notin D_n\right)\le O(\delta_n^2)$ and that $\P\left(|\tW_n(G_n)\setminus D_n|\ge \delta_n |\tW_n(G_n)|\right)\le O(\delta_n)$ where the last bound is similar to the last two lines of (\ref{formula:Markov_for_sizes}).
\end{itemize}

Note that, on $\mathcal{E}_n$, we have
\begin{equation}
    \begin{aligned}
            \sum_{v\in V(G_n)}\sum_{\substack{u\notin A_n:\\u\sim v}}\frac{c_{\beta_n}(v,u)}{C_{\beta_n}(u)}=\sum_{u\notin A_n}\frac{1}{C_{\beta_n}(u)}\sum_{\substack{v\in V(G_n):\\u\sim v}}c_{\beta_n}(v,u)=\sum_{u\notin A_n}1=|A_n^c|\le\delta_n|V(G_n)|,
    \end{aligned}
\end{equation}
and, using (\ref{formula:unif_vs_stationary}) and that $C_{\beta_n}(V(G_n))\le (1+\delta_n)|V(G_n)|K\gamma_n\le 2|V(G_n)|K\mu_nd_n/2$ since $(G_n, \c_{\beta_n})$ is balanced,
\begin{equation}
    \begin{aligned}
            \sum_{v\in V(G_n)}\sum_{\substack{u\in W_n^\tc(G_n)\setminus A_n:\\u\sim v}}\frac{c_{\beta_n}(v,u)}{\mu_n\deg(u)}&\le \pi_n(A_n^c)\frac{C_{\beta_n}(V(G_n))}{\mu_nd_n}\le \pi_n(A_n^c)\frac{C_{\beta_n}(V(G_n))}{\mu_nd_n}\\
            &\le \left(2K\lambda_n(A_n^c)+2\delta_n\right)\frac{|V(G_n)|2K\mu_nd_n/2}{\mu_nd_n}\\
            &\le (2K\delta_n+2\delta_n)K|V(G_n)|\le 2K(K+1)\delta_n|V(G_n)|
    \end{aligned}
\end{equation}
hence, writing $$B_n:=\left\{v\in V(G_n)\Given \sum_{\substack{u\notin A_n:\;u\sim v}}\frac{c_{\beta_n}(v,u)}{C_{\beta_n}(u)}+\sum_{\substack{u\in W_n^\tc(G_n)\setminus A_n:\;u\sim v}}\frac{c_{\beta_n}(v,u)}{\mu_n\deg(u)}\le 5K^2\delta_n^{1/2}\right\},$$ the bound $|B_n|\ge (1-\delta_n^{1/2})|V(G_n)|$ holds on $\mathcal{E}_n$.

Writing $\overline{b}_{\c_{\beta_n}}(v):=\sum_{u\in W_n^{\tc}(G_n):\;u\sim v}\frac{c_{\beta_n}(v,u)}{\mu_n\deg(u)}$, for any $v\in W^\tc_n(\c_{\beta_n})\cap B_n$ and if $n$ is so large that $\delta_n\le 1/2$, we have that
\begin{equation*}
    \begin{aligned}
        \left|b_{\c_{\beta_n}}(v)-\overline{b}_{\c_{\beta_n}}(v)\right|&\le \sum_{\substack{u\in A_n:\\u\sim v}}\left|\frac{c_{\beta_n}(v,u)}{C_{\beta_n}(u)}-\frac{c_{\beta_n}(v,u)}{\mu_n\deg(u)}\right|+\sum_{\substack{u\notin A_n:\\u\sim v}}\frac{c_{\beta_n}(v,u)}{C_{\beta_n}(u)}+\sum_{\substack{u\in W_n^\tc(G_n)\setminus A_n:\\u\sim v}}\frac{c_{\beta_n}(v,u)}{\mu_n\deg(u)}\\
        &\le \sum_{u\in A_n:\;u\sim v}\frac{c_{\beta_n}(v,u)}{\mu_n\deg(u)}\left(\frac{1}{1-\delta_n}-\frac{1}{1+\delta_n}\right)+5K^2\delta_n^{1/2}\\
        &\le \frac{8}{3}\delta_n\sum_{u\in A_n:\;u\sim v}\frac{c_{\beta_n}(v,u)}{\mu_n\deg(u)}+5K^2\delta_n^{1/2}\le \frac{8}{3}\delta_nK+5K^2\delta_n^{1/2}
    \end{aligned}
\end{equation*}
where the first two inequalities hold by the definitions of $\overline{b}_{\c_{\beta_n}}$ $A_n$ and $B_n$, the fourth by $v\in W^\tc_n(\c_{\beta_n})$ and $A_n\subseteq W_n^\tc$.

On $D_n$, by (\ref{formula:b(v)_typ}),
\begin{equation*}
    \begin{aligned}
    \left|\overline{b}_{\c_{\beta_n}}(v)-b_{G_n}(v)\right|&\le\left|\overline{b}_{\c_{\beta_n}}(v)-\sum_{\substack{u\in W_n^{\tc}(G_n):\\u\sim v}}\frac{1}{\deg(u)}\right|+\sum_{\substack{u\notin W_n^{\tc}(G_n):\\u\sim v}}\frac{1}{\deg(u)}\\
    &\le\left|\sum_{\substack{u\in W_n^{\tc}(G_n):\\u\sim v}}\frac{c_{\beta_n}(v,u)-\mu_n}{\mu_n\deg(u)}\right|+8\delta_n^{1/2}\le\delta_n+8\delta_n^{1/2}
    \end{aligned}
\end{equation*}

Therefore, for any $v\in W^\tc_n(\c_{\beta_n})\cap A_n\cap B_n\cap D_n$ and $n$ such that $\delta_n\le 1/2$, we have that 
\begin{align*}
    \left|b_{\c_{\beta_n}}(v)-b_{G_n}(v)\right|\le \frac{8}{3}\delta_nK+5K^2\delta_n^{1/2}+\delta_n+8\delta_n^{1/2}=o(1),
\end{align*}
giving us that one can choose $\eps_n\rightarrow0$ satisfying $W_n^{\textrm{(conc)}}(\eps_n)\supseteq W^\tc_n(\c_{\beta_n})\cap A_n\cap B_n\cap D_n$. Using Lemma~\ref{lemma_main:prob_not_nice} to estimate $|V(G_n)\setminus D_n|\le \delta_n|V(G_n)|+ |V(G_n)\setminus \tW_n|\le (\delta_n+10\delta_n^{1/2})|V(G_n)|$, on $\mathcal{E}_n$, 
\begin{align*}
    |W_n^{\textrm{(conc)}}(\eps_n)|&\ge |W^\tc_n(\c_{\beta_n})\cap A_n\cap B_n\cap D_n|\\
    &\ge (1-\delta_n-\delta_n-\delta_n^{1/2}-\delta_n-10\delta_n^{1/2})|V(G_n)|\ge (1-14\delta_n^{1/2})|V(G_n)|
\end{align*}
holds where we have seen that  $\P(\mathcal{E}_n)=1-O(\delta_n)$.
\end{proof}

\begin{proof}[Proof of Lemma \ref{lemma:F_annealed-quenched}] By Lemma \ref{lemma:b-cond_concentration}, we can choose some $\eps_n$ tending to 0 such that the event $\mathcal{F}_n(\eps_n):=\{|W_n^{\textrm{(conc)}}(\eps_n)|\ge (1-14\delta_n^{1/2})|V(G_n)|\}$ has probability at least $1-O(\delta_n)$, where $W_n^{\textrm{(conc)}}(\eps)$ is defined in its statement. Considering $\mathcal{H}_n:=\{C_{\beta_n}(V(G_n))\in [1-\delta_n, 1+\delta_n]2\mu_n|E(G_n)|\}$, by Lemma~\ref{lemma:chebyshev} with $a_n=\delta_n$ and $N_n=|E(G_n)|$, we have that $\P(\mathcal{H}_n^c)=\frac{O(1)}{\delta_n^2|V(G_n)|d_n\mu_n}=O(\delta_n^4/\delta_n^2)=O(\delta_n^2)$ by the conditions of this lemma. Writing
\begin{align*}
    &\chi_{n, 1}:=\sum_{\mathbf{v}_k\in \C\left(W_n^\tc(G_n)\right)}\mu_n^{-k+1}F_{G_n}(T(\mathbf{v}_k))\left(\prod_{j=2}^k c_{\beta_n}(v_{p_T(j)}, v_j)-\mu_n^{k-1}\right),\\
    &\chi_{n, 2}:=
\sum_{\mathbf{v}_k\in \C\left(W_n^\tc(G_n)\right)\setminus\C\left(W_n^{\textrm{(conc)}}(\eps_n)\right)}-\mu_n^{-k+1}F_{G_n}(T(\mathbf{v}_k))\prod_{j=2}^k c_{\beta_n}(v_{p_T(j)}, v_j),\\
    &\chi_{n, 3}:=\sum_{\mathbf{v}_k\in \C\left(W_n^{\textrm{(conc)}}(\eps_n)\right)}\left(F_{\c_{\beta_n}}(T(\mathbf{v}_k))-\mu_n^{-k+1}F_{G_n}(T(\mathbf{v}_k))\right)\prod_{j=2}^k c_{\beta_n}(v_{p_T(j)}, v_j),\\
    &\chi_{n, 4}:=\sum_{\mathbf{v}_k\in \C\left(W_n^\tc(\c_{\beta_n})\right)\setminus \C\left(W_n^{\textrm{(conc)}}(\eps_n))\right)}F_{\c_{\beta_n}}(T(\mathbf{v}_k))\prod_{j=2}^k c_{\beta_n}(v_{p_T(j)}, v_j).
\end{align*}
we have the decomposition \begin{equation}
    \sum_{\mathbf{v}_k\in \C\left(W_n^\tc(\c_{\beta_n})\right)}F_{\c_{\beta_n}}(T(\mathbf{v}_k))\prod_{j=2}^k c_{\beta_n}(v_{p_T(j)}, v_j)=\sum_{i=1}^4\chi_{n, i}+\sum_{\mathbf{v}_k\in \C\left(W^\tc_n(G_n)\right)}F_{G_n}(T(\mathbf{v}_k)).
\end{equation}

We start by estimating $\chi_{n, 1}$. Note that $\E[\chi_{n, 1}]=0$ 
since any summand has 0 expected value as we have $\E\left[\prod_{j=2}^k \frac{c_{\beta_n}(v_{p_T(j)}, v_j)}{\mu_n}\right]=1$ by independence. 

Now, we show the concentration of $\chi_{n, 1}$. Note that $\prod_{j=2}^k c_{\beta_n}(v_{p_T(j)}, v_j)$ and $\prod_{j=2}^k c_{\beta_n}(v'_{p_T(j)}, v'_j)$ are independent for $E(T(\mathbf{v_k}))\cap E(T(\mathbf{v'_k}))=\emptyset$, and have a positive covariance otherwise. Indeed, we abbreviate as $i(\mathbf{v_k}, \mathbf{v_k'}):=|E(T(\mathbf{v_k}))\cap E(T(\mathbf{v'_k}))|$, $d(\mathbf{v_k}, \mathbf{v_k'}):=|E(T(\mathbf{v_k}))\triangle E(T(\mathbf{v'_k}))|$ and $X:=\exp(-U\beta_n)$ for $U\sim \unif[0, 1]$, clearly,
\begin{align*}
    &\cov\left(\prod_{j=2}^k c_{\beta_n}(v_{p_T(j)}, v_j), \prod_{j=2}^k c_{\beta_n}(v'_{p_T(j)}, v'_j)\right)=\E\left[X\right]^{d(\mathbf{v_k}, \mathbf{v_k'})}\left(\E\left[X^2\right]^{i(\mathbf{v_k}, \mathbf{v_k'})}-\E\left[X\right]^{2i(\mathbf{v_k}, \mathbf{v_k'})}\right)
\end{align*}
which is positive if $i(\mathbf{v_k}, \mathbf{v_k'})>0$. It is clear from this form that the covariance has order $O\left(\mu_n^{2(k-1)-i(\mathbf{v_k}, \mathbf{v_k'})}\right)$, and the number of pairs $(\mathbf{v_k}, \mathbf{v_k'})$ with $i(\mathbf{v_k}, \mathbf{v_k'})=i$ is $O(|V(G_n)|^2d_n^{2(k-1)-i})$ ways. Since there exists a constant $M$ such that $F_{G_n}(T(\mathbf{v}_k))\le M|V(G_n)|^{-1}d_n^{-(k-1)}$ for any $\mathbf{v_k}\in \C\left(W_n^\tc(G_n)\right)$, the nonnegative covariances give us that:
\begin{align*}
    \var\left(\chi_{n, 1}\right)&=\var\left(\sum_{\mathbf{v}_k\in \C\left(W_n^\tc(G_n)\right)}\mu_n^{-k+1}F_{G_n}(T(\mathbf{v}_k))\prod_{j=2}^k c_{\beta_n}(v_{p_T(j)}, v_j)\right)\\
     &\le O(1)\var\left(\sum_{\mathbf{v}_k\in \C\left(W_n^\tc(G_n)\right)}\frac{1}{|V(G_n)|(\mu_nd_n)^{k-1}}\prod_{j=2}^k c_{\beta_n}(v_{p(j)}, v_j)\right)\\
     &\le O(1)\frac{1}{|V(G_n)|^2(\mu_nd_n)^{2(k-1)}}\sum_{i=1}^\infty |V(G_n)|^2d_n^{2(k-1)-i}\mu_n^{2(k-1)-i}\\
     &=\sum_{i=1}^\infty(\mu_nd_n)^{-i}=O((\mu_nd_n)^{-1})=O(\delta_n^4).
\end{align*}
Then, Chebyshev's and Jensen's inequalities imply that \begin{equation}\label{formula:chi1}
    \begin{aligned}
        \P(|\chi_{n, 1}|&\ge \delta_n)\le O(
        \delta_n^4)/\delta_n^2= O(\delta_n^2)\text{ and }
        \E\left[|\chi_{n, 1}|\right]\le\E\left[|\chi_{n, 1}|^2\right]^{1/2}\le O(\delta_n^2).
    \end{aligned}
\end{equation}

In order to deal with $\chi_{n, 2}$, we introduce the sum
$$M_{n, i}(v):=\sum_{\mathbf{v}_k\in \C\left(W_n^\tc(G_n)\right):\; v_i=v}\mu_n^{-k+1}F_{G_n}(T(\mathbf{v}_k))\prod_{j=2}^k c_{\beta_n}(v_{p_T(j)}, v_j).$$
Then, by Cauchy-Schwarz inequality,
\begin{align*}
\E[|\chi_{n, 2}|]&\le \sum_{i=1}^k\E\left[\sum_{v\in W_n^\tc(G_n)\setminus W_n^{\textrm{(conc)}}(\eps_n)}M_{n, i}(v)\right]\\
&\le \sum_{i=1}^k\E\left[|W_n^\tc(G_n)\setminus W_n^{\textrm{(conc)}}(\eps_n)|\right]^{1/2}\E\left[\sum_{v\in W_n^\tc(G_n)}M_{n, i}(v)^2\right]^{1/2},
\end{align*}
where $\E[|W_n^\tc(G_n)\setminus W_n^{\textrm{(conc)}}(\eps_n)|]\le O(\delta_n^{1/2})|V(G_n)|$; moreover, similarly to the computation of $\var(\chi_{n, 1})$, 
\begin{align*}
\E\left[\sum_{v\in W_n^\tc(G_n)}M_{n, i}(v)^2\right]&\le O(1)\frac{1}{|V(G_n)|^2(\mu_nd_n)^{2(k-1)}}\sum_{j=0}^\infty |V(G_n)|d_n^{2(k-1)-j}\mu_n^{2(k-1)-j}\\
     &=\frac{1}{|V(G_n)|}\sum_{j=0}^\infty(\mu_nd_n)^{-j}=O(|V(G_n)|^{-1}).
\end{align*}
So, we have obtained
\begin{equation}\label{formula:chi2}
    \begin{aligned}
        \E[|\chi_{n, 2}|]\le O\left(\delta_n^{1/4}|V(G_n)|^{1/2}O(|V(G_n)|^{-1/2})\right)=O\left(\delta_n^{1/4}\right)\text{ and }
        \P[|\chi_{n, 2}|\ge \delta_n^{1/8}]=O(\delta_n^{1/8}),
    \end{aligned}
\end{equation}
where we used Markov's inequality.

Now, we estimate $|\chi_{n, 3}|$. For $\mathbf{v}_k\in \C(W_n^{\textrm{(conc)}}(\eps_n))$ and any $j\le k$, $C_{\beta_n}(v_j)$ and $b_{\c_{\beta_n}}(v_j)$ are close to $\E[C_{\beta_n}(v_j)]=\mu_n\deg(v_j)$ and $\E[b_{\c_{\beta_n}}(v_j)]=b_{G_n}(v_j)$, hence $F_{\c_{\beta_n}}(T(\mathbf{v}_k))$ is close to
$\frac{1}{|V(G_n)|}\exp\left(-\sum_{j=1}^{t}b_{G_n}(v_j)\right)\frac{\sum_{j=t+1}^{k}\mu_n\deg(v_j)}{\prod_{j=1}^{k}\mu_n\deg(v_j)}=\frac{F_{G_n}(T(\mathbf{v}_k))}{\mu_n^{k-1}}$. More precisely, for small enough $\eps_n$ and $\mathbf{v}_k\in \C(W_n^{\textrm{(conc)}}(\eps_n))$,
\begin{align*}
    \left|F_{\c_{\beta_n}}(T(\mathbf{v}_k))-\frac{F_{G_n}(T(\mathbf{v}_k))}{\mu_n^{k-1}}\right|&\le \left((1+\eps_n)^t\frac{(1+\eps_n)}{(1-\eps_n)^k}-(1-\eps_n)^t\frac{(1-\eps_n)}{(1+\eps_n)^k}\right)F_{G_n}(T(\mathbf{v}_k))\mu_n^{-k+1}\\
    &=\left(\frac{(1+\eps_n)^{t+1+k}-(1-\eps_n)^{t+1+k}}{(1-\eps_n^2)^k}\right)F_{G_n}(T(\mathbf{v}_k))\mu_n^{-k+1}\\ 
    &\le \eps_n 2(t+k+1)\left(\frac{3}{2}\right)^{k+t}\left(\frac{4}{3}\right)^k 2F_{\c_{\beta_n}}(T(\mathbf{v}_k))\\
    &=O(\eps_n) F_{\c_{\beta_n}}(T(\mathbf{v}_k))
\end{align*}
where we used for $0\le x\le 1/2$, that $\frac{\partial}{\partial x}\left((1+x)^{t+k+1}-(1-x)^{t+k+1}\right)\le 2(t+k+1)(1+x)^{t+k}=2(t+k+1)(3/2)^{k+t}$ and $(1-x^2)^{-k}\le(3/4)^{-k}$, and that $ F_{G_n}(T(\mathbf{v}_k))\mu_n^{-k+1}\le 2F_{\c_{\beta_n}}(T(\mathbf{v}_k))$ for small enough $\eps_n$ and $\mathbf{v}_k\in \C(W_n^{\textrm{(conc)}}(\eps_n))$. Therefore, by Lemma \ref{lemma_main:F_on_W^tc} with $U_n=\emptyset$ and $f_n\equiv1$,
\begin{equation}\label{formula:chi3}
    \begin{aligned}
        |\chi_{n, 3}|&\le O(\eps_n)\sum_{\mathbf{v}_k\in \C(W_n^\tc(\c_{\beta_n}))}F_{\c_{\beta_n}}(T(\mathbf{v}_k))\prod_{j=2}^k c_{\beta_n}(v_{p_T(j)}, v_j)\le
        \eps_n O(1).
    \end{aligned}
\end{equation}

We deal with $\mathds{1}[\mathcal{F}_n(\eps_n)]\chi_{n, 4}$ by Lemma \ref{lemma_main:F_on_W^tc} with $U_n=W_n^{\textrm{(conc)}}(\eps_n)$ and $f_n={O(\delta_n^{1/2})}$ on $\mathcal{F}_n(\eps_n)$:
\begin{equation}\label{formula:chi4_typical}
    \begin{aligned}
        \mathds{1}[\mathcal{F}_n(\eps_n)]\chi_{n, 4}\le O(\delta_n^{1/2}).
    \end{aligned}
\end{equation}
To estimate the expected value, recall that $\P((\mathcal{F}_n(\eps_n)\cap \mathcal{H}_n)^c)=O(\delta_n)$, so
\begin{equation*}
    \begin{aligned}
            \E[\mathds{1}[(\mathcal{F}_n(\eps_n)\cap \mathcal{H}_n)^c]C_{\beta_n}(V(G_n))]&=\E[C_{\beta_n}(V(G_n))]-\E[\mathds{1}[\mathcal{F}_n(\eps_n)\cap \mathcal{H}_n]C_{\beta_n}(V(G_n))]\\
    &\le 2\mu_n|E(G_n)|-(1-O(\delta_n))(1-\delta_n)2\mu_n|E(G_n)|\\
    &=O(\delta_n)\mu_n|E(G_n)|.
    \end{aligned}
\end{equation*}
Moreover, using that $F(T(\mathbf{v}_k))\prod_{j=2}^k c_{\beta_n}(v_{p_T(j)}, v_j)=O(1)\P\left(\mathbf{X}_k=\mathbf{v}_k\right)\frac{C_{\beta_n}(V(G_n))}{\gamma_n|V(G_n)|}$ for any $\mathbf{v}_k\in(W_n^\tc(\c_{\beta_n}))$, where $\mathbf{X}_k$ is defined in Definition \ref{def:X_k}, we have
\begin{equation}\label{formula:chi4_expected}
    \begin{aligned}
        &\E[\mathds{1}[\mathcal{F}_n(\eps_n)^c]\chi_{n, 4}]\\
        &\hspace{1.5 cm}\le O(1)\E\left[\frac{\mathds{1}[\mathcal{F}_n(\eps_n)^c]C_{\beta_n}(V(G_n))}{\mu_nd_n|V(G_n)|}\P\left(\mathbf{X}_k\in \C(W_n^\tc(\c_{\beta_n}))\setminus \C(W_n^{\textrm{(conc)}}(\eps_n)))\right)\right]\\
        &\hspace{1.5 cm}\le O(1)\E\left[\frac{\mathds{1}[(\mathcal{F}_n(\eps_n)\cap \mathcal{H}_n)^c]C_{\beta_n}(V(G_n))}{\mu_nd_n|V(G_n)|}\right]\\
        &\hspace{1.5cm}\le O(\delta_n).
    \end{aligned}
\end{equation}

So, by (\ref{formula:chi1}, \ref{formula:chi2}, \ref{formula:chi3}, \ref{formula:chi4_typical}, \ref{formula:chi4_expected}), we get $\E[|\chi_{n, 1}+\chi_{n, 2}+\chi_{n, 3}+\chi_{n, 4}|]\le O(\delta_n^{1/4}+\eps_n)=o(1)$---which we needed for (\ref{formula:F_annealed})---and $\P(|\chi_{n, 1}+\chi_{n, 2}+\chi_{n, 3}+\chi_{n, 4}|\ge C(\delta_n^{1/8}+\eps_n))\le O(\delta_n^{1/8})=o(1)$ for a large enough $C$---which we needed for (\ref{formula:F_quenched}).
\end{proof}

\section{Fast-growing \texorpdfstring{$\beta_n$}{}'s}\label{sec:mst-like}
\subsection{Expected edge overlaps}\label{subsec:mstlike_edge_overlap}
This subsection contains the proof of Theorem \ref{thm:edge_overlap_mst}.

For a tree $T$ and $u,v\in V(T)$, we write $\path_{T}(u, v)$ to denote the unique path between $u$ and $v$ in $T$.

\begin{defi}
    We call $\maxst(\c)$ the spanning tree $T$ maximizing $\prod_{e\in E(T)}c(e)$.
    
    We call an edge $e=(u,v)$ \emph{$\eps$-significant external edge} if $e\notin \maxst(c)$ and there exists an $f\in \path_{\maxst(c)}(u, v)$ satisfying $\frac{c(e)}{c(f)}
    \ge \eps$. The \emph{set of $\eps$-significant external edges} is denoted by $E_{\mathrm{sig}}(\eps):=\{e\notin \maxst(c):\; e\text{ is an $\eps$-significant external edge}\}$.
\end{defi}

\begin{prop}[Proposition 2.16 b) of \cite{kusz2024diameter}]\label{prop:wst_vs_mst}
Consider a connected electric network $(G, \c)$ on a simple graph $G$ with a unique $\maxst(\c)$ and $\eps_0\in (0, 1)$.
    Then
    $$\E[|E(\wst(\c))\setminus E(\maxst(\c))|]\le \left|E_{\mathrm{sig}}\left(\eps_0\right)\right|+\eps_0^{1/2}|V(G)|^{11/2}.$$
\end{prop}
We give one tiny technical remark: although Proposition \ref{prop:wst_vs_mst} has been stated in \cite{kusz2024diameter} for the complete graph (assuming that all the conductances are different), its proof uses only the existence of a unique $\maxst(\c)$ therefore it holds for any electric network $(G, \c)$ with a unique $\maxst(\c)$.

\begin{proof}[Proof of Theorem \ref{thm:edge_overlap_mst}]Note that, any bridge edge is automatically in any spanning tree. For a non-bridge edge $e=(u, v)$, we define $h_e:=\min_{P\text{ is a path between $u$ and $v$: }e\notin P}\max_{f\in P}U_f$. Then, clearly, $e\in \mst(G_n)$ is equivalent to $U_e<h_e$, and for $e\notin \mst(G_n)$, the edge $e\notin \mst(G_n)$ is $|V(G_n)|^{-\kappa}$-significant is equivalent to $\frac{\exp(-\beta_nU_e)}{\exp(-\beta_nh_e)}>|V(G_n)|^{-\kappa}$ which can be written as $U_e<h_e+\frac{\kappa \log |V(G_n)|}{\beta_n}$. Since $h_e$ depends only on $\{U_f\}_{f\ne e}$, which is independent of $U_e$, so one has that $$\P(e\text{ is $|V(G_n)|^{-\kappa}$-significant}\given h_e)=\P\left(h_e<U_e<h_e+\frac{\kappa \log |V(G_n)|}{\beta_n}\right)\le \frac{\kappa \log |V(G_n)|}{\beta_n}.$$ 

Given $\{U_e\}_{e}$, we choose two independent  $\mathcal{T}_n, \mathcal{T}'_n\sim \wst^{\beta_n}(G_n)$. For $\beta_n\gg \frac{|E(G_n)|\log |V(G_n)|}{|V(G_n)|}$, Proposition \ref{prop:wst_vs_mst} with $\kappa>11$ gives us 
\begin{equation}\label{formula:expected_diff}
    \begin{aligned}
    \E[|E(\mathcal{T}_n)\triangle E(\mathcal{T}'_n)|]&\le 2\E[|E(\mathcal{T}_n)\setminus E(\mathcal{T}'_n)|]\le 
    4\E\left[\left|E(\wst^{\beta_n}(G_n))\setminus E(\mst(G_n))\right|\right]\\
    &\le 4\E\left[ \left|E_{\mathrm{sig}}(|V(G_n)|^{-\kappa})\right|\right]+o(1)\le 4\frac{\kappa |E(G_n)|\log |V(G_n)|}{\beta_n}+o(1)=o(|V(G_n)|).
\end{aligned}
\end{equation}
Combining $\E[|E(\mathcal{T}_n)\triangle E(\mathcal{T}'_n)|]=o(|V(G_n)|)$ and $|E(\mathcal{T}_n)\cap E(\mathcal{T}'_n)|=|V(G_n)|-1-\frac{1}{2}|E(\mathcal{T}_n)\triangle E(\mathcal{T}'_n)|$, we get $\mathcal{O}(G_n, \beta_n)\ge (1-o(1))|V(G_n)|$.
\end{proof}
\begin{remark*}
    Following from a similar computation, we have $\P(\wst^{\beta_n}(G_n)=\mst(G_n))\rightarrow 1$ for $\beta_n\gg |E(G_n)|\log |V(G_n)|$.
\end{remark*}

\subsection{Local limit}
In this subsection, we prove Theorem \ref{thm:local_mst} by applying Theorem \ref{thm:edge_overlap_mst}.
\label{subsec:mstlike_local_limit}
\begin{prop}\label{prop:big_edge_overlap_local_limit}
Consider two sequences of (possibly random) graphs $G_n$ and $H_n$ with deterministic sets of vertices $V(G_n)=V(H_n)$, some $M_n$ and some $t_n$. Then, for $X\sim\unif(V(G_n))$, we have that
\begin{gather}\label{formula:loc_stat_diff}
    \P\left(B_{G_n}(X, r)\ne B_{H_n}(X, r)\right)\le \frac{2t_nM_n}{|V(G_n)|}+\P(|V(B_{G_n}(X, 2r))|>M_n)+\P(|E(G_n)\triangle E(H_n)|>t_n).
\end{gather}
\end{prop}
\begin{corollary}
Consider some sequences of (possibly random) graphs $G_n$ and $H_n$ such that $G_n$ has a locally finite limit local weak limit. Suppose that there exist some  $t_n=o(|V (G_n)|)$ such that $\P(|E(G_n)\triangle E(H_n)| \le t_n)\rightarrow 1$. Then $H_n$ has the same local limit as $G_n$.
\end{corollary}
\begin{proof}[Proof of the corollary]
    If $G_n$ has a local limit, then for any $\eps>0$, there exists some $M>0$ such that $\sup_n \P(|V (B_{G_n}(o, 2r))| > M) \le \eps$ holds. Hence, (\ref{formula:loc_stat_diff}) with $M_n\equiv M$ gives us that  $\limsup_{n \rightarrow \infty}\P\left(B_{G_n}(X, r)\ne B_{H_n}(X, r)\right)\le \eps$, and letting $\eps\rightarrow 0$ the asymptotic agreement of the local limits.
\end{proof}
\begin{proof}[Proof of Proposition \ref{prop:big_edge_overlap_local_limit}] We denote by $S_n$ the set of vertices formed by the endpoints of the edges from $E(G_n)\triangle E(H_n)$. Conditioning on $(G_n, H_n)$, by the double counting of the set $A_n:=\{ (o,s) : o \in V(G_n), s \in S_n\cap B_{G_n}(o ,r), |B_{G_n}(s, r)| \le M_n \}$, we have
\[|V(G_n)|\P(B_{G_n}(X, r) \cap S_n\ne \emptyset, |V (B_{G_n}(X, 2r))| \le M_n\given G_n, H_n)\le |A_n| \le |S_n|M_n.\]
Therefore, 
\begin{equation*}
    \begin{aligned}
        \P(B_{G_n}(X, r) \cap S_n\ne \emptyset, \; |S_n|\le 2t_n,\;|V(B_{G_n}(X, 2r))|\le M_n)\le \frac{2t_nM_n}{|V(G_n)|},
    \end{aligned}
\end{equation*}
resulting in (\ref{formula:loc_stat_diff}).
\end{proof}

\begin{proof}[Proof of Theorem \ref{thm:local_mst}] The asymptotic agreement of the local limits is a direct consequence of (\ref{formula:expected_diff}) and Proposition \ref{prop:big_edge_overlap_local_limit}. However, we would like to show the quenched agreement (\ref{formula:local_wst_vs_mst}), hence we are a little bit more careful. 

We write $\Delta_n:=|E(\wst^{\beta_n}(G_n))\triangle E(\mst(G_n))|$. We can assume that $\E[\Delta_n]>0$, otherwise $\Delta_n$ is almost surely 0 and the statement automatically follows. We abbreviate as $\P_{\mathbf{U}}(\cdot):=\P(\cdot\given \{U_e\}_{e\in E(G_n)})$.

For $\beta_n\gg |E(G_n)|/|V(G_n)|\log |V(G_n)|$, by formula (\ref{formula:expected_diff}) and applying Markov's inequality twice, for example, for $t_n:=\E[\Delta_n]^{1/4}|V(G_n)|^{3/4}=o(|V(G_n)|)$ and $a_n:=t_n/|V(G_n)|=o(1)$, we have 
\begin{equation}\label{formula:for_local_lim}
    \begin{aligned}
\P\left(\P_{\mathbf{U}}(\Delta_n\ge t_n)\ge a_n\right)&\le\frac{\E[\Delta_n]}{a_nt_n}=\frac{\E[\Delta_n]^{1/2}}{|V(G_n)|^{1/2}}=o(1).
    \end{aligned}
\end{equation}
We consider some $1\ll M_n\ll |V(G_n)|/t_n$. Since the existence of a locally finite local limit of $\mst(G_n)$ is assumed in our statement, $q_n:=\P(|V (B_{\mst(G_n)}(o, 2r))|>M_n)\rightarrow 0$ holds. Then, by Markov's inequality, we have $\P(\P_{\mathbf{U}}(|V (B_{\mst(G_n)}(o, 2r))| >M_n)\ge \sqrt{q_n})\le\sqrt{q_n}$ if $q_n>0$; and $|V (B_{\mst(G_n)}(o, 2r))|\le M_n$ almost surely if $q_n=0$.

We denote by $\mathcal{A}_n$ the $\mathbf{U}$-measurable event that $\P_{\mathbf{U}}(|V (B_{\mst(G_n)}(o, 2r))| >M_n)\le \sqrt{q_n}$ and $\P_{\mathbf{U}}(\Delta_n\ge t_n)<a_n$ hold. Then, we have $\P(\mathcal{A}_n^c)=o(1)$. By plugging in (\ref{formula:loc_stat_diff}), on $\mathcal{A}_n$, we obtain $\P_{\mathbf{U}}(B_{\mst(G_n)}(X, r)\ne B_{\wst^{\beta_n}(G_n)}(X, r))\rightarrow 0$.
\end{proof}

\subsection{Expected total length for the complete graph}
\label{subsec:length}
We prove Theorem \ref{thm:length} in this subsection.
\begin{notation*}
    In this subsection, we abbreviate as $\wst^{\beta_n}_n:=\wst^{\beta_n}(K_n)$.
\end{notation*}

Behind Theorem \ref{thm:length} a) and b), we have the following heuristics: Typically $\max_{e\in \mst_n}U_e\le \frac{2\log n}{n}$, therefore for $\beta_n\gg n\log^2 n$, any $n^{-\kappa}$-significant edge has label $\le \frac{2\log n}{n}+\frac{\kappa\log n}{\beta_n}=\Theta\left(\frac{\log n}{n}\right)$ and
\begin{align*}
\E\left[L(\wst^{\beta_n}_n)\right]\stackrel{(\ref{formula:expected_diff})}{\le} \E\left[L(\mst^{\beta_n}_n)\right]+\Theta\left(\frac{\log n}{n}\right)\Theta\left(\frac{n^2\log n}{\beta_n}\right)=\zeta(3)+o(1).
\end{align*}
Below, we give a precise reasoning which works even for $\beta_n\gg n\log n$.

\begin{lemma}[Lemma 3.10 of \cite{kusz2024diameter}]\label{lemma:E-R_comps_A-B} 
We write $\P_{\mathbf{u}}(\cdot)=\P(\;\cdot\given U_e=u_e\; \forall e\in E(K_n))$, and $\mathcal{C}_v^E$ for the connected component w.r.t.~$E$ containing $v$. For any spanning tree $T$ of $K_n$, we denote by $\langle V\rangle_T$ the smallest connected graph w.r.t.~the edges $E(T)$ containing the vertices $V$.

Then there exists $N_0$ such that for any $\mathbf{u}\in [0,1]^{\binom{n}{2}}$ and $n\ge N_0$
  \begin{gather*}
  {\P}_{\mathbf{u}}\left(\left\langle V\left(\mathcal{C}_v^{E(G_{n, p})}\right)\right\rangle_{\wst_{n}^{\beta_n}}\subseteq G_{n, p+\frac{50\log n}{\beta_n}} \ \forall v\in V(K_n) \;\forall p\in[0, 1]\right)\ge 1-n^{-2}.
	\end{gather*}
\end{lemma}

\begin{proof}[Proof Theorem \ref{thm:length} a) and b)]
We show that $\E\left[L(\wst_n^{\beta_n})\right]-\E\left[L(\mst_n)\right]=O\left(\frac{n\log n}{\beta_n}\right)+O(n^{-1})$.

We denote by $k(\wst^{\beta_n}_n,p)$ and $k(\mst_n,p)$ the number of the connected components of $\wst^{\beta_n}_n \cap G_{n, p}$ and $\mst_n \cap G_{n, p}$, respectively. We also use the notation $N(\wst^{\beta_n}_n, p):=|E(\wst^{\beta_n}_n \cap G_{n, p})|$ which can be written also as $N(\wst^{\beta_n}_n, p)=\sum_{e\in \wst^{\beta_n}_n}\mathds{1}[U_e\le p]$.

The integral (8) of \cite{steele2002minimal} expressing $L\left(\mst_n\right)$ can be generalized to the weighted spanning trees as  $L\left(\wst_n^{\beta_n}\right)=\int_{0}^1 k(\wst^{\beta_n}_n, t) \;d t-1$. Indeed,
\begin{align*}
L\left(\wst_n^{\beta_n}\right)&=\sum_{e\in \wst^{\beta_n}_n}U_e=\sum_{e\in \wst^{\beta_n}_n}\int_0^1 \mathds{1}\left[U_e>t\right] \;dt\\
&=\int_0^1 n-1-N(\wst^{\beta_n}_n, t) \;dt=\int_0^1 k(\wst^{\beta_n}_n, t)-1 \;dt.
\end{align*}

Lemma \ref{lemma:E-R_comps_A-B} with $\delta_n=\frac{50\log n}{\beta_n}$ implies that for large $n$ and any $\{u_e\}_{e\in E(K_n)}$,
$$\P\left(k(\wst^{\beta_n}_n, p+\delta_n)\le k\left(\mst_n, p\right) \ \forall p\in [0, 1]\given U_e=u_e \;\forall e\in E(K_n)\right)\ge 1-n^{-2},$$
so our statement for $\beta_n\ge \log n$ is obtained by
\begin{align*}
\E\left[L(\wst_n^{\beta_n})-L(\mst_n)\right]&\le \E\left[\int_{0}^1 k_{\mst_n}\left(t-\delta_n\right)-k_{\mst_n}\left(t\right) \;dt\right]+n\cdot n^{-2}\\
&\le\E\left[\int_{0}^{\delta_n}n \;dp\right]+n^{-1}=O\left(\frac{n\log n}{\beta_n}\right)+n^{-1}
\end{align*}
For $\beta_n\le \log n$, the statement follows from the trivial estimate $\E[L(\wst_n^{\beta_n})]\le n-1$.
\end{proof}

\begin{proof}[Proof of Theorem \ref{thm:length} c)] If $\beta_n=0$, then by independence of the tree and the labels, we have that $\E\left[L\left(\wst_n^{\beta_n}\right)\right]=(n-1)/2$. So, in this proof we will assume that $\beta_n>0$.

We consider
\begin{align*}
A_{n}:=\left\{\mathcal{R}_{\mathrm{eff}}(v\leftrightarrow w)\in \left(\frac{(1-\delta_n)2\beta_n}{(1-e^{-\beta_n})n}, \frac{(1+\delta_n)2\beta_n}{(1-e^{-\beta_n})n}\right) \; \forall v\ne w\in V(K_n)\right\},
\end{align*}
where we can choose some $\delta_n\rightarrow 0$ to have $\P(A_n)\ge 1-o(n^{-2})$ by Theorems 3.1 and 3.5 of \cite{makowiec2024local}.

By Kirchhoff's formula, we have
\begin{align*}
\E\left[\mathds{1}[A_n]L\left(\wst_n^{\beta_n}\right)\right]&=\sum_{(v,w)\in E(K_n)} \E\left[\mathds{1}[A_n]U_{(v,w)}c(v, w)\mathcal{R}_{\mathrm{eff}}(v\leftrightarrow w)\right]\\
&=\sum_{f\in E(K_n)} \E\left[\mathds{1}[A_n]U_{f}e^{-\beta_nU_f}\frac{(1+o(1))2\beta_n}{(1-e^{-\beta_n})n}\right]\\
&=\binom{n}{2}\E\left[U_fe^{-\beta_nU_f}\right]\frac{(1+o(1))2\beta_n}{(1-e^{-\beta_n})n}-O(\P(A_n^c)n^2)\\
&=(1+o(1))\frac{n^2}{2}\frac{1-\beta_ne^{-\beta_n}-e^{-\beta_n}}{\beta_n^2}\frac{2\beta_n}{(1-e^{-\beta_n})n},\end{align*}
where, we used $U_fe^{-\beta_nU_f}\le 1$ and $\beta_n=O(n)$ in the third line, and $\P(A_n^c)=o(n^{-2})$ and $\frac{(1-\beta_ne^{-\beta_n}-e^{-\beta_n})n}{(1-e^{-\beta_n})\beta_n}=\Theta(1)\frac{n}{\max(1, \beta_n)}\gg 1\gg \P(A_n^c)n^2$ in the fourth line. Hence, we have
$$\E\left[L\left(\wst_n^{\beta_n}\right)\right]=\P(A_n^c)O(n)+\E\left[\mathds{1}[A_n]L\left(\wst_n^{\beta_n}\right)\right]=(1+o(1))\frac{(1-\beta_ne^{-\beta_n}-e^{-\beta_n})n}{(1-e^{-\beta_n})\beta_n},$$
since $|L\left(\wst_n^{\beta_n}\right)|\le n-1$ and $\P(A_n^c)=o(n^{-2})$.
\end{proof}

\section*{Statement on AI use}\addcontentsline{toc}{section}{Statement on AI use}

The first version of this preprint available at arXiv was prepared without AI assistance. The second version contains numerous minor corrections throughout the manuscript. During the revision process, the author used large language models to assist with proofreading and consistency checks. The author takes full responsibility for the final content of the paper.

\bibliographystyle{alpha}

\end{document}